\documentclass[11pt,a4paper]{article}
\usepackage{amssymb}
\usepackage{amsmath}
\usepackage{amsthm}
\usepackage[dvips]{epsfig}
\usepackage{graphicx}
\usepackage{color}
\usepackage{url}
\usepackage{tikz}
\usepackage{caption}
\usepackage{subcaption}
\usepackage{epstopdf}
\usepackage[arrow, matrix, curve]{xy}
\usepackage{framed}
\usepackage{dsfont}
\usepackage{bbold}
\usetikzlibrary{decorations.markings}
\usetikzlibrary{decorations.pathmorphing}
\usetikzlibrary{datavisualization}

\makeatletter\@addtoreset {equation}{section}\makeatother

\usetikzlibrary{matrix,arrows}
\newtheorem{thm}{Theorem}[section]

\newtheorem{prop}[thm]{Proposition}
\newtheorem{lem}[thm]{Lemma}
\theoremstyle{definition}
\newtheorem{mydef}[thm]{Definition}

\newtheorem{rhp}[thm]{Riemann-Hilbert problem}
\newtheorem{dbrhp}[thm]{$\overline{\partial}$-Riemann-Hilbert problem}
\newtheorem{delbar}[thm]{$\overline{\partial}$-Problem}

\newtheorem{remark}[thm]{Remark}

{\begin{trivlist} \item[]{\em Proof }}%
{\hspace*{\fill}$\rule{.3\baselineskip}{.35\baselineskip} $\end{trivlist}}

\DeclareMathOperator{\diag}{diag}
\DeclareMathOperator{\sech}{sech}

\DeclareMathOperator{\Imag}{Im}
\DeclareMathOperator{\Real}{Re}

\DeclareMathOperator{\C}{\mathbb{C}}

\DeclareMathOperator{\R}{\mathbb{R}}
\DeclareMathOperator{\N}{\mathbb{N}}
\DeclareMathOperator{\eps}{\varepsilon}
\DeclareMathOperator{\db}{\overline{\partial}\!}
\DeclareMathOperator{\dist}{dist}
\def\res{\mathop{\mathrm{Res}}}
\begin{document}

\title{Long-time asymptotics for the massive Thirring model}

\author{Aaron Saalmann\footnote{Mathematisches Institut, Universit\"{a}t zu K\"{o}ln, 50931 K\"{o}ln, Germany, e-mail: asaalman@math.uni-koeln.de}}
\date{\today}
\maketitle
\begin{abstract}
    We consider the massive Thirring model and establish pointwise long-time behavior of its solutions in weighted Sobolev spaces. For soliton-free initial data we can show that the solution converges to a linear solution modulo a phase correction caused by the cubic nonlinearity. For initial data that support finitely many solitons we obtain long-time behavior in the form of a multi-soliton which in turn splits into a sum of localized solitons. This phenomenon is known as soliton resolution. The methods we will is the nonlinear steepest descent of Deift and Zhou and the paper also relies on recent progress in the inverse scattering transform for the massive Thirring model.
\end{abstract}
\section{Introduction}\label{s intro}
Our goal is to explore long-time asymptotics for the massive Thirring model (MTM) which can be written in laboratory coordinates as
\begin{equation}\label{e mtm}
  \left\{
    \begin{array}{ll}
      i(u_t+u_x)+v+u|v|^2=0,\\
      i(v_t-v_x)+u+|u|^2v=0,
    \end{array}
  \right.
\end{equation}
for the two unknowns $u,v:\R\times\R\to\C$ with $u(0,x)=u_0(x)$ and $v(0,x)=v_0(x)$ for given initial data $(u_0,v_0)$. For $u_0,v_0\in H^s(\R)$ with $s\geq 0$, the problem is globally well-posed, see \cite{Candy2011}. The model (\ref{e mtm}) was introduced by Thirring in \cite{Thirring1958} in the context of general relativity and describes the self interaction of a Dirac field in the space of one dimension. The MTM system is integrable in the sense that it admits a Lax pair. This Lax pair was already presented and studied in \cite{Mikhailov76,Kaup1977,Kuznetsov1977}. It allows one to solve the equation with the inverse scattering transform. This procedure may be described as follows:
\begin{description}
  \item[Step 1:]
  For generic initial data $(u_0,v_0)$ compute the scattering data $\mathcal{S}(u_0,v_0)=(p, \{\lambda_j,C_j\}_{j=1}^N)$, where $p$ is a function $p:\R\cup i\R\to\C$, called \emph{reflection coefficient}, and $\lambda_j$ and $C_j$ are complex numbers, called \emph{eigenvalues} and \emph{norming constants}. The precise connection between $(u_0,v_0)$ and $\mathcal{S}(u_0,v_0)$ is achieved through a spectral problem obtained from the Lax pair, see (\ref{e psi x=L psi}).
  \item[Step 2:]
  As it is stated later in Theorem \ref{t time evolution} the time evolution of the scattering data is trivial if $(u,v)$ evolves according to the MTM flow. Namely, at any time $t$ we then have $\mathcal{S}(u(t),v(t))=(p(\lambda) e^{-it(\lambda^2+\lambda^{-2})/2}, \{\lambda_j,C_j e^{-it(\lambda_j^2+\lambda_j^{-2})/2}\}_{j=1}^N)$.
  \item[Step 3:]
  Use the inverse map $(p, \{\lambda_j,C_j\}_{j=1}^N)\mapsto (u,v)$ and the time evolution of the scattering data to compute the solution of (\ref{e mtm}) at time $t>0$.
\end{description}
Even though the inverse scattering machinery is a powerful tool in many nonlinear equations, so far it has not been applied to the question of pointwise asymptotic behavior of (\ref{e mtm}). Certainly, one reason for this might be the fact that particularly Step 3 of the method is presented in \cite{Kuznetsov1977} only formally. In \cite{Villarroel1991} the treatment of the inverse problem in terms of Riemann-Hilbert problems is somehow sketchy and in \cite{Zhou89} an abstract condition for solvability through Riemann-Hilbert problems is given. But, thanks to the transformation (\ref{e transformation of psi}) discovered in the recent work \cite{PelinovskySaalmann2017b}, it is known how the inverse scattering transform can be developed rigorously with the Riemann-Hilbert approach. Once a differential equation is transformed into an oscillatory Riemann-Hilbert problem, the steepest descent method of Deift and Zhou \cite{DeiftZhou1993} (extended in various ways, see \cite{Do2011} for instance and the references therein) can be used in evaluating the long-time behavior of the respective equation. As an example for the success of this method we want to mention the series of papers \cite{Liu2016,LiuPerrySulem2017,JenkinsLiuPerrySulem2018a, JenkinsLiuPerrySulem2018b}, where the asymptotic behavior of the derivative nonlinear Schr\"{odinger} equation is obtained in four steps: inverse scattering and global well-posedness without solitons, \cite{Liu2016}; asymptotic behavior for soliton-free initial data, \cite{LiuPerrySulem2017}; inverse Scattering and global well-posedness with solitons, \cite{JenkinsLiuPerrySulem2018a}; and finally soliton resolution, \cite{JenkinsLiuPerrySulem2018b}. The present paper together with the preparatory work \cite{PelinovskySaalmann2017b} can be seen as an analogue to these four papers. As in \cite{LiuPerrySulem2017} our first main result deals with the specific class of initial data $(u_0,v_0)\in\mathcal{G}_0$, see Definition \ref{d G}. These are soliton-free initial data, i.e. functions for which the scattering data consists of only the reflection coefficient (i.e. $N=0$). $\mathcal{G}_0$ contains a neighborhood of the zero solution $(0,0)$ and it is therefore of great importance and of special interest. The following theorem states that solutions for (\ref{e mtm}) in $\mathcal{G}_0$ scatter to a linear solution as $t\to\infty$ inside the light cone:
\begin{thm}\label{t main}
    Let $(u_0,v_0)\in\mathcal{G}_0\cap X_{2,1}$ and fix $\eps>0$. Then there exist positive constants $C=C(u_0,v_0,\eps)$ and $\tau_0=\tau_0(u_0,v_0,\eps)$ and bounded functions $f_{\pm}:(-1,1)\to\C$ such that
    \begin{equation}\label{e asymptotic u and v}
        \begin{aligned}
            \left|u(t,x)-\frac{1}{\sqrt{t-x}}\left(e^{i\tau+i |f_-(\frac{x}{t})|^2\ln(\tau)} f_-\left(\frac{x}{t}\right)+e^{-i\tau+i |f_+(\frac{x}{t})|^2\ln(\tau)} f_+\left(\frac{x}{t}\right)\right)\right|&\leq C\tau^{-3/4},\\
            \left|v(t,x)-\frac{1}{\sqrt{t+x}}\left(e^{i\tau+i |f_-(\frac{x}{t})|^2\ln(\tau)} f_-\left(\frac{x}{t}\right)-e^{-i\tau+i |f_+(\frac{x}{t})|^2\ln(\tau)} f_+\left(\frac{x}{t}\right)\right)\right|&\leq C\tau^{-3/4}
        \end{aligned}
    \end{equation}
    for all $\tau:=\sqrt{t^2-x^2}>\tau_0$ and $\eps\leq \sqrt{\frac{t+x}{t-x}}\leq \eps^{-1}$. Moreover, there exists a $\lambda_0=\lambda_0(\eps)$  such that for all initial data satisfying
    \begin{equation}\label{e assumption u_0,v_0}
        \lambda:=\|u_0\|_{H^2(\R)\cap H^{1,1}(\R)}+ \|v_0\|_{H^2(\R)\cap H^{1,1}(\R)}\leq \lambda_0,
    \end{equation}
    the constant $C$ in (\ref{e asymptotic u and v}) can be chosen as $C=c\lambda$ where the constant $c$ does depend on $\eps$ only.
\end{thm}
The function space $X_{2,1}$ used in Theorem \ref{t main} and in the entire paper is defined by
\begin{equation*}
  X_{2,1}:= (H^2(\mathbb{R}) \cap H^{1,1}(\mathbb{R}))\times (H^2(\mathbb{R}) \cap H^{1,1}(\mathbb{R})).
\end{equation*}
Here, $H^k$ is the standard Sobolev space and $H^{1,1}$ is defined to be
\begin{equation*}
H^{1,1}(\mathbb{R}) = \left\{ u \in L^{2,1}(\mathbb{R}), \quad \partial_x u \in L^{2,1}(\mathbb{R}) \right\}
\end{equation*}
with $\|f\|_{L^{p,s}}^p:=\int (1+x^2)^{ps/2} |f(x)|^pdx$. As it is pointed out in \cite[Remark 10]{PelinovskySaalmann2017b}, the space $X_{2,1}$ for $(u,v)$ is optimal for using the method of inverse scattering.
We need to mention that in \cite{Candy16} the same asymptotic behavior (\ref{e asymptotic u and v}) is already derived, even with a better error term and without the restriction $\eps\leq \sqrt{\frac{t+x}{t-x}}\leq \eps^{-1}$. However, our result has three main features:
\begin{itemize}
  \item [(i)] Compared to \cite{Candy16} our assumptions $(u_0,v_0)\in\mathcal{G}_0\cap X_{2,1}$ and (\ref{e assumption u_0,v_0}) on the initial data are an improvement.
  \item [(ii)] Whereas in \cite{Candy16} the functions $f_{\pm}$ are given implicitly, we are able to express them explicitly in terms of the reflection coefficient, see (\ref{e |f pm|^2}) and (\ref{e arg(f pm)}).
  \item [(iii)] The method used to establish (\ref{e asymptotic u and v}) is readily extended to prove the soliton resolution conjecture for the massive Thirring model, see Theorem \ref{t solres} below. Thus, the question of asymptotic behavior is fully answered. This shows that the inverse scattering transform is best suited for the pointwise analysis of the MTM.
\end{itemize}
Our second main result is stronger than Theorem \ref{t main} and the result in \cite{Candy16} because it holds for initial data that contain finitely many solitons. Solitons are explicit solutions of (\ref{e mtm}) with specific properties, \cite{Newell1985,Drazin1989}. So-called $N$-solitons can be characterized as functions having a vanishing reflection coefficient and only discrete spectrum $\mathcal{D}=\{\lambda_j,C_j\}_{j=1}^N$.  We denote them by $(u_{sol}(t,x;\mathcal{D}),v_{sol}(t,x;\mathcal{D}))$. In general we use the notation $\mathcal{G}_N$ to denote $(u,v)$ which admit $N$ eigenvalues (but not necessarily a vanishing reflection coefficient). Hence, $N$-solitons are contained in $\mathcal{G}_N$ and as it turns out also neighborhoods of $N$-solitons belong to $\mathcal{G}_N$. Moreover, according to the following theorem all solutions to initial data in $\mathcal{G}_N$ converge to reflection-free solutions.
\begin{thm}\label{t solres}
  Let $(u_0,v_0)\in\mathcal{G}_N\cap X_{2,1}$ with scattering data $\mathcal{S}(u_0,v_0)=(p, \{\lambda_j,C_j\}_{j=1}^N)$ and fix $\eps>0$. Then there exist complex non-zero constants $\widetilde{C}_1,...,\widetilde{C}_N$ and positive constants $t_0$ and $C$ such that for $\mathcal{D}= \{\lambda_j,\widetilde{C}_j\}_{j=1}^N$
  \begin{equation}\label{e convergence to N sol}
    \begin{aligned}
      &\left|u(t,x)- u_{sol}(t,x;\mathcal{D})\right|\leq C t^{-1/2},\\
      &\left|v(t,x)- v_{sol}(t,x;\mathcal{D})\right|\leq C t^{-1/2},
    \end{aligned}
  \end{equation}
  for all $t>t_0$ and $\eps\leq \sqrt{\frac{t+x}{t-x}}\leq \eps^{-1}$. The constants $t_0$ and $C$ depend on $u_0,v_0$ and $\eps$ only.
\end{thm}
Estimates (\ref{e convergence to N sol}) can be viewed as asymptotic stability of $N$-solitons. Thus, the theorem complements the orbital stability of solitons, proved in \cite{PelinovskyShimabokuro2014}.
\medskip

The paper is organized as follows: We begin with a review of the scattering and inverse scattering transform in Section \ref{s ist}. Section \ref{s summary} gives a summary of the proof of Theorem \ref{t main}. All technical details of the steepest descent can be found in Sections \ref{s scalar RHP}--\ref{s pure db}. Theorem \ref{t solres} is proved in Section \ref{s solitons}. 
\section{Method of Inverse Scattering}\label{s ist}

The method of inverse scattering relies on the fact that the MTM system (\ref{e mtm}) has the following zero curvature representation
\begin{equation}\label{e mtm presentation}
  [\partial_x-L,\partial_t-A]=0,
\end{equation}
where the $2\times 2$-matrices $L$ and $A$ are given by
\begin{equation}\label{e def L}
  L=\frac{i}{4}(|u|^2-|v|^2)\sigma_3-
  \frac{i\lambda}{2}
  \left(
    \begin{array}{cc}
      0 & \overline{v} \\
      v & 0 \\
    \end{array}
  \right)+
  \frac{i}{2\lambda}
  \left(
    \begin{array}{cc}
      0 & \overline{u} \\
      u & 0 \\
    \end{array}
  \right)+
  \frac{i}{4}\left(\lambda^2-\frac{1}{\lambda^2}\right) \sigma_3
\end{equation}
and
\begin{equation}\label{e def A}
  A=-\frac{i}{4}(|u|^2+|v|^2)\sigma_3-
  \frac{i\lambda}{2}
  \left(
    \begin{array}{cc}
      0 & \overline{v} \\
      v & 0 \\
    \end{array}
  \right)-
  \frac{i}{2\lambda}
  \left(
    \begin{array}{cc}
      0 & \overline{u} \\
      u & 0 \\
    \end{array}
  \right)+
  \frac{i}{4}\left(\lambda^2+\frac{1}{\lambda^2}\right) \sigma_3,
\end{equation}
with is the third Pauli matrix $\sigma_3=\diag(1,-1)$.
Freezing the time variable $t$ we can define two particular matrix-valued solutions $\psi^{(\pm)}(\lambda;x)$ of the spectral problem
\begin{equation}\label{e psi x=L psi}
    \psi_x=L\psi
\end{equation}
by the following asymptotic behavior:
\begin{equation*}
    \psi^{(\pm)}(\lambda;x)\sim e^{ix\sigma_3(\lambda^2-\lambda^{-2})/4},\qquad\text{as } x\to\pm\infty.
\end{equation*}
For $u,v\in L^1(\R)\cap L^2(\R)$ it can be shown that those solutions exist for any $\lambda\in(\R\cup i\R)\setminus\{0\}$. Moreover, writing $\psi^{(\pm)}=[\psi^{(\pm)}_1|\psi_2^{(\pm)}]$ with column vectors $\psi^{(\pm)}_j$, it can be shown that $\psi^{(+)}_1$ and $\psi_2^{(-)}$ admit analytic continuations to the set $\{\lambda\in\C:\Imag(\lambda^2)>0\}$ whereas $\psi^{(-)}_1$ and $\psi_2^{(+)}$ admit analytic continuations to the set $\{\lambda\in\C:\Imag(\lambda^2)<0\}$.
From elementary ODE theory it follows that there exists a unique matrix $T(\lambda)$ not depending on $x$ such that $\psi^{(-)}(\lambda;x)=\psi^{(+)}(\lambda;x)T(\lambda)$. As it is shown in \cite{PelinovskySaalmann2017b} this matrix, which is commonly referred to as the \emph{transition matrix}, takes the form
\begin{equation}\label{e orig trans matr}
    T(\lambda)=
    \left[
      \begin{array}{cc}
        \alpha(\lambda) & -\overline{\beta(\overline{\lambda})} \\
        \beta(\lambda) & \overline{\alpha(\overline{\lambda})} \\
      \end{array}
    \right].
\end{equation}
Applying Cramer's rule and $\det\psi^{(\pm)}=1$ we find
\begin{equation}\label{e alpha=det}
    \alpha(\lambda) =\det[\psi_1^{(-)}(\lambda;x) |\psi_2^{(+)}(\lambda;x)].
\end{equation}
From the latter equation we firstly learn that $\alpha$ can be extended analytically from $\R\cup i\R$ to the second and fourth quadrant $\C_{II}\cup\C_{IV}$. Secondly, (\ref{e alpha=det}) tells us that for any zero $\lambda_j\in\C_{II}\cup\C_{IV}$ of the analytic continuation of $\alpha$ we can find a constant $\gamma_j$ such that
\begin{equation}\label{e def norming constant}
    \psi_1^{(-)}(\lambda_j;x) =\gamma_j\psi_2^{(+)}(\lambda_j;x).
\end{equation}
By symmetries we have $\alpha(-\lambda)=\alpha(\lambda)$ and thus each zero $\lambda_j$ located in the second quadrant corresponds to a zero $-\lambda_j$ in the fourth quadrant. We call zeroes $\lambda_j$ \emph{simple}, if $\alpha'(\lambda_j)\neq0$. Zeroes of $\alpha$ with $\Imag(\lambda_j^2)=0$ are called \emph{resonances} and simple zeroes with $\Imag(\lambda_j^2)\neq0$ are called \emph{eigenvalues}. All initial data for (\ref{e mtm}) considered in the present paper are \emph{generic} potentials in the following sense, \cite{Beals1984,Beals1985,Beals1988}.
\begin{mydef}\label{d G}
    For $N\in\N_0$ we define $\mathcal{G}_N\subset (L^1(\R)\cap L^2(\R))\times (L^1(\R)\cap L^2(\R))$ to be the set of functions $(u,v)$ such that the associated function $\alpha$ admits no zeroes on $\R\cup i\R$ (i.e. no resonances) and exactly $N$ simple zeroes in the second quadrant. Moreover we define the set
    \begin{equation*}
        \mathcal{G}=\bigcup_{N=0}^{\infty}\mathcal{G}_N.
    \end{equation*}
    Functions $(u,v)\in\mathcal{G}$ are called \emph{generic} potentials.
\end{mydef}

Formal results of inverse scattering can be found in \cite{Kuznetsov1977}. According to this reference we define the scattering data of $(u,v)$ in the following way.
\begin{mydef}\label{d Scattering data}
    Let $N\in\N_0$. We call
    \begin{equation}\label{e def Scattering data}
        \mathcal{S}(u,v)=(p, \{\lambda_j,C_j\}_{j=1}^N),
    \end{equation}
    the \emph{scattering data} of $(u,v)\in\mathcal{G}_N$. The \emph{reflection coefficient} $p:\R\cup i\R\to\C$ is defined in terms of the transition matrix by $p(\lambda)=\beta(\lambda)/\alpha(\lambda)$. $\lambda_1,...,\lambda_N$ are the pair-wise distinct simple zeroes of $\alpha$ in the second quadrant and $C_1,...,C_N$ are non-zero constants defined by $C_j=\gamma_j/\alpha'(\lambda_j)$ where $\gamma_j$ is supposed to be the constant as in (\ref{e def norming constant}).
\end{mydef}
Now, let us state the evolution of the scattering data in time.
This remarkable fact demonstrates the significance of the direct scattering transform because it linearizes the MTM system and thus makes the MTM system integrable.
\begin{thm}\label{t time evolution}
  Assume that $(u(t),v(t))$ solves (\ref{e mtm}) for initial data $(u_0,v_0)$ with scattering data $\mathcal{S}(u_0,v_0)=(p(\lambda;0), \{\lambda_j(0),C_j(0)\}_{j=1}^N)$. Then, at any time $t$ the scattering data $\mathcal{S}(u(t),v(t))$ of the solution $(u(t),v(t))$ is calculated from $\mathcal{S}(u_0,v_0)$ by the following:
  \begin{equation}\label{e time evolution scattering data}
    \mathcal{S}(u(t),v(t))=(p(\lambda;0) e^{-it(\lambda^2+\lambda^{-2})/2}, \{\lambda_j(0),C_j(0) e^{-it(\lambda_j^2+\lambda_j^{-2})/2}\}_{j=1}^N).
  \end{equation}
  In particular, the number of zeroes does not vary in time and the sets $\mathcal{G}_N$ are invariant under the MTM flow.
\end{thm}
In order to describe the map $(u,v)\mapsto \mathcal{S}(u,v)$ rigorously, we need to specify spaces which $\mathcal{S}$ maps into. This issue is treated with the help of explicit transformations of (\ref{e psi x=L psi}) in the following subsection.
\subsection{Transformation of the scattering data}
The key element of \cite{PelinovskySaalmann2017b} is the following pair of transformations:
\begin{equation}\label{e transformation of psi}
    \Psi(w;x)=
    \left(
      \begin{array}{cc}
        1 & 0 \\
        u(x) & \lambda^{-1} \\
      \end{array}
    \right)\psi(\lambda;x),\qquad
    \widehat{\Psi}(z;x)=
    \left(
      \begin{array}{cc}
        1 & 0 \\
        v(x) & \lambda \\
      \end{array}
    \right)\psi(\lambda;x),
\end{equation}
where $\psi(\lambda;x)$ is supposed to be a solution of the spectral problem (\ref{e psi x=L psi}). The new spectral parameters are given by
\begin{equation}\label{def w z}
    w:=\lambda^{-2},\qquad z:=\lambda^2.
\end{equation}
It is shown by direct computations that the transformations (\ref{e transformation of psi}) make (\ref{e psi x=L psi}) equivalent to spectral problems
\begin{equation}\label{e transformed spectral problems}
    \Psi_x(w;x)=\mathcal{L}(w;x)\Psi(w;x),\qquad \widehat{\Psi}_x(z;x)= \widehat{\mathcal{L}}(z;x)\widehat{\Psi}(z;x),
\end{equation}
with matrix operators $\mathcal{L}$ and $\widehat{\mathcal{L}}$. Now, the procedure described above for the operator $L$ can be adapted to the transformed operators. More precisely, for $w,z\in\R$ we define solutions $\Psi^{(\pm)}(w;x)$ and $\widehat{\Psi}^{(\pm)}(z;x)$ of (\ref{e transformed spectral problems}) satisfying the following asymptotic behavior:
\begin{equation*}
  \begin{aligned}
    &\Psi^{(\pm)}(w;x)\sim e^{-ix\sigma_3(w-w^{-1})/4},&\qquad\text{as } x\to\pm\infty,\\
    &\widehat{\Psi}^{(\pm)}(z;x)\sim e^{ix\sigma_3(z-z^{-1})/4},&\qquad\text{as } x\to\pm\infty.
  \end{aligned}
\end{equation*}
Again, this determines $x$-independent transition matrices $T_1(w)$ and $T_2(z)$ such that $\Psi^{(-)}(w;x)=\Psi^{(+)}(w;x)T_1(w)$ and $\widehat{\Psi}^{(-)}(z;x)=\widehat{\Psi}^{(+)}(z;x)T_2(z)$. The transformations (\ref{e transformation of psi}) change some symmetries and so we find $T_1$ and $T_2$ of the following form:
\begin{equation*}
    T_1(w)=
    \left[
      \begin{array}{cc}
        a(w) & -\overline{b(\overline{w})} \\
        w\,b(w) & \overline{a(\overline{w})} \\
      \end{array}
    \right],
    \qquad
    T_2(z)=
    \left[
      \begin{array}{cc}
        \widehat{a}(z) & -\overline{\widehat{b}(\overline{z})} \\
        z\,\widehat{b}(z) & \overline{\widehat{a}(\overline{z})} \\
      \end{array}
    \right].
\end{equation*}
Since the transformations (\ref{e transformation of psi}) are given in a very explicit way, we can observe the following relations between the entries of the original transition matrix $T(\lambda)$, (\ref{e orig trans matr}), and the entries of $T_1(w)$ and $T_2(w)$, respectively:
\begin{equation}\label{e rel scattering coeff}
    \begin{aligned}
        &\alpha(\lambda)=a(\lambda^{-2})= \widehat{a}(\lambda^2),\\
        &\lambda \beta(\lambda)=b(\lambda^{-2})= \lambda^2\widehat{b}(\lambda^2).
    \end{aligned}
\end{equation}
Defining the transformed reflection coefficients by  $r(w)=b(w)/a(w)$ and $\widehat{r}(z)=\widehat{b}(z)/\widehat{a}(z)$ we find that they are actually related to the original reflection coefficient $p(\lambda)$ in the following way.
\begin{equation}\label{e new reflection coeff}
    \begin{aligned}
        &r(w)=\lambda p(\lambda),&&w=\lambda^{-2}\in\R,\\
        &\widehat{r}(z)=\lambda^{-1}p(\lambda),&& z=\lambda^2\in\R.
    \end{aligned}
\end{equation}
These relations also imply a relation of $r$ and $\widehat{r}$ amongst each other:
\begin{equation}\label{e relation r and r hat}
    \widehat{r}(z)=\frac{r(z^{-1})}{z}, \qquad z\in\R\setminus\{0\}.
\end{equation}
Moreover, assuming that $a(w)$ does not vanish at any $w\in\R$ it follows by $\det T_1(w)=1$ and the definition of $r$ that $1+w|r(w)|^2>0$ for all $w\in\R$. If additional $r(w)\to 0$ as $w\to\pm\infty$ we can find a positive constant $c$ such that
\begin{equation}\label{e 1+w|r|^2>c}
    1+z|\widehat{r}(z)|^2=1+w|r(w)|^2\geq c,\quad\text{for all }w=z^{-1}\in\R.
\end{equation}
From the expressions in (\ref{e rel scattering coeff}) we also learn that $a$ admits an analytic continuation into $\C^+$, whereas $\widehat{a}$ is continued analytically into $\C^-$. Moreover, let $\{\lambda_1,...,\lambda_N\}$ be the set of zeroes of $\alpha(\lambda)$ in the second quadrant and set
\begin{equation}\label{e rel new poles}
    w_j=\lambda_j^{-2},\qquad z_j=\lambda_j^2,\qquad j=1,...,N.
\end{equation}
Then, $\{w_1,...,w_N\}\subset\C^+$ is exactly the set of zeroes of $a(w)$ in the upper half plane and, analogously, $\{z_1,...,z_N\}\subset\C^-$ forms the set of zeroes of $\widehat{a}$. The meaning of zeroes of $a$ and $\widehat{a}$ is similar to the meaning of zeroes of $\alpha$: if $a(w_j)=0$, there exists a constant $\Gamma_j$ such that $\Psi_1^{(-)}(w_j;x) =\Gamma_j\Psi_2^{(+)}(w_j;x)$. Correspondingly, if $a(z_j)=0$, then there exists a constant $\widehat{\Gamma}_j$ such that $\widehat{\Psi}_1^{(-)}(z_j;x) =\widehat{\Gamma}_j\widehat{\Psi}_2^{(+)}(z_j;x)$. The \emph{transformed norming constants} are now defined by $c_j=\Gamma_j/a'(w_j)$ and $\widehat{c}_j=\widehat{\Gamma}_j/\widehat{a}'(z_j)$. They satisfy the following relations
\begin{equation}\label{e rel new c_j}
    c_j=\frac{-2C_j}{\lambda_j^4},\qquad \widehat{c}_j=2C_j.
\end{equation}
The full set of \emph{transformed scattering data} reads as follows:
\begin{equation}\label{e transformed scatterin data}
        \mathcal{S}_w(u,v) =(r, \{w_j,c_j\}_{j=1}^N),\qquad \mathcal{S}_z(u,v) =(\widehat{r},\{z_j,\widehat{c}_j\}_{j=1}^N).
\end{equation}
Since $\mathcal{S}_w(u,v)$ and $\mathcal{S}_z(u,v)$ can be computed also without solving the spectral problems (\ref{e transformed spectral problems}) but only by making use of relations (\ref{e relation r and r hat}), (\ref{e rel new poles}) and (\ref{e rel new c_j}) one might wonder why we have actually introduced the transformations (\ref{e transformation of psi}).
The usefulness of (\ref{e transformation of psi}) lies in the structure of the matrix operators $\mathcal{L}$ and $\widehat{\mathcal{L}}$. Namely, in contrast to $L$, $\mathcal{L}$ and $\widehat{\mathcal{L}}$ do not admit singularities at $w=\infty$ and $z=\infty$, respectively. This fact is the crucial investigation in \cite{PelinovskySaalmann2017b} and allows a rigorous analysis of the scattering maps $(u,v)\mapsto\mathcal{S}_w(u,v)$ and $(u,v)\mapsto\mathcal{S}_z(u,v)$. The following is proven in \cite[Theorem 2 and Remark 9]{PelinovskySaalmann2017b} and was not obtained in the formal works \cite{Kuznetsov1977} and \cite{Villarroel1991}.
\begin{thm}\label{t forward scattering}
    For $(u,v)\in\mathcal{G}\cap X_{2,1}$ we have
    \begin{equation}\label{e regularity r}
        \begin{aligned}
        &r(w)\in H_w^{1,1}(\R)\cap L^{2,-2}_w(\R), &&wr(w)\in H_w^{1,1}(\R)\cap L^{2,-2}_w(\R),\\
        &\widehat{r}(z)\in H_z^{1,1}(\R)\cap L^{2,-2}_z(\R), &&z\widehat{r}(z)\in H_z^{1,1}(\R)\cap L^{2,-2}_z(\R).\\
        \end{aligned}
    \end{equation}
    In particular, relation (\ref{e 1+w|r|^2>c}) holds. Moreover, for each $N\in\N_0$ the maps
    \begin{equation*}
        \mathcal{S}_w: \mathcal{G}_N\cap X_{2,1} \to H^{1,1}(\R)\cap L^{2,-2}(\R) \times (\C^+)^N\times (\C^*)^N
    \end{equation*}
    and
    \begin{equation*}
        \mathcal{S}_z: \mathcal{G}_N\cap X_{2,1} \to H^{1,1}(\R)\cap L^{2,-2}(\R) \times (\C^-)^N\times (\C^*)^N
    \end{equation*}
    are Lipschitz continuous.
\end{thm}
Using the relations (\ref{e new reflection coeff}), (\ref{e rel new poles}) and (\ref{e rel new c_j}) and the time evolution (\ref{e time evolution scattering data}) we can calculate the time evolution of the transformed scattering data. Assume that $(u(t),v(t))\in \mathcal{G}\cap X_{2,1}$ solves the MTM system (\ref{e mtm}) and at time $t=0$ the transformed scattering data are given as in (\ref{e transformed scatterin data}). Then, at any other time $t$ we have
\begin{equation*}
    \begin{aligned}
      \mathcal{S}_w(u(t),v(t)) &=(r(w)e^{-it(w+w^{-1})/2}, \{w_j,c_je^{-it(w_j+w_j^{-1})/2}\}_{j=1}^N),\\
      \mathcal{S}_z(u(t),v(t))& =(\widehat{r}(z)e^{-it(z+z^{-1})/2 },\{z_j,\widehat{c}_je^{-it(z_j+z_j^{-1})/2}\}_{j=1}^N).
    \end{aligned}
\end{equation*}

\subsection{Inverse scattering problem}
The inverse scattering problem is the construction of a map  $(p, \{\lambda_j,C_j\}_{j=1}^N)\mapsto (u,v)$. In fact we will work with the transformed scattering data (\ref{e transformed scatterin data}). The potential $u$ will be reconstructed from $\mathcal{S}_w(u,v)$ and $v$ will be reconstructed from $\mathcal{S}_z(u,v)$. The inverse scattering map is based on the analyticity of the functions $\Psi_x^{\pm}(\cdot;x)$ and $\widehat{\Psi}_x^{\pm}(\cdot;x)$ solving the transformed spectral problems (\ref{e transformed spectral problems}). According to \cite[eq. (4.1)]{PelinovskySaalmann2017b} the relation $\Psi^{(-)}(w;x)=\Psi^{(+)}(w;x)T_1(w)$ can be used to define a matrix-valued function $M(t,x;w)$ from  $\Psi^{(\pm)}$ and $a(w)$ that yields a solution for the following Riemann-Hilbert problem (RHP).
\begin{samepage}
\begin{framed}
\begin{rhp}\label{rhp M}
For given scattering data $(r, \{w_j,c_j\}_{j=1}^N)$ and each $(t,x)\in\R^2$, find a $2\times 2$-matrix-valued function $\C\setminus\R\ni w\mapsto M(t,x;w)$ satisfying
\begin{enumerate}
  \item $M(t,x;\cdot)$ is meromorphic in $\C\setminus\R$.
  \item $M(t,x;w)=1+\mathcal{O}\left(\frac{1}{w}\right)$ as $|w|\to\infty$.
  \item The non-tangential boundary values $M_{\pm}(t,x;w)$ exist for $w\in\R$ and satisfy the jump relation
      \begin{equation*}
        M_+=M_-(1+R)
      \end{equation*}
      where $R=R(t,x;w)$ is defined by $R(t,x;0)=0$ and
      \begin{equation}\label{e jump M}
        R(t,x;w)=
        \left[
         \begin{array}{cc}
           w|r(w)|^2  & \overline{r(w)}e^{- \frac{i}{2}(w-w^{-1})x+ \frac{i}{2}(w+w^{-1})t} \\
          wr(w)e^{ \frac{i}{2}(w-w^{-1})x- \frac{i}{2}(w+w^{-1})t} & 0\\
         \end{array}
        \right]
      \end{equation}
      for $w\in\R\setminus\{0\}$.
  \item $M(t,x;\cdot)$ has simple poles at $w_1,...,w_N,\overline{w}_1,...,\overline{w}_N$ with
      \begin{equation}\label{e res M}
         \begin{aligned}
            \res_{w=w_j}M(t,x;w)&=\lim_{w\to w_j} M(t,x;w)
            \left[
            \begin{array}{cc}
            0 & 0 \\
            w_jc_je^{\frac{i}{2}(w_j-w_j^{-1})x- \frac{i}{2}(w_j+w_j^{-1})t} & 0 \\
          \end{array}
            \right],\\
        \res_{w=\overline{w}_j}M(t,x;w)&=\lim_{w\to \overline{w}_j} M(t,x;w)
        \left[
          \begin{array}{cc}
            0 &
            -\overline{c}_j e^{ -\frac{i}{2}(\overline{w}_j- \overline{w}_j^{-1})x+ \frac{i}{2}(\overline{w}_j+ \overline{w}_j^{-1})t}\\
            0&0 \\
          \end{array}
        \right].\\
    \end{aligned}
\end{equation}
\end{enumerate}
\end{rhp}
\end{framed}
\end{samepage}
An expansion of $\Psi^{(\pm)}$ for large $w$ shows that the following holds:
\begin{equation}\label{e rec u 2}
   \overline{u}(t,x)e^{-\frac{i}{2} \int_{x}^{+\infty}|u|^2+|v|^2dy} =\lim_{|w|\to\infty}w\cdot\left[M(t,x;w)\right]_{12}.
\end{equation}
Moreover, we find that  $M$ is continuous at $w=0$ with
\begin{equation}\label{e M(0)}
  \left[M(t,x;0)\right]_{11}=e^{-\frac{i}{2} \int_{x}^{+\infty}|u|^2+|v|^2dy}.
\end{equation}
In either of the equations (\ref{e rec u 2}) and (\ref{e M(0)}) we use the notation $[\cdot]_{ij}$ to denote the $i$-$j$-entry of the matrix inside the brackets.
As it is explained in \cite{PelinovskySaalmann2017b,Pelinovsky2016}, RHP \ref{rhp M} is uniquely solvable under the assumptions (\ref{e 1+w|r|^2>c}) and (\ref{e regularity r}) for the case of no eigenvalues, i.e. $N=0$. As explained in \cite[Lemma 5.3]{Saalmann2017}, poles can always be added by means of a B\"{a}cklund transformation. Thus, solvability of RHP \ref{rhp M} is inductively guaranteed for any $N\in\N$. The solvability of the Riemann--Hilbert problem and the reconstruction formulae (\ref{e rec u 2}) and (\ref{e M(0)}) guarantee the existence of a map $\mathcal{S}_w(u,v)\mapsto u$. This is virtually half of the inverse map for the direct scattering map $(u,v)\mapsto\mathcal{S}_w(u,v)$. The remaining part $\mathcal{S}_z(u,v)\mapsto v$ is found by means of a second RHP. We refer to \cite[eq. (4.1)]{PelinovskySaalmann2017b}, where a matrix-valued function $\widehat{M}$ is defined in terms of $\widehat{\Psi}^{(\pm)}$ and $\widehat{a}(w)$ that yields a solution for the following Riemann--Hilbert problem:
\begin{samepage}
\begin{framed}
\begin{rhp}\label{rhp M hat}
For given scattering data $(\widehat{r}, \{z_j,\widehat{c}_j\}_{j=1}^N)$ and each $(t,x)\in\R^2$, find a $2\times 2$-matrix-valued function $\C\setminus\R\ni z\mapsto \widehat{M}(t,x;z)$ satisfying
\begin{enumerate}
  \item $\widehat{M}(t,x;\cdot)$ is meromorphic in $\C\setminus\R$.
  \item $\widehat{M}(t,x;z)=1+\mathcal{O}\left(\frac{1}{z}\right)$ as $|z|\to\infty$.
  \item The non-tangential boundary values $\widehat{M}_{\pm}(t,x;z)$ exist for $z\in\R$ and satisfy the jump relation
      \begin{equation*}
        \widehat{M}_+=\widehat{M}_-(1+\widehat{R})
      \end{equation*}
      where $\widehat{R}=\widehat{R}(t,x;z)$ is defined by $\widehat{R}(t,x;0)=0$ and
      \begin{equation}\label{e jump M hat}
       \widehat{R}(t,x;z)=
        \left[
         \begin{array}{cc}
           0  & -\overline{\widehat{r}(z)}e^{ \frac{i}{2}(z-z^{-1})x+ \frac{i}{2}(z+z^{-1})t} \\
          -z\widehat{r}(z)e^{ -\frac{i}{2}(z-z^{-1})x- \frac{i}{2}(z+z^{-1})t} & z|\widehat{r}(z)|^2 \\
         \end{array}
        \right].
      \end{equation}
      for $z\in\R\setminus\{0\}$.
  \item $\widehat{M}(t,x;\cdot)$ has simple poles at $z_1,...,z_N,\overline{z}_1,...,\overline{z}_N$ with
      \begin{equation}\label{e res M hat}
      \begin{aligned}
        \res_{z=z_j}\widehat{M}(t,x;z)&=\lim_{z\to z_j} \widehat{M}(t,x;w)
        \left[
          \begin{array}{cc}
            0 & 0 \\
            z_j\widehat{c}_je^{ -\frac{i}{2}(z_j-z_j^{-1})x- \frac{i}{2}(z_j+z_j^{-1})t} & 0 \\
          \end{array}
        \right],\\
        \res_{z=\overline{z}_j}\widehat{M}(t,x;z)&=\lim_{z\to \overline{z}_j} \widehat{M}(t,x;z)
        \left[
          \begin{array}{cc}
            0 &
            -\overline{\widehat{c}}_j e^{+ \frac{i}{2} (\overline{z}_j-\overline{z}_j^{-1})x+ \frac{i}{2}(\overline{z}_j+ \overline{z}_j^{-1})t}\\
            0&0 \\
          \end{array}
        \right].\\
      \end{aligned}
      \end{equation}
\end{enumerate}
\end{rhp}
\end{framed}
\end{samepage}
Analogously to (\ref{e rec u 2}) and (\ref{e M(0)}), the following reconstruction formulae are available.
\begin{equation}\label{e rec v 2}
 \overline{v}(t,x)e^{\frac{i}{2} \int_{x}^{+\infty}|u|^2+|v|^2dy} =\lim_{|z|\to\infty} z\cdot\left[\widehat{M}(t,x;z)\right]_{12},
\end{equation}
\begin{equation}\label{e M hat (0)}
        \left[\widehat{M}(t,x;0)\right]_{11} =e^{\frac{i}{2} \int_{x}^{+\infty}|u|^2+|v|^2dy}.
\end{equation}
Again, solvability of RHP \ref{rhp M hat} is ensured by assuming (\ref{e 1+w|r|^2>c}) and (\ref{e regularity r}). Thus the construction of the map $\mathcal{S}_z(u,v)\mapsto v$ is completed. For the following rigorous result we refer again to \cite{PelinovskySaalmann2017b} for the case of $N=0$ and to \cite{Saalmann2017} for the case of $N\geq 1$.
\begin{thm}\label{t inverse map}
   Assume data $(r, \{w_j,c_j\}_{j=1}^N)$ and $(\widehat{r},\{z_j,\widehat{c}_j\}_{j=1}^N)$ to satisfy relations (\ref{e relation r and r hat}), (\ref{e rel new poles}) and (\ref{e rel new c_j}). Moreover, let the reflection coefficients $r$ and $\widehat{r}$ admit the regularity (\ref{e regularity r}). Then the pair $(u(t,\cdot),v(t,\cdot))$ obtained by means of RHP's \ref{rhp M} and \ref{rhp M hat} and formulae (\ref{e rec u 2}), (\ref{e M(0)}), (\ref{e rec v 2}) and (\ref{e M hat (0)}) lies in the space $X_{2,1}$ for all $t\in\R$ and solves (\ref{e mtm}).
\end{thm}
\begin{remark}
  Without requiring relations (\ref{e relation r and r hat}), (\ref{e rel new poles}) and (\ref{e rel new c_j}) in Theorem \ref{t inverse map} we can still obtain a pair of functions $(u(t,\cdot),v(t,\cdot))\in X_{2,1}$ but in general it is not going to be a solution of (\ref{e mtm}).
\end{remark}
\begin{remark}\label{r e^i alpha u}
  Assume that $M$ solves RHP \ref{rhp M} for scattering data   $(r, \{w_j,c_j\}_{j=1}^N)$. Moreover, assume that $\widetilde{M}$ is a solution for RHP \ref{rhp M} with scattering data   $(e^{i\alpha} r, \{w_j,e^{i\alpha}c_j\}_{j=1}^N)$, where $\alpha$ is a fixed real number. Then a direct computation shows that
  \begin{equation*}
    \widetilde{M}(t,x;w)=
    \left[
      \begin{array}{cc}
        e^{-i\alpha/2} & 0 \\
        0 & e^{i\alpha/2} \\
      \end{array}
    \right]
    M(t,x;w)
    \left[
      \begin{array}{cc}
        e^{i\alpha/2} & 0 \\
        0 & e^{-i\alpha/2} \\
      \end{array}
    \right]
  \end{equation*}
  for all $w\in\C\setminus\R$. Thus we may conclude that $[M(t,x;0)]_{11}=[\widetilde{M}(t,x;0)]_{11}$ and
  \begin{equation*}
    \lim_{|w|\to\infty} w\cdot\left[\widetilde{M}(t,x;w)\right]_{12}=
    e^{-i\alpha}\lim_{|w|\to\infty} w\cdot\left[M(t,x;w)\right]_{12}.
  \end{equation*}
  Taking into account the reconstruction formulae (\ref{e rec u 2}) and (\ref{e M(0)}) we learn from this observation, that if a function $u$ belongs to scattering data $(r, \{w_j,c_j\}_{j=1}^N)$, then $e^{i\alpha}u$ belongs to scattering data $(e^{i\alpha} r, \{w_j,e^{i\alpha}c_j\}_{j=1}^N)$. Repeating the same argument for RHP \ref{rhp M hat} we may summarize the statements as follows. If $\mathcal{S}(u,v)=(p, \{\lambda_j,C_j\}_{j=1}^N)$, then $\mathcal{S}(e^{i\alpha}u,e^{i\alpha}v)=(e^{i\alpha}p, \{\lambda_j,e^{i\alpha}C_j\}_{j=1}^N)$. This property of the scattering map will be useful in the proof of Theorem \ref{t solres} and represents the invariance of (\ref{e mtm}) under phase shifts $(u,v)\mapsto (e^{i\alpha}u,e^{i\alpha}v)$.
\end{remark}
We close the overview by introducing the notion of \emph{soliton solutions}:
\begin{mydef}\label{d solitons}
  In the case where the initial data generate eigenvalues $\lambda_1,...,\lambda_N$ but $p(\lambda)=0$ for all $\lambda\in(\R\cup i\R)\setminus\{0\}$ we call the solution of (\ref{e mtm}) an \emph{$N$-soliton solution} or \emph{multi-soliton solution}. For $\mathcal{D}= \{\lambda_j,C_j\}_{j=1}^N\subset (\C_{II})^N\times (\C^*)^N$ we use the notation
  \begin{equation*}
    (u_{sol}(t,x;\mathcal{D}),v_{sol}(t,x;\mathcal{D}))
  \end{equation*}
  for the associated $N$-soliton.
\end{mydef}
For $N=1$ and $r\equiv 0$ and $\widehat{r}\equiv 0$, respectively, RHP's \ref{rhp M} and \ref{rhp M hat} reduce to a system of linear equations which can be solved explicitly. As a result of straight forward computations we obtain the known explicit one-soliton solutions for (\ref{e mtm}).
\begin{equation}\label{e 1-sol}
  \begin{aligned}
    u_{sol}(t,x;\{\lambda_1,C_1\})&=\delta^{-1} \sin(\gamma) \,\sech\left[E(x-\nu t-x_0)-i\frac{\gamma}{2}\right]e^{-i\beta(t+\nu x)+i\phi_0}\\
    v_{sol}(t,x;\{\lambda_1,C_1\})&=-\delta \sin(\gamma) \,\sech\left[E(x-\nu t-x_0)+i\frac{\gamma}{2}\right]e^{-i\beta(t+\nu x)+i\phi_0}
  \end{aligned}
\end{equation}
Here we set $\lambda_1=\delta e^{i\gamma/2}$ and the physical parameters are
\begin{equation*}
  E=\frac{\delta^2+\delta^{-2}}{2} \sin(\gamma),\quad \beta  =\frac{\delta^2+\delta^{-2}}{2} \cos(\gamma),\quad \nu = \frac{\delta^{-2}-\delta^{2}}{\delta^{-2}+\delta^{2}}.
\end{equation*}
The parameters $x_0$ and $\phi_0$ are determined by $C_1$. The exact formulas can be found in \cite{Kuznetsov1977}. Explicit expressions for multi-solitons can be in principal obtained by solving the corresponding RHP's. However, expressions get very bulky. Instead we want to mention that we can apply Theorem \ref{t solres1} to find the following: If $\mathcal{D}=\{\lambda_j,C_j\}_{j=1}^N$ and $|\lambda_j|\neq |\lambda_k|$ for $j\neq k$, then
\begin{equation*}
  (u_{sol}(t,x;\mathcal{D}),v_{sol}(t,x;\mathcal{D}))\sim
  \sum_{j=1}^{N} (u_{sol}(t,x;\{\lambda_j,C_j^{(\pm)}\}), v_{sol}(t,x;\{\lambda_j,C_j^{(\pm)}\})),\quad\text{as } t\to \pm\infty.
\end{equation*}
Thus, a multi-soliton breaks up into individual one-solitons of the form (\ref{e 1-sol}). The modified norming constants $C_j^{(\pm)}$ are given seperately for $+$ and $-$ as described in \cite{Kuznetsov1977}.  
\section{Strategy of the proof of Theorem \ref{t main} }\label{s summary}
The summary of forward and inverse scattering in the preceding section shows that to analyze the asymptotics of (\ref{e mtm}) we need to analyze RHP's \ref{rhp M} and \ref{rhp M hat} and evaluate formulae (\ref{e rec u 2}), (\ref{e M(0)}), (\ref{e rec v 2}) and (\ref{e M hat (0)}).
\subsection{New coordinates}
Although the asymptotics (\ref{e asymptotic u and v}) are given in laboratory coordinates, inside the light-cone $\{(x,t)\in\R^2:t>|x|\}$ which forms  the physically interesting region for the massive Thirring model, we define two sets of coordinates, $(\tau,w_0)$ and $(\tau,z_0)$, by the following:
\begin{equation}\label{e def tau w0 z0}
  \tau:=\sqrt{t^2-x^2}\in\R^+,\qquad w_0:= \sqrt{\frac{t+x}{t-x}}\in\R^{+},\qquad z_0:=w_0^{-1}.
\end{equation}
As it is illustrated in Figure \ref{f lightcone}, \begin{figure}
\begin{center}
\begin{tikzpicture}[domain=-5:5]

\draw[very thin,color=gray] (-5.1,-.1) grid (5.1,5.1);
\draw[->] (-5.2,0) -- (5.2,0) node[below] {$x$};
\draw[->] (0,-0.2) -- (0,5.6) node[left] {$t$};
\draw[domain=-1.4:1.4] plot (\x,{sqrt(1+\x*\x)});
\draw[domain=-2.4:2.4] plot (\x,{sqrt(4+\x*\x)});
\draw[domain=-3.4:3.4] plot (\x,{sqrt(9+\x*\x)});
\draw[domain=-3:3] plot (\x,{sqrt(16+\x*\x)});
\draw[domain=0:3.65] plot (\x,1.38*\x) ;
\draw[domain=0:2.58] plot (\x,1.96*\x) ;
\draw[domain=0:1.64] plot (\x,3.08*\x) ;
\draw[domain=0:0.8] plot (\x,6.31*\x) ;
\draw[domain=-3.65:0] plot (\x,-1.38*\x) ;
\draw[domain=-2.58:0] plot (\x,-1.96*\x) ;
\draw[domain=-1.64:0] plot (\x,-3.08*\x) ;
\draw[domain=-0.8:0] plot (\x,-6.31*\x);
\draw[domain=-5.05:0, dashed] plot (\x,-\x) ;
\draw[domain=0:5.05, dashed] plot (\x,\x) ;

\node[rotate=42] at (1.7,2.) {\tiny$\tau=1$};
\node[rotate=34] at (1.7,2.8) {\tiny$\tau=2$};
\node[rotate=34] at (2.4,4.) {\tiny$\tau=3$};
\node[rotate=-45] at (-6.2,6.2) {$w_0=0$, $z_0=\infty$};
\node[rotate=45] at (6.2,6.2) {$w_0=\infty$, $z_0=0$};
\node[rotate=-54] at (-4.5,6.3) {$w_0=\varepsilon$, $z_0=\varepsilon^{-1}$};
\node[rotate=54] at (4.55,6.4){$w_0=\varepsilon^{-1}$, $z_0=\varepsilon$};

\path[fill=black,opacity=0.21,domain=-1.052:1.052] plot (\x,{sqrt(1+\x*\x)}) ;
\path[fill=black,opacity=0.21] (1.052,1.451) -- (3.65,5.037) -- (-3.65,5.037) -- (-1.052,1.451) -- (1.052,1.451);

\end{tikzpicture}
\end{center}
  \caption{Illustration of the new coordinates $(\tau,w_0)$ and $(\tau,z_0)$.
  }\label{f lightcone}
\end{figure}
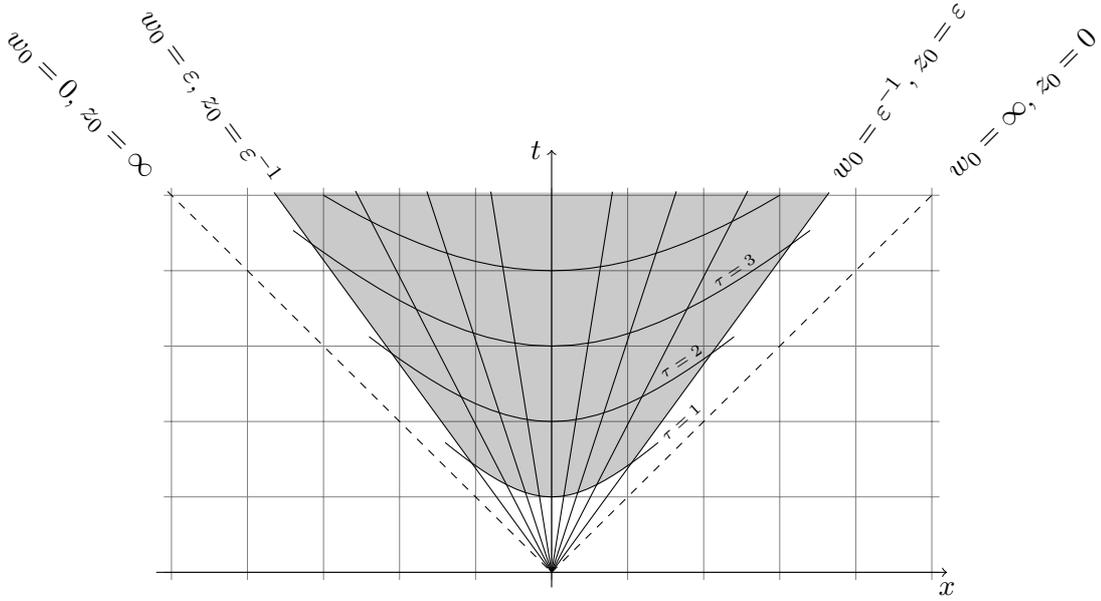
$w_0$ and $z_0$ are constant along rays where $\frac{x}{t}=const.$
The purpose of these new coordinates is as follows. Using (\ref{e def tau w0 z0}) it is easily checked that the phase functions occurring in (\ref{e jump M}) and (\ref{e jump M hat}), respectively, can be simplified in the following way:
\begin{equation*}
    \begin{aligned}
        -\frac{i}{2}(w-w^{-1})x+ \frac{i}{2}(w+w^{-1})t&=i\tau Z\left(\frac{w}{w_0}\right), \phantom{\int_{\int_{\int}}}\\
        +\frac{i}{2}(z-z^{-1})x+ \frac{i}{2}(z+z^{-1})t &=i\tau Z\left(\frac{z}{z_0}\right),
    \end{aligned}
\end{equation*}
where $Z$ is the \emph{Joukowsky} transform defined by
\begin{equation}\label{e Joukowsky}
        Z(\zeta)= \frac{1}{2}\left(\zeta+\frac{1}{\zeta}\right).
\end{equation}
These new coordinates show that RHP's \ref{rhp M} and \ref{rhp M hat} in case of $N=0$ are structurally similar to the following Riemann-Hilbert problem:
\begin{samepage}
\begin{framed}
\begin{rhp}\label{rhp M^0}
For given functions $\rho,\breve{\rho}$ and $\tau\in\R$, find a $2\times 2$-matrix valued function $\C\setminus\R \ni \zeta\mapsto M^{(0)}(\tau;\zeta)$ which satisfies
\begin{enumerate}
  \item $M^{(0)}(\tau;\cdot)$ is analytic in $\C\setminus\R$ .
  \item $M^{(0)}(\tau;\zeta)=1+\mathcal{O} \left(\frac{1}{\zeta}\right)$ as $|\zeta|\to\infty$.
  \item The non-tangential boundary values $M^{(0)}_{\pm}(\tau;\zeta)$ exist for $\zeta\in\R$ and satisfy the jump relation
      \begin{equation*}
        M^{(0)}_+=M^{(0)}_-(1+R_{\tau}^{(0)}),
      \end{equation*}
      where
      \begin{equation}\label{e def R^0}
         R^{(0)}_{\tau}(\zeta)=
         \left[
           \begin{array}{cc}
           0 & \breve{\rho}(\zeta)e^{i\tau Z(\zeta)} \\
           \rho(\zeta)e^{-i\tau Z(\zeta)} & \rho(\zeta)\breve{\rho}(\zeta) \\
         \end{array}
      \right].
\end{equation}
\end{enumerate}
\end{rhp}
\end{framed}
\end{samepage}
The following Lemma shows precisely how one can make a connection between the two RHP's \ref{rhp M} and \ref{rhp M hat} and the Joukowsky type RHP \ref{rhp M^0}. Thus, the Lemma describes a way of eliminating the coordinates $w_0$ and $z_0$.
\begin{lem}\label{l M and M hat <-> M^0}
    \begin{itemize}
      \item[(i)]  RHP \ref{rhp M} with $N=0$ and RHP \ref{rhp M^0} are equivalent for the following choice of $\rho$ and $\breve{\rho}$ in  (\ref{e def R^0}):
          \begin{equation}\label{e def rho from r}
             \rho(\zeta):=\frac{w_0\cdot\zeta\cdot r(w_0\cdot\zeta)} {d_-(w_0\cdot\zeta) \cdot d_+(w_0\cdot\zeta)},\qquad
             \breve{\rho}(\zeta) :=\overline{r(w_0\cdot\zeta)} \cdot d_-(w_0\cdot\zeta) \cdot d_+(w_0\cdot\zeta),
          \end{equation}
          where
          \begin{equation}\label{e def d}
            \left\{
              \begin{array}{ll}
                \displaystyle d(w)=\exp\left\{\frac{1}{2\pi i} \int_{-\infty}^{\infty} \frac{\log(1+\varsigma |r(\varsigma)|^2)}{\varsigma-w}d\varsigma\right\}, & w\in\C\setminus\R, \\
                \displaystyle
                d_{\pm}(w)=\lim_{\eps\searrow 0}d(w\pm i \eps), & w\in\R.
              \end{array}
            \right.
          \end{equation}
          The solution $M(t,x;w)$ of RHP \ref{rhp M} can be obtained from the solution $M^{(0)}(\tau,\zeta)$ of RHP \ref{rhp M^0} as
          \begin{equation}\label{e M <-> M^0}
              M(t,x;w)=M^{(0)} \left(\tau;\frac{w}{w_0}\right) \left[d\left( w\right)\right]^{\sigma_3}.
          \end{equation}
      \item [(ii)] RHP \ref{rhp M hat} with $N=0$ and RHP \ref{rhp M^0} are equivalent for the following choice of $\rho$ and $\breve{\rho}$ in (\ref{e def R^0}):
      \begin{equation}\label{e def rho from r hat}
        \rho(\zeta):=-z_0\cdot\zeta\cdot \widehat{r}(z_0\cdot\zeta),\qquad
        \breve{\rho}(\zeta) :=-\overline{\widehat{r}(z_0\cdot\zeta)}.
      \end{equation}
          The solution $\widehat{M}(t,x,z)$ of RHP \ref{rhp M hat} can be obtained from the solution $M^{(0)}(\tau,\zeta)$ of RHP \ref{rhp M^0} as
          \begin{equation}\label{e M hat <-> M^0}
              \widehat{M}(t,x,z)=M^{(0)} \left(\tau;\frac{z}{z_0}\right).
          \end{equation}
    \end{itemize}
\end{lem}
\begin{proof}
    Part $(ii)$ of the Lemma is immediate. Indeed, for $\rho$ and $\breve{\rho}$ defined as in (\ref{e def rho from r hat}) we have $\widehat{R}(t,x;z)=R^{(0)}_{\tau}(z/z_0)$. This leads to (\ref{e M hat <-> M^0}). In order to understand part $(i)$ of the Lemma we notice that the right hand side of (\ref{e M <-> M^0}) has a jump on $\R$ with jump matrix
    \begin{equation*}
        \left[d_-\left( w\right)\right]^{-\sigma_3} R^{(0)}_{\tau}\left(\frac{w}{w_0}\right) \left[d_+\left( w\right)\right]^{\sigma_3}.
    \end{equation*}
    Because of the fact, that $d_+(w)=d_-(w)(1+w|r(w)|^2)$ for $w\in\R$ we find that this jump matrix is nothing but $R(w;t,x)$ defined in (\ref{e jump M}).
\end{proof}
Fix $\eps>0$ and assume $\eps <z_0=w_0^{-1}<\eps^{-1}$. Furthermore, assume that functions $r,\widehat{r}\in H^{1,1}(\R)\cap L^{2,-2}(\R)$ satisfy (\ref{e 1+w|r|^2>c}). Then, the functions $\rho$ and $\breve{\rho}$ given by either of the formulae (\ref{e def rho from r}) or (\ref{e def rho from r hat}) satisfy
\begin{equation}\label{e assumption1}
    \rho,\breve{\rho}\in H^{1,1}(\R),
\end{equation}
\begin{equation}\label{e assumption2}
    \rho(0)=\breve{\rho}(0)=0,
\end{equation}
\begin{equation}\label{e assumption3}
    \Gamma_0:=\left\|\rho\right\|_{L^{\infty}(\R)} +\left\|\breve{\rho}\right\|_{L^{\infty}(\R)} +\left\|\frac{\rho}{1+\rho\breve{\rho}} \right\|_{L^{\infty}(\R)} +\left\|\frac{\breve{\rho}}{ 1+\rho\breve{\rho}}\right\|_{L^{\infty}(\R)}<\infty
\end{equation}
and
\begin{equation}\label{e assumption4} \rho(\zeta)\breve{\rho}(\zeta)\in\R\text{ for } \zeta\in\R.
\end{equation}
Moreover, by the Lipschitz continuity of the forward scattering, Theorem \ref{t forward scattering}, there is a constant $c>0$ depending on $\eps$ but not depending on $z_0$ and $w_0$ such that
\begin{equation}\label{e rho leq r}
    \|\rho\|_{H^{1,1}(\R)}+\|\breve{\rho}\|_{H^{1,1}(\R)}\leq c \left(\|u_0\|_{H^2(\R)\cap H^{1,1}(\R)}+ \|v_0\|_{H^2(\R)\cap H^{1,1}(\R)}\right).
\end{equation}
Particularly with regard to the reconstruction formulas (\ref{e rec u 2}), (\ref{e M(0)}), (\ref{e rec v 2}) and (\ref{e M hat (0)}) we are interested in behavior of $M^{(0)}(\tau;0)$ as $\tau\to\infty$ and the function defined by the following,
\begin{equation}\label{e def q^0}
    q^{(0)}(\tau):=\lim_{\zeta\to\infty} \zeta\left[M^{(0)}(\tau;\zeta)\right]_{12}.
\end{equation}
In order to formulate the asymptotic expansion of $M^{(0)}(\tau;0)$ and $q^{(0)}(\tau)$, it is convenient to work with the following package of definitions. Let functions $\rho$ and $\breve{\rho}$ be given:
\begin{equation}\label{e notation 1}
  \nu(\zeta):=\frac{1}{2\pi} \log\left(1+\rho(\zeta)\breve{\rho}(\zeta) \right),\qquad \nu_0^{\pm}:=\nu(\pm 1),
\end{equation}
\begin{equation}\label{e notation 2}
  \delta(\zeta):=\exp\left\{\frac{1}{i} \int_{-1}^{1}\frac{\nu(s)}{s-\zeta}ds\right\},\quad \zeta\in\C\setminus [-1,1],
\end{equation}
\begin{equation}\label{e notation 3}
    \delta^{\pm}_0=
    \exp\left\{\pm\frac{1}{i}\int_0^{\pm1} \frac{\nu(s)\mp s\cdot\nu^{\pm}_0}{s\mp 1}ds \mp \frac{1}{i} \int_0^{\mp1} \frac{\nu(s)}{s\mp 1}ds \mp i\nu_0^{\pm}\right\}.
\end{equation}
Over the course of the present paper we will repeat each of the above definitions and explain their meanings. The integral appearing in the expression for $\delta(\zeta)$ is well-defined for $\zeta=0$ due to (\ref{e assumption4}). For the moment we need definitions (\ref{e notation 1})--(\ref{e notation 3}) only in order to express the following result:
\begin{thm}\label{t main thm general}
    There exist positive constants $\Gamma_0$, $c$ and $\tau_0$ such that for any two functions $\rho$ and $\breve{\rho}$ satisfying the assumptions (\ref{e assumption1})--(\ref{e assumption4}), the Riemann-Hilbert problem \ref{rhp M^0} admits a unique solution $M^{(0)}(\tau;\zeta)$ for which the following holds for all $\tau\geq \tau_0$:
    \begin{equation}\label{e M^0(0)}
        \left|M^{(0)}(\tau;0)-\left[\delta(0)\right]^{-\sigma_3}\right| \leq c\left(\|\rho\|_{H^{1,1}(\R)}+ \|\breve{\rho}\|_{H^{1,1}(\R)}\right) \tau^{-1/2}.
    \end{equation}
    Moreover, the function $q^{(0)}(\tau)$ defined in (\ref{e def q^0}) satisfies, for all $\tau\geq \tau_0$,
    \begin{equation}\label{e asymptotics q}
        |q^{(0)}(\tau)-q^{(as)}(\tau)|\leq c\left(\|\rho\|_{H^{1,1}(\R)}+ \|\breve{\rho}\|_{H^{1,1}(\R)}\right) \tau^{-3/4},
    \end{equation}
    where the limit function $q^{(as)}(\tau)$ is given by
    \begin{equation}\label{e def q^as}
        q^{(as)}(\tau)=
        \frac{e^{-i\tau }e^{i\nu_0^{-}\ln(\tau)}}{\tau^{1/2}} \frac{\sqrt{2\pi}e^{\pi\nu_0^-/2} e^{-i\pi/4}} {\rho(-1)(\delta_0^-)^2\Gamma(i\nu_0^-)}+ \frac{e^{i\tau }e^{-i\nu_0^{+}\ln(\tau)}}{\tau^{1/2}} \frac{\sqrt{2\pi}e^{\pi\nu_0^+/2} e^{i\pi/4}} {\rho(1)(\delta_0^+)^2\Gamma(-i\nu_0^+)}
     \end{equation}
     in the case of $\rho(\pm 1)\neq 0$. If either $\rho(-1)=0$ or $\rho(1)=0$ the corresponding summand in (\ref{e def q^as}) has to be set to zero. In particular we have
     \begin{equation}\label{e |q^0|<tau^1/2}
        |q^{(0)}(\tau)|\leq c\left(\|\rho\|_{H^{1,1}(\R)}+ \|\breve{\rho}\|_{H^{1,1}(\R)}\right) \tau^{-1/2},
    \end{equation}
    for all $\tau\geq \tau_0$.
\end{thm}
The proof of this theorem is the main part of the present paper. Theorem \ref{t main} is a direct consequence of Theorem \ref{t main thm general} and Lemma \ref{l M and M hat <-> M^0}. Substituting (\ref{e def rho from r}) and (\ref{e def rho from r hat}) into (\ref{e def q^as}) we obtain after lengthy calculations (see Appendix \ref{a forumlae})
\begin{equation}\label{e |f pm|^2}
    \left|f_{\pm}\left(\frac{x}{t}\right)\right|^2= \pm \widehat{\kappa}(\pm z_0)
    \end{equation}
and
\begin{equation}\label{e arg(f pm)}
       \begin{aligned}
           \arg \left(f_{\pm}\left(\frac{x}{t}\right)\right) =&\mp \frac{\pi}{4}+\arg(\widehat{r}(\pm z_0)) +\arg(\Gamma(\mp i\widehat{\kappa}(\pm z_0))\\
           &\mp 2\int_{0}^{\pm z_0}\frac{\widehat{\kappa}(s)\mp \frac{s}{z_0} \widehat{\kappa}(\pm z_0)}{s\mp z_0}ds \pm 2 \int_0^{\mp z_0}\frac{\widehat{\kappa}(s)}{s\mp z_0} ds
           \mp \widehat{\kappa}(\pm z_0)+\int_{-z_0}^{z_0} \frac{\widehat{\kappa}(s)}{s}ds
       \end{aligned}
\end{equation}
with
\begin{equation}\label{e kappa z_0}
       \widehat{\kappa}(z)= \frac{1}{2\pi}\log(1\pm z|\widehat{r}(z)|^2)
\end{equation}
for the functions in (\ref{e asymptotic u and v}). In (\ref{e arg(f pm)}), $\Gamma$ is the gamma function.

\subsection{Summary of the proof of Theorem \ref{t main thm general}}

The long-time behavior results (\ref{e M^0(0)}) and (\ref{e asymptotics q}) are obtained through a sequence of transformations of RHP's.  The initial RHP is RHP \ref{rhp M^0} above and it has contour $\R$ and jump matrix $R^{(0)}_{\tau}$. We will use the notation RHP($\Sigma^{(j)}$,$R^{(j)}$) to denote the (normalized) Riemann-Hilbert problem \ref{rhp M^0} where $\R$ is replaced by the contour $\Sigma^{(j)}$ and the jump matrix $R^{(0)}_{\tau}$ is replaced by $R^{(j)}_{\tau}$. Then, by $M^{(j)}(\tau;\zeta)$ we will denote the solution of RHP($\Sigma^{(j)}$,$R^{(j)}_{\tau}$) and set
\begin{equation}\label{e def q^j}
 q^{(j)}(\tau) :=\lim_{|\zeta|\to\infty}\zeta\cdot[M^{(j)}(\tau;\zeta)]_{12}.
\end{equation}
The sequence of the assigned functions $q^{(j)}$ is thus determined by the sequence of pairs of contours and jump matrices which reads as follows.
\begin{equation*}
    (\R,R^{(0)})\to (\R,R^{(1)})\to  (\Sigma^{(3)},R^{(3)})
    \begin{aligned}
        \nearrow\\
        \searrow
    \end{aligned}
    \begin{aligned}
        (\Sigma^{(4-)},R^{(4-)})\to (\Sigma^{(5-)},R^{(5-)})\phantom{\int}\\
        (\Sigma^{(4+)},R^{(4+)})\to (\Sigma^{(5+)},R^{(5+)}) \phantom{\int}
    \end{aligned}
\end{equation*}
The pair $(\Sigma^{(2)},R^{(2)})$ does not appear in this schematic graph because as we will see later in this summary, in the second step it is necessary to consider a mixed $\db$-RHP instead of a pure RHP.  In what follows we give a summary of the computations without many details. We refer to the subsequent sections for full calculations.
\\
\\
\textbf{Step 1:}
The first step is standard in proofs of long-time behavior of oscillatory Riemann-Hilbert problems. In order to prepare the initial Riemann-Hilbert problem RHP($\R$,$R_{\tau}^{(0)}$) for the method of steepest descent by (\ref{e def M^1}) below, we first have to solve the following scalar Riemann-Hilbert problem.
\begin{samepage}
\begin{framed}
\begin{rhp}\label{rhp delta}
For given functions $\rho,\breve{\rho}\in L^2(\R)$ find a scalar function $\C\setminus\R\ni\zeta\mapsto \delta(\zeta)$ which satisfies
\begin{enumerate}
  \item $\delta(\zeta)$ is analytic in $\C\setminus\R$.
  \item $\delta(\zeta)=1+\mathcal{O} \left(\zeta^{-1}\right)$ as $|\zeta|\to\infty$.
  \item The non-tangential boundary values $\delta_{\pm}(\zeta)$ exist for $\zeta\in\R$ and satisfy the jump relation
      \begin{equation*}
        \delta_+(\zeta)=
        \left\{
          \begin{array}{ll}
            \delta_-(\zeta), & \zeta\in\R\setminus [-1,1], \\
            \delta_-(\zeta)\left(1+\rho(\zeta) \breve{\rho}(\zeta)\right), & \zeta\in [-1,1].
          \end{array}
        \right.
      \end{equation*}
\end{enumerate}
\end{rhp}
\end{framed}
\end{samepage}
In Section \ref{s scalar RHP} we list properties of $\delta$ and give details about the solvability of the scalar Riemann-Hilbert problem \ref{rhp delta}. In particular we find an explicit solution formula for $\delta(\zeta)$, see (\ref{e def delta}). The function $\delta$ is finally used to define the following transfomation:
\begin{equation}\label{e def M^1}
  M^{(1)}(\tau;\zeta):=M^{(0)}(\tau;\zeta) [\delta(\zeta)]^{\sigma_3}.
\end{equation}
It is easy to verify that we obtain a solution of a new Riemann-Hilbert problem RHP($\R$,$R^{(1)}_{\tau}$), where
\begin{equation}\label{e def R^1}
  R^{(1)}_{\tau}(\zeta)=
  \left\{
    \begin{array}{ll}\vspace{2mm}
      \left[
         \begin{array}{cc}
           0& \breve{\rho}\delta^{-2}e^{i\tau Z} \\
           \rho\delta^{+
         2}e^{-i\tau Z} & \rho\breve{\rho} \\
         \end{array}
        \right]
      , & \hbox{if }|\zeta|\geq 1, \\
      \left[
         \begin{array}{cc}
           \rho\breve{\rho}& \frac{\breve{\rho}\delta_-^{-2}} {1+\rho\breve{\rho}}e^{i\tau Z} \\
          \frac{\rho\delta_+^{2}}{1+\rho\breve{\rho}} e^{-i\tau Z} & 0 \\
         \end{array}
        \right]
      , & \hbox{if }|\zeta|<1.
    \end{array}
  \right.
\end{equation}
Since the factor $[\delta]^{\sigma_3}$ is diagonal the manipulation (\ref{e def M^1}) does not affect the reconstruction formula  (\ref{e def q^j}) and we have $q^{(1)}(\tau)=q^{(0)}(\tau)$ and moreover $M^{(1)}(\tau;0)=M^{(0)}(\tau;\zeta)[\delta(0)]^{\sigma_3}$.\\
\textbf{Step 2:}
The next transformation deforms the contour $\R$ to a new contour $\Sigma^{(2)}$ a picture of which is given in Figure \ref{f Sigma2}.
\begin{figure}
\begin{center}
\begin{tikzpicture}
\draw 	[->, thick]  	(-4.5,2.5) -- (-3,1) ;
\draw 	[->, thick]  	(-2,0) -- (-1,1) ;
\draw 	[->, thick]  	(-4.5,-2.5) -- (-3,-1) ;
\draw 	[->, thick]  	(-2,0) -- (-1,-1) ; ;
\draw 	[->, thick]  	(0,2) -- (1,1) ;
\draw 	[->, thick]  	(2,0) -- (3,1) ;
\draw 	[->, thick]  	(0,-2) -- (1,-1) ;
\draw 	[->, thick]  	(2,0) -- (3,-1) ;
\draw 	[thick]  	(-3,1) -- (-2,0) ;
\draw 	[thick]  	(-1,1) -- (0,2) ;
\draw 	[thick]  	(-3,-1) -- (-2,0) ;
\draw 	[thick]  	(-1,-1) -- (0,-2) ;
\draw 	[thick]  	(1,1) -- (2,0) ;
\draw 	[thick]  	(3,1) -- (4.5,2.5) ;
\draw 	[thick]  	(1,-1) -- (2,0) ;
\draw 	[thick]  	(3,-1) -- (4.5,-2.5) ;
\draw 	[thick]  	(-2,.1) -- (-2,-.1) ;
\draw 	[thick]  	(2,.1) -- (2,-.1) ;

\node at (-2,-.5) {$-1$};
\node at (2,-.5) {$1$};
\node at (-4.5,1.9) {$\Sigma^{(2)}_1$};
\node at (-1.1,1.4) {$\Sigma^{(2)}_2$};
\node at (1.2,1.4) {$\Sigma^{(2)}_3$};
\node at (4.5,1.9) {$\Sigma^{(2)}_4$};
\node at (-4.5,-1.9) {$\Sigma^{(2)}_5$};
\node at (-1.1,-1.5) {$\Sigma^{(2)}_6$};
\node at (1.1,-1.5) {$\Sigma^{(2)}_7$};
\node at (4.5,-1.9) {$\Sigma^{(2)}_8$};
\end{tikzpicture}
\end{center}
  \caption{The augmented contour $\Sigma^{(2)}=\Sigma^{(2)}_1\cup...\cup\Sigma^{(2)}_8$.} \label{f Sigma2}
\end{figure}
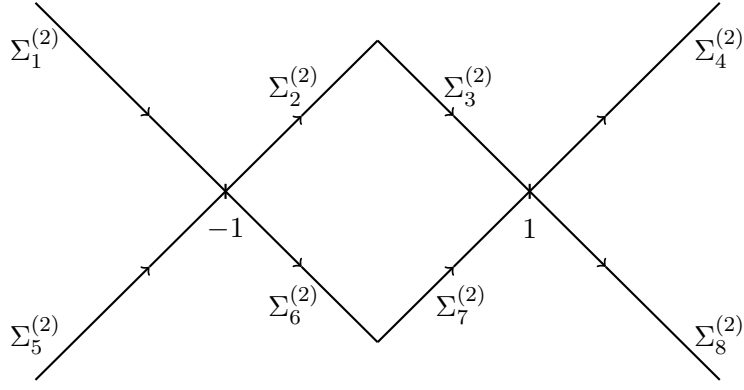
The transformation is based on the fact that $R^{(1)}_{\tau}$ defined in (\ref{e def R^1}) admits a factorisation of the form
\begin{equation}\label{e factorisation R^1}
  1+R^{(1)}_{\tau}= (1+R^{(1)}_{L})(1+R^{(1)}_{R}),
\end{equation}
where
\begin{equation}\label{e def R^1L and R^1R}
  (R^{(1)}_{L},R^{(1)}_{R})=
  \left\{
    \begin{array}{ll}\vspace{2mm}
      \left(
      \left[
         \begin{array}{cc}
           0& 0 \\
           \rho\delta^{
         2}e^{-i\tau Z} & 0 \\
         \end{array}
        \right],
       \left[
         \begin{array}{cc}
           0& \breve{\rho}\delta^{-2}e^{i\tau Z} \\
           0& 0\\
         \end{array}
        \right]
      \right), & \hbox{if }|\zeta|\geq 1,\\
      \left(
      \left[
         \begin{array}{cc}
           0&\frac{\breve{\rho}} {1+\rho\breve{\rho}}\delta_-^{-2}e^{i\tau Z} \\
         0& 0 \\
         \end{array}
        \right],
        \left[
         \begin{array}{cc}
          0&0\\
          \frac{\rho}{1+\rho\breve{\rho}} \delta_+^{2}e^{-i\tau Z} & 0 \\
         \end{array}
        \right]
      \right), & \hbox{if }|\zeta|<1.
    \end{array}
  \right.
\end{equation}
As we will specify in Section \ref{s dbextension} there exists a matrix-valued function $\zeta\to\mathcal{W}(\tau;\zeta)$ of the form (\ref{e def W}) which is continuous on $\C\setminus (\R\cup\Sigma^{(2)})$ and satisfies for $\zeta\in\R$ the following.
\begin{equation}\label{e condition1 W}
  \begin{aligned}
    \mathcal{W}_+&=1-R^{(1)}_{R},\\
    \mathcal{W}_-&=1+R^{(1)}_{L}.
  \end{aligned}
\end{equation}
Here $\mathcal{W}_{\pm}$ are the boundary values of $\mathcal{W}$ as $\pm\Imag(\zeta)\downarrow 0$. It can be verified easily by the triangularity of $R^{(1)}_{R}$ and $R^{(1)}_{L}$ that the new unknown
\begin{equation}\label{e def M^2}
    M^{(2)}(\tau;\zeta):=M^{(1)}(\tau;\zeta) \mathcal{W}(\tau;\zeta),
\end{equation}
has no jump on the real axis. The discontinuity of $\mathcal{W}$ on $\Sigma^{(2)}$ will be arranged in such a way that $M^{(2)}_+=M^{(2)}_-(1+R_{\tau}^{(2)})$ on $\Sigma^{(2)}$ with
\begin{equation}\label{e def R^2-}
    R_{\tau}^{(2)}(\zeta):=
    \left\{
      \begin{array}{ll}    \vspace{1mm}
        \left[
          \begin{array}{cc}
             0& \breve{\rho}(-1)(\delta_0^{-})^{-2} (-(\zeta+1))^{-2i\nu_0^{-}}e^{i\tau Z(\zeta)} \\
            0 & 0 \\
          \end{array}
        \right]
        , & \zeta\in\Sigma^{(2)}_1\\ \vspace{1mm}
        \left[
          \begin{array}{cc}
            0 & 0 \\
            \frac{\rho(-1)}{1+\rho(-1)\breve{\rho}(-1)} (\delta_0^{-})^{2}(-(\zeta+1))^{2i\nu_0^{-}} e^{-i\tau Z(\zeta)}(1-\chi(\zeta)) & 0 \\
          \end{array}
        \right]
        , &  \zeta\in\Sigma^{(2)}_2,\\ \vspace{1mm}
        \left[
          \begin{array}{cc}
            0 & 0 \\
            \rho(-1)(\delta_0^{-})^{2}(-(\zeta+ 1))^{2i\nu_0^{-}}e^{-i\tau Z(\zeta)} & 0 \\
          \end{array}
        \right], & \zeta\in\Sigma^{(2)}_5,\\ \vspace{1mm}
        \left[
          \begin{array}{cc}
             0& \frac{\breve{\rho}(-1)}{1+\rho(-1)\breve{\rho}(-1)} (\delta_0^{-})^{-2}(-(\zeta+1))^{-2i\nu_0^{-}} e^{i\tau Z(\zeta)}(1-\chi(\zeta)) \\
            0 & 0 \\
          \end{array}
        \right], &\zeta\in\Sigma^{(2)}_6, \\
      \end{array}
    \right.
\end{equation}
and
\begin{equation}\label{e def R^2+}
    R^{(2)}_{\tau}(\zeta):=
    \left\{
      \begin{array}{ll}    \vspace{1mm}
        \left[
          \begin{array}{cc}
             0&0 \\
            \frac{\rho(1)}{1+\rho(1)\breve{\rho}(1)} (\delta_0^{+})^{2}(\zeta-1)^ {-2i\nu_0^{+}}e^{-i\tau Z(\zeta)}(1-\chi(\zeta)) & 0 \\
          \end{array}
        \right]
        , & \zeta\in\Sigma^{(2)}_3\\ \vspace{1mm}
        \left[
          \begin{array}{cc}
            0 &  \breve{\rho}(1)(\delta_0^{+})^{-2} (\zeta-1)^{2i\nu_0^{+}} e^{i\tau Z(\zeta)}  \\
            0 & 0 \\
          \end{array}
        \right]
        , &  \zeta\in\Sigma^{(2)}_4,\\ \vspace{1mm}
        \left[
          \begin{array}{cc}
             0& \frac{\breve{\rho}(1)}{1+\rho(1)\breve{\rho}(1)} (\delta_0^{+})^{-2}(\zeta-1)^ {2i\nu_0^{+}}e^{i\tau Z(\zeta)}(1-\chi(\zeta)) \\
            0 & 0 \\
          \end{array}
        \right]
        , & \zeta\in\Sigma^{(2)}_7,\\ \vspace{1mm}
        \left[
          \begin{array}{cc}
            0 & 0 \\
            \rho(1)(\delta_0^{+})^{2}(\zeta-1)^{-2i\nu_0^{+}} e^{-i\tau Z(\zeta)} & 0 \\
          \end{array}
        \right]
        , &\zeta\in\Sigma^{(2)}_8, \\
      \end{array}
    \right.
\end{equation}
We notice that $R^{(2)}_{\tau}$ is determined as follows. Scattering data are replaced by their values at $-1$, see (\ref{e def R^2-}), or by their values at $+1$, see (\ref{e def R^2+}). Powers of $\delta$ are replaced by their asymptotic forms near $-1$, see (\ref{e def R^2-}), or by their asymptotic forms near $+1$, see (\ref{e def R^2+}). We refer to Proposition \ref{p delta near +-1} which provides these asymptotics of $\delta$ near $\pm1$. The function $\chi$ is a smooth cut-off function introduced for technical reasons. The crucial point in the definition of $R^{(2)}_{\tau}$, (\ref{e def R^2-}) and (\ref{e def R^2+}), is the fact that on $\Sigma^{(2)}$ any of the factors $e^{\pm i\tau Z(\zeta)}$ is decaying exponentially as $\tau\to\infty$. Indeed, let us make the following observation, see Figure \ref{f signaturetable},
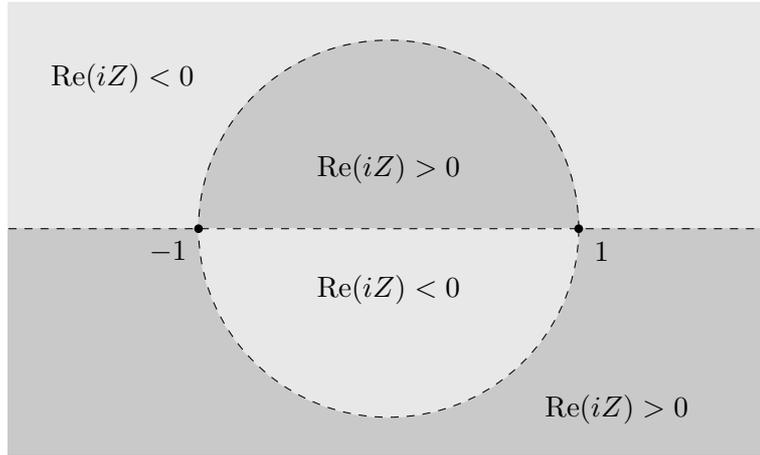
\begin{figure}
\begin{center}
\begin{tikzpicture}
\filldraw[fill=black!9,draw=black!9]
(-5,0) rectangle (5,3);
\filldraw[fill=black!21,draw=black!21]
(-5,0) rectangle (5,-3);
\filldraw[fill=black!21,draw=black!21]
(0,0) -- (2.5,0) arc (0:180:2.5) -- cycle;
\filldraw[fill=black!9,draw=black!9]
(0,0) -- (2.5,0) arc (0:-180:2.5) -- cycle;
\draw[dashed] (0,0) circle (2.5);
\draw[dashed] (-5,0) -- (5,0) ;
\node at (-3.5,2) {$\Real(iZ)<0$};
\node at (3,-2.4) {$\Real(iZ)>0$};
\node at (0,.8) {$\Real(iZ)>0$};
\node at (0,-.8) {$\Real(iZ)<0$};
\draw [fill] (2.5,0) circle [radius=0.05];
\draw [fill] (-2.5,0) circle [radius=0.05];
\node at (-2.9,-.3) {$-1$};
\node at (2.8,-.3) {$1$};
\end{tikzpicture}
\end{center}
  \caption{Signature table for $\Real(iZ)$.}\label{f signaturetable}
\end{figure}
\begin{equation}\label{e sign of i Z}
  \Real(iZ(\zeta))
  \left\{
    \begin{array}{ll}
      >0, & \text{if}
      \left\{
        \begin{array}{ll}
          &\Imag(\zeta)>0\text{ and }|\zeta|<1,\\
          \text{or}&\Imag(\zeta)<0\text{ and }|\zeta|>1,
        \end{array}
      \right.
      \\
      <0, & \text{if}
      \left\{
        \begin{array}{ll}
          &\Imag(\zeta)>0\text{ and }|\zeta|>1,\\
          \text{or}&\Imag(\zeta)<0\text{ and }|\zeta|<1.
        \end{array}
      \right.
    \end{array}
  \right.
\end{equation}
We want to mention that the purpose of the first step (\ref{e def M^1}) was exactly to alow this construction and the contour $\Sigma^{(2)}$ fits in an optimal way to the signature table for $\Real(iZ)$.  Note that in (\ref{e factorisation R^1}) the sign of $\pm i\tau Z(\zeta)$ in $R_L^{(1)}$ and $R_R^{(1)}$ is depending on wether $\zeta\in[-1,1]$ or not.\\
In general $\mathcal{W}$ cannot be chosen as a  holomorphic function. Because it has the specific form (\ref{e def W}) we find
\begin{equation}\label{e db M^2}
    \db M^{(2)}=M^{(1)}\db\mathcal{W} =M^{(2)}\mathcal{W}^{-1}\db\mathcal{W} =M^{(2)}\db\mathcal{W}.
\end{equation}
The jump condition given by (\ref{e def R^2-}) and (\ref{e def R^2+}) and the lack of analyticity are summarized in the following mixed $\db$-RHP:
\begin{samepage}
\begin{framed}
\begin{dbrhp}\label{dbrhp M^2}
For given functions $\rho,\breve{\rho}$ and $\tau\in\R$, find a $2\times 2$-matrix valued function $\C\setminus\Sigma^{(2)} \ni \zeta\mapsto M^{(2)}(\tau;\zeta)$ which satisfies
\begin{enumerate}
  \item $M^{(2)}(\tau;\cdot)$ has continuous first partial derivatives in $\C\setminus\Sigma^{(2)}$ (with respect to $\zeta$).
  \item $M^{(2)}(\tau;\zeta)=1+\mathcal{O} \left(\frac{1}{\zeta}\right)$ as $|\zeta|\to\infty$.
  \item The non-tangential boundary values $M^{(2)}_{\pm}(\tau;\zeta)$ exist for $\zeta\in\Sigma^{(2)}$ and satisfy the jump relation
      \begin{equation}\label{e jump M^2}
        M^{(2)}_+=M^{(2)}_-(1+R_{\tau}^{(2)}),
      \end{equation}
      where $R_{\tau}^{(2)}=R_{\tau}^{(2)}(\zeta)$ is given in (\ref{e def R^2-}) and (\ref{e def R^2+}).
  \item The relation (\ref{e db M^2}) holds in $\C\setminus\Sigma^{(2)}$.
\end{enumerate}
\end{dbrhp}
\end{framed}
\end{samepage}
\textbf{Step 3:} The idea of the third step is to split the mixed $\db$-RHP \ref{dbrhp M^2} into a pure RHP and a pure $\db$-problem.  For consistency of notation we set
\begin{equation}\label{e def R^3 Sigma^3}
    \Sigma^{(3)}:=\Sigma^{(2)},\qquad
    R_{\tau}^{(3)}:=R_{\tau}^{(2)},
\end{equation}
and define $M^{(3)}$ to be the solution of the normalized pure Riemann-Hilbert problem RHP($\Sigma^{(3)}$,$R^{(3)}$). Notice that RHP($\Sigma^{(3)}$,$R^{(3)}$) and $\db$-RHP \ref{dbrhp M^2} would coincide if $\mathcal{W}$ was analytic. In general we have $M^{(2)}\neq M^{(3)}$ and seek for a function $D$ such that
\begin{equation}\label{e equ M^2=DM^3}
    M^{(2)}(\tau;\zeta)=D(\tau;\zeta)M^{(3)}(\tau;\zeta).
\end{equation}
It follows, that $D$ has to be continuous in the entire $\C$-plane. Furthermore, the following equation must hold:
\begin{equation}\label{e db D}
    \db D(\tau;\zeta)=D(\tau;\zeta) \Upsilon(\tau;\zeta),\quad\text{ where }\Upsilon(\tau;\zeta):=M^{(3)}(\tau;\zeta) \db\mathcal{W}(\tau;\zeta) \left[M^{(3)}(\tau;\zeta)\right]^{-1}.
\end{equation}
It is easy to check (\ref{e db D}) by direct calculations and we arrive at the following pure $\db$-problem for $D$.
\begin{samepage}
\begin{framed}
\begin{delbar}\label{db}
For each $\tau\in\R^+$, find a $2\times 2$-matrix valued function $\C\ni \zeta\mapsto D(\tau;\zeta)$ which satisfies
\begin{enumerate}
  \item $D(\tau;\zeta)$ is continuous in $\C$ (with respect to the parameter $\zeta$).
  \item $D(\tau;\zeta)\to1$ as $|\zeta|\to\infty$.
  \item The relation (\ref{e db D}) is satisfied.
\end{enumerate}
\end{delbar}
\end{framed}
\end{samepage}
We will solve $\db$-Problem \ref{db} in Section \ref{s pure db}. It turns out that the contribution of $D$ in (\ref{e equ M^2=DM^3}) is of order $\tau^{-3/4}$, see Lemma \ref{l existence and asymptotic of db problem}. Thus, the asymptotic (\ref{e def q^as}) is mainly determined by $M^{(3)}$ which represents the Riemann-Hilbert part of the mixed $\db$-RHP \ref{dbrhp M^2}. Section \ref{s pure RHP} is devoted to RHP($\Sigma^{(3)}$,$R^{(3)}$). In contrast to, for instance, NLS and DNLS equation, at this point RHP($\Sigma^{(3)}$,$R^{(3)}$) is not solvable directly and some further work is required. The following two remaining steps are necessary.\\
\\
\textbf{Step 4:} The next step separates out the influence of the jumps on the two crosses $\Sigma^{(4-)}$ and $\Sigma^{(4+)}$, see Figure \ref{f Sigma4}.
\begin{figure}
\begin{center}
\begin{tikzpicture}
\draw 	[->]  	(-6,0) -- (-1,0) ;
\draw 	[->]  	(1,0) -- (6,0) ;
\draw 	[->]  	(-2,-2.5) -- (-2,2.5) ;
\draw 	[->]  	(2,-2.5) -- (2,2.5) ;
\draw 	[->, thick]  	(-5.5,2.5) -- (-4,1) ;
\draw 	[->, thick]  	(-3,0) -- (-2.5,.5) ;
\draw 	[->, thick]  	(-5.5,-2.5) -- (-4,-1) ;
\draw 	[->, thick]  	(-3,0) -- (-2.5,-.5) ; ;
\draw 	[->, thick]  	(2,1) -- (2.5,0.5) ;
\draw 	[->, thick]  	(3,0) -- (4,1) ;
\draw 	[->, thick]  	(2,-1) -- (2.5,-0.5) ;
\draw 	[->, thick]  	(3,0) -- (4,-1) ;
\draw 	[thick]  	(-4,1) -- (-3,0) ;
\draw 	[thick]  	(-2,1) -- (-2.5,.5) ;
\draw 	[thick]  	(-4,-1) -- (-3,0) ;
\draw 	[thick]  	(-2,-1) -- (-2.5,-.5) ;
\draw 	[thick]  	(2.5,.5) -- (3,0) ;
\draw 	[thick]  	(4,1) -- (5.5,2.5) ;
\draw 	[thick]  	(2.5,-.5) -- (3,0) ;
\draw 	[thick]  	(4,-1) -- (5.5,-2.5) ;
\draw 	[thick]  	(-3,.1) -- (-3,-.1) ;
\draw 	[thick]  	(3,.1) -- (3,-.1) ;

\node at (-3,-.4) {$-1$};
\node at (3,-.4) {$1$};
\node at (-5.6,1.9) {$\Sigma^{(4-)}_1$};
\node at (-2.5,1.1) {$\Sigma^{(4-)}_2$};
\node at (2.6,1.1) {$\Sigma^{(4+)}_1$};
\node at (5.55,1.9) {$\Sigma^{(4+)}_2$};
\node at (-5.5,-1.9) {$\Sigma^{(4-)}_3$};
\node at (-2.6,-1.1) {$\Sigma^{(4-)}_4$};
\node at (2.6,-1) {$\Sigma^{(4+)}_3$};
\node at (5.6,-1.9) {$\Sigma^{(4+)}_4$};

\draw [fill] (-2,1) circle [radius=0.05];
\draw [fill] (-2,-1) circle [radius=0.05];
\draw [fill] (2,1) circle [radius=0.05];
\draw [fill] (2,-1) circle [radius=0.05];
\end{tikzpicture}
\end{center}
  \caption{The two crosses $\Sigma^{(4-)}$ and $\Sigma^{(4+)}$.} \label{f Sigma4}
\end{figure}
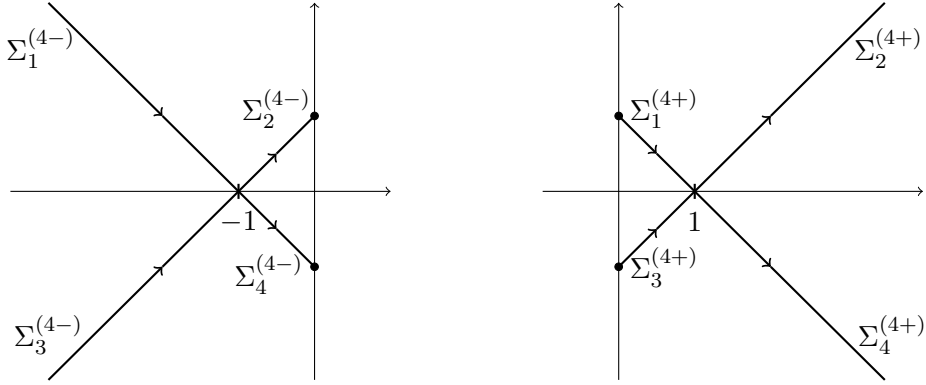
Let us split the jump matrix $R_{\tau}^{(3)}$ ($=R_{\tau}^{(2)}$) in the following way.
\begin{equation*}
    R_{\tau}^{(3)}(\zeta)= R_{\tau}^{(4-)}(\zeta) +R_{\tau}^{(4+)}(\zeta),
\end{equation*}
where
\begin{equation*}
    R_{\tau}^{(4-)}(\zeta)=0\quad\text{for }\zeta\in\Sigma^{(4+)},\qquad
    R_{\tau}^{(4+)}(\zeta)=0\quad\text{for }\zeta\in\Sigma^{(4-)}.
\end{equation*}
To each cross and corresponding jump matrix we associate a Riemann-Hilbert problem. That is, we look at solutions $M^{(4\pm)}$ of RHP($\Sigma^{(4\pm)}$,$R^{(4\pm)}$) and consider the assigned functions $q^{(4\pm)}(\tau)$. In Proposition \ref{p M^3 approx M^4} we will show that $M^{(3)}$ is approximated by the product $M^{(4-)}M^{(4+)}$ and thus $q^{(3)}(\tau)$ is approximated by the sum $q^{(4-)}(\tau)+q^{(4+)}(\tau)$. Since there are no explicit expressions available for $M^{(4-)}$ and $M^{(4+)}$ we introduce the final step 5.\\
\\
\textbf{Step 5:}
In the last step we end up with two model RHP's for which explicit solutions are known. The first model RHP is obtained from RHP($\Sigma^{(4-)}$,$R_{\tau}^{(4-)}$), by replacing the phase $Z(\zeta)$ occurring in $R_{\tau}^{(4-)}$ with its second-order Taylor expansion around the negative stationary phase point $-1$:
\begin{equation*}
    Z(\zeta)=-1-\frac{1}{2}(\zeta+1)^2 +\mathcal{O}\left(|\zeta+1|^3\right),\quad\text{as }\zeta\to -1.
\end{equation*}
We also enlarge the rays $\Sigma^{(4-)}_2$ and $\Sigma^{(4-)}_4$ and define, see Figure \ref{f Sigma5},
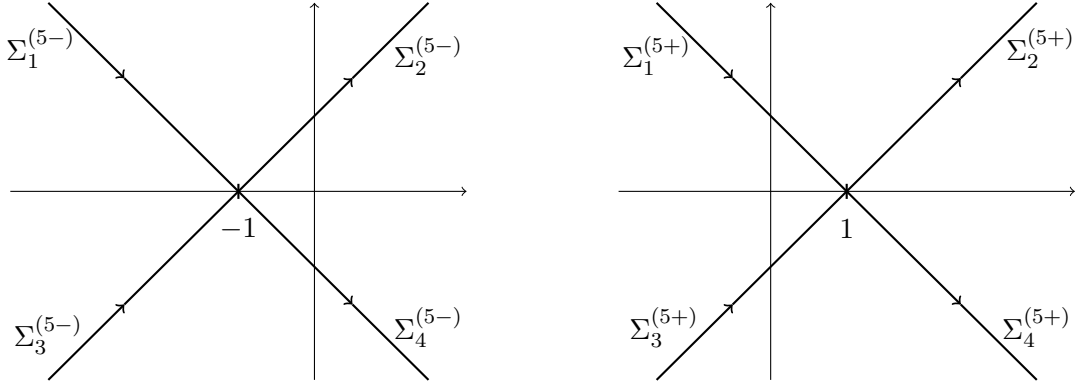
\begin{figure}
\begin{center}
\begin{tikzpicture}
\draw 	[->]  	(-7,0) -- (-1,0) ;
\draw 	[->]  	(1,0) -- (7,0) ;
\draw 	[->]  	(-3,-2.5) -- (-3,2.5) ;
\draw 	[->]  	(3,-2.5) -- (3,2.5) ;
\draw 	[->, thick]  	(-6.5,2.5) -- (-5.5,1.5) ;
\draw 	[->, thick]  	(-4,0) -- (-2.5,1.5) ;
\draw 	[->, thick]  	(-6.5,-2.5) -- (-5.5,-1.5) ;
\draw 	[->, thick]  	(-4,0) -- (-2.5,-1.5) ;
\draw 	[->, thick]  	(1.5,2.5) -- (2.5,1.5) ;
\draw 	[->, thick]  	(4,0) -- (5.5,1.5) ;
\draw 	[->, thick]  	(1.5,-2.5) -- (2.5,-1.5) ;
\draw 	[->, thick]  	(4,0) -- (5.5,-1.5) ;
\draw 	[thick]  	(-5.5,1.5) -- (-4,0) ;
\draw 	[thick]  	(-1.5,-2.5) -- (-2.5,-1.5) ;
\draw 	[thick]  	(-5.5,-1.5) -- (-4,0) ;
\draw 	[thick]  	(-1.5,2.5) -- (-2.5,1.5) ;
\draw 	[thick]  	(2.5,1.5) -- (4,0) ;
\draw 	[thick]  	(6.5,-2.5) -- (5.5,-1.5) ;
\draw 	[thick]  	(2.5,-1.5) -- (4,0) ;
\draw 	[thick]  	(6.5,2.5) -- (5.5,1.5) ;
\draw 	[thick]  	(-4,.1) -- (-4,-.1) ;
\draw 	[thick]  	(4,.1) -- (4,-.1) ;

\node at (-4,-.5) {$-1$};
\node at (4,-.5) {$1$};
\node at (-6.6,1.9) {$\Sigma^{(5-)}_1$};
\node at (-1.5,1.8) {$\Sigma^{(5-)}_2$};
\node at (1.5,1.8) {$\Sigma^{(5+)}_1$};
\node at (6.55,1.9) {$\Sigma^{(5+)}_2$};
\node at (-6.5,-1.9) {$\Sigma^{(5-)}_3$};
\node at (-1.5,-1.8) {$\Sigma^{(5-)}_4$};
\node at (1.6,-1.8) {$\Sigma^{(5+)}_3$};
\node at (6.5,-1.8) {$\Sigma^{(5+)}_4$};

\end{tikzpicture}
\end{center}
  \caption{The augmented crosses $\Sigma^{(5-)}$ and $\Sigma^{(5+)}$.} \label{f Sigma5}
\end{figure}
\begin{equation*}
  \Sigma^{(5-)}=\Sigma^{(5-)}_1\cup...\cup\Sigma^{(5-)}_4:= \left[(e^{-i\pi/4}\R_-)\cup (e^{i\pi/4}\R_+)\cup ( e^{i\pi/4}\R_-)\cup (e^{-i\pi/4}\R_+)\right]-1.
\end{equation*}
Now by extending $R_{\tau}^{(4-)}$ in a natural way we set
\begin{equation}\label{e def R^5-}
    R_{\tau}^{(5-)}(\zeta):=
    \left\{
      \begin{array}{ll}    \vspace{1mm}
        \left[
          \begin{array}{cc}
             0& \breve{\rho}(-1)(\delta_0^{-})^{-2} (-(\zeta+1))^{-2i\nu_0^{-}} e^{-i\tau\left(1+\frac{1}{2}(\zeta+1)^2\right)} \\
            0 & 0 \\
          \end{array}
        \right]
        , & \zeta\in\Sigma^{(5-)}_1\\ \vspace{1mm}
        \left[
          \begin{array}{cc}
            0 & 0 \\
            \frac{\rho(-1)}{1+\rho(-1)\breve{\rho}(-1)} (\delta_0^{-})^{2} (-(\zeta+1))^{2i\nu_0^{-}}e^{i\tau \left(1+\frac{1}{2}(\zeta+1)^2\right)} & 0 \\
          \end{array}
        \right]
        , &  \zeta\in\Sigma^{(5-)}_2,\\ \vspace{1mm}
        \left[
          \begin{array}{cc}
            0 & 0 \\
            \rho(-1)(\delta_0^{-})^{2} (-(\zeta+1))^{2i\nu_0^{-}} e^{i\tau\left(1+\frac{1}{2}(\zeta+1)^2\right)} & 0 \\
          \end{array}
        \right], & \zeta\in\Sigma^{(5-)}_3,\\ \vspace{1mm}
        \left[
          \begin{array}{cc}
             0& \frac{\breve{\rho}(-1)} {1+\rho(-1)\breve{\rho}(-1)}(\delta_0^{-})^{-2} (-(\zeta+1))^{-2i\nu_0^{-}} e^{-i\tau\left(1+\frac{1}{2}(\zeta+1)^2\right)} \\
            0 & 0 \\
          \end{array}
        \right], &\zeta\in\Sigma^{(5-)}_4. \\
      \end{array}
    \right.
\end{equation}
Analogously we derive a model RHP from RHP($\Sigma^{(4+)}$,$R^{(4+)}$) by setting
\begin{equation*}
  \Sigma^{(5+)}=\Sigma^{(5+)}_1\cup...\cup\Sigma^{(5+)}_4:= \left[(e^{-i\pi/4}\R_-)\cup (e^{i\pi/4}\R_+)\cup ( e^{i\pi/4}\R_-)\cup (e^{-i\pi/4}\R_+)\right]+1
\end{equation*}
and defining
\begin{equation}\label{e def R^5+}
    R_{\tau}^{(5+)}(\zeta):=
    \left\{
      \begin{array}{ll}    \vspace{1mm}
        \left[
          \begin{array}{cc}
             0&0 \\
            \frac{\rho(1)}{1+\rho(1)\breve{\rho}(1)} (\delta_0^{+})^{2} (\zeta-1)^{-2i\nu_0^{+}}e^{-i\tau \left(1+\frac{1}{2}(\zeta-1)^2\right)} & 0 \\
          \end{array}
        \right]
        , & \zeta\in\Sigma^{(5+)}_1\\ \vspace{1mm}
        \left[
          \begin{array}{cc}
            0 &  \breve{\rho}(1)(\delta_0^{+})^{-2} (\zeta-1)^{2i\nu_0^{+}} e^{i\tau\left(1+\frac{1}{2}(\zeta-1)^2\right)}  \\
            0 & 0 \\
          \end{array}
        \right]
        , &  \zeta\in\Sigma^{(5+)}_2,\\ \vspace{1mm}
        \left[
          \begin{array}{cc}
             0& \frac{\breve{\rho}(1)} {1+\rho(1)\breve{\rho}(1)} (\delta_0^{+})^{-2}(\zeta-1)^{2i\nu_0^{+}} e^{i\tau \left(1+\frac{1}{2}(\zeta-1)^2\right)} \\
            0 & 0 \\
          \end{array}
        \right]
        , & \zeta\in\Sigma^{(5+)}_3,\\ \vspace{1mm}
        \left[
          \begin{array}{cc}
            0 & 0 \\
            \rho(1)(\delta_0^{+})^{2} (\zeta-1)^{-2i\nu_0^{+}} e^{-i\tau\left(1+\frac{1}{2}(\zeta-1)^2\right)} & 0 \\
          \end{array}
        \right]
        , &\zeta\in\Sigma^{(5+)}_4. \\
      \end{array}
    \right.
\end{equation}
To compute the solution of RHP($\Sigma^{(5-)}$,$R^{(5-)}$) is subject of Subsection \ref{s two model RHPs}. In subsection \ref{s M^4 approx M^5} we show that the two solutions $M^{(5\pm)}$ approximate the two solutions  $M^{(4\pm)}$ in the sense that $M^{(4\pm)}=F^{(\pm)}M^{(5\pm)}$ with matrix functions $F^{(\pm)}$ close to identity, see (\ref{e solution formula M^4}) and Proposition \ref{p M^4 approx M^5}.\\
\\
\textbf{Regrouping of the transformations:}
We close this summary with regrouping the above explained transformations. Recalling successively (\ref{e def M^1}), (\ref{e def M^2}) and  (\ref{e equ M^2=DM^3}) we obtain
\begin{equation*}
  M^{(0)}(\tau;\zeta)= D(\tau;\zeta)\,M^{(3)}(\tau;\zeta)\, \left[\mathcal{W}(\tau;\zeta)\right]^{-1} [\delta(\zeta)]^{-\sigma_3}.
\end{equation*}
Using $\mathcal{W}(\tau;\zeta)=1$ for $\zeta=0$ and $\zeta\in\Omega_9\cup\Omega_{10}$ and making use of the fact that $[\delta(\zeta)]^{-\sigma_3}$ is diagonal we derive the following two solution formulas for the expressions we want to evaluate in Theorem \ref{t main thm general}:
\begin{equation}\label{e M^0 and q^0 solution formula}
  \begin{aligned}
     M^{(0)}(\tau;0)&= D(\tau;0)\,M^{(3)}(\tau;0)\, [\delta(0)]^{-\sigma_3},\\
     q^{(0)}(\tau)&=  \lim_{\zeta\to\infty} \zeta\left[D(\tau;\zeta)\right]_{12} +\lim_{\zeta\to\infty} \zeta\left[M^{(3)}(\tau;\zeta)\right]_{12}.
  \end{aligned}
\end{equation}
Based on these formulas, Theorem \ref{t main thm general} is a direct consequence of Propositions \ref{p M^5} -- \ref{p M^3 approx M^4} and Lemma \ref{l existence and asymptotic of db problem} below. Therein the following is shown:
\begin{center}
\begin{tabular}{|c|l|}
  \hline
  Propositions \ref{p M^5} -- \ref{p M^3 approx M^4} & $\begin{array}{l}
  \displaystyle M^{(3)}(\tau;0)=
       \left[
          \begin{array}{cc}
            1 & 0 \\
            0 & 1 \\
          \end{array}
        \right]+\mathcal{O}(\tau^{-1/2}) \phantom{\int^{\int}}\\
     \displaystyle \lim_{\zeta\to\infty} \zeta\left[M^{(3)}(\tau;\zeta)\right]_{12} =q^{(as)}(\tau)+ \mathcal{O}(\tau^{-1}) \phantom{\int_{\int}}\\
    
   \end{array}$
   \\ \hline
  Lemma \ref{l existence and asymptotic of db problem} & 
   $\begin{array}{l}
     \displaystyle D(\tau;0)=
       \left[
          \begin{array}{cc}
            1 & 0 \\
            0 & 1 \\
          \end{array}
        \right]+\mathcal{O}(\tau^{-3/4}) \phantom{\int^{\int}}\\
     \displaystyle \lim_{\zeta\to\infty} \zeta\left[D(\tau;\zeta)\right]_{12} =\mathcal{O}(\tau^{-3/4}) \phantom{\int_{\int}}
   \end{array}$
  \\
  \hline
\end{tabular}
\end{center}
Estimates (\ref{e M^0(0)}) and (\ref{e asymptotics q}) of Theorem \ref{t main thm general} follow easily by substituting these results into (\ref{e M^0 and q^0 solution formula}).

\section{Preparation for Steepest Descent}\label{s scalar RHP}
Scalar Riemann-Hilbert problems such as RHP \ref{rhp delta} are well understood and explicit representations for their solution are available in terms of the Cauchy operator, see (\ref{e def delta}) below. We will state important global properties of the solution $\delta$ of RHP \ref{rhp delta} in the following subsection and compute the asymptotic behavior at the stationary phase points $\pm 1$ in the subsequent subsection.
\subsection{The scalar Riemann-Hilbert problem}
We define
\begin{equation}\label{e def nu}
  \nu(\zeta):=\frac{1}{2\pi} \log\left(1+\rho(\zeta)\breve{\rho}(\zeta) \right)
\end{equation}
and consider the following function
\begin{equation}\label{e def delta}
  \delta(\zeta):=\exp\left\{\frac{1}{i} \int_{-1}^{1}\frac{\nu(s)}{s-\zeta}ds\right\},\quad \zeta\in\C\setminus [-1,1].
\end{equation}
The following can be found in many works, see for instance \cite{Deift1994}.
\begin{prop}\label{p properties delta}
    The function $\delta$ defined in (\ref{e def delta}) satisfies the following:
    \begin{itemize}
      \item[(i)] $\delta$ is a solution of Riemann-Hilbert problem \ref{rhp delta}.
      \item[(ii)] For $\mp \Imag(\zeta)>0$ we have $|\delta^{\pm 1}(\zeta)|\leq 1$.
      \item[(iii)] For $\zeta\notin [-1,1]$ we have $e^{-\|\nu\|_{L^{\infty}}/2}\leq  |\delta(\zeta)|\leq e^{\|\nu\|_{L^{\infty}}/2}$.
    \end{itemize}
\end{prop} 
\subsection{Asymptotics near the Stationary Phase Points}
In this subsection we address the asymptotic behavior of $\delta(\zeta)$ near the stationary phase points $-1$ and $+1$. Using the notation introduced in (\ref{e def nu}) we set
\begin{equation}\label{e def nu0}
    \nu_0^{\pm}:=\nu(\pm 1).
\end{equation}
Now, let $\zeta\in\C\setminus [-1,\infty)$ and use $-\pi<\arg(-(\zeta+1))<\pi$ to define $(-(\zeta+1))^{i\nu_0^{-}}$.
Then, for $\zeta\in [-1,\infty)$ we compute,
\begin{equation*}
    \lim_{\eps \downarrow 0}(-(\zeta+i\eps+1))^{i\nu_0^{-}}=
    \left(\lim_{\eps \downarrow 0}(-(\zeta-i\eps+ 1))^{i\nu_0^{-}}\right)(1+\rho(-1)\breve{\rho}(-1)),
\end{equation*}
and we learn that locally for $\zeta\to -1$ the two functions $(-(\,\cdot\,+1))^{i\nu_0^{-}}$ and $\delta(\,\cdot\,)$ as given in (\ref{e def delta}) satisfy the same jump condition. Analogously, let $\zeta\in\C\setminus (-\infty,1]$ and use $-\pi<\arg(\zeta-1)<\pi$ to define $(\zeta-1)^{i\nu_0^{+}}$. It then turns out that the function $(\,\cdot\,-1)^{i\nu_0^{+}}$ fulfills for $\zeta\to +1$ the same jump condition as the function $\delta$ defined in (\ref{e def delta}). In fact, we have the following asymptotics.
\begin{prop}\label{p delta near +-1}
    Let $\rho,\breve{\rho}\in H^{1}(\R)$ and $\delta$ given by (\ref{e def delta}). There exists a constants $\delta^{\pm}_0\in\C$ with $|\delta^{\pm}_0|=1$ and $c>0$ such that
    \begin{equation}\label{e asymptotic delta -1}
    \left|\delta(\zeta)-\delta_0^{-}\cdot (-(\zeta+1))^{i\nu_0^{-}}\right|\leq c \left(\|\rho\|_{H^{1}(\R)}+ \|\breve{\rho}\|_{H^{1}(\R)}\right)|\zeta+1|^{1/2}, \quad \text{for all }\zeta\in\C\setminus[-1,\infty),
    \end{equation}
    and
    \begin{equation}\label{e asymptotic delta +1}
    \left|\delta(\zeta)-\delta_0^{+}\cdot(\zeta- 1)^{-i\nu_0^{+}}\right|\leq c\left(\|\rho\|_{H^{1}(\R)}+ \|\breve{\rho}\|_{H^{1}(\R)}\right)|\zeta- 1|^{1/2}, \quad \text{for all }\zeta\in\C\setminus(-\infty,1].
    \end{equation}
    Here the complex powers of $\mp(\zeta\pm1)$ are defined with the branch of the logarithm as described above: $-\pi<\arg(-(\zeta+1))<\pi$ in (\ref{e asymptotic delta -1}) and $-\pi<\arg(\zeta-1)<\pi$ in (\ref{e asymptotic delta +1}), respectively.
\end{prop}
\begin{proof}
    In the following we give some standard computations. Let us define the following two functions,
    \begin{equation*}
        \varphi(s):=-s\nu_0^-\mathbb{1}_{[-1,0]}(s) +s\nu_0^+\mathbb{1}_{[0,1]}(s),\qquad \beta(\zeta):=\frac{1}{i} \int_{-1}^{1}\frac{\nu(s)-\varphi(s)}{s-\zeta}ds
    \end{equation*}
    and denote the exponent in (\ref{e def delta}) by $\gamma$ such that $\delta(\zeta)=\exp(\gamma(\zeta))$ and
    \begin{equation*}
        \gamma(\zeta) = \beta(\zeta)-\frac{\nu_0^-}{i} \int_{-1}^{0}\frac{s}{s-\zeta}ds+ \frac{\nu_0^+}{i} \int_{0}^{1}\frac{s}{s-\zeta}ds,\qquad \zeta\notin [-1,1].
    \end{equation*}
    Using the notation
    \begin{equation*}
        \mathfrak{b}^{(-)}(\zeta):=-\frac{\nu_0^-}{i} \int_{-1}^{0}\frac{s}{s-\zeta}ds,\qquad \mathfrak{b}^{(+)}(\zeta):=\frac{\nu_0^+}{i} \int_{0}^{1}\frac{s}{s-\zeta}ds,
    \end{equation*}
    we find $\gamma(\zeta)=\beta(\zeta)+\mathfrak{b}^{(-)}(\zeta) +\mathfrak{b}^{(+)}(\zeta)$. The functions $\mathfrak{b}^{(\pm)}$ are continuous around $\mp 1$. Indeed, we have
    \begin{equation}\label{e mathfrac b}
        |\mathfrak{b}^{(\pm)}(\zeta) -\mathfrak{b}^{(\pm)}(\mp1)|\leq c|\nu_0^{\pm}|\cdot |\zeta\pm1|,
    \end{equation}
    if $\zeta$ is closed to $\mp1$. Since the function $s\mapsto\nu(s)-\varphi(s)$ is in $H^1(\R)$ and has zeroes at $\pm1$ of at least order $1/2$, the values $\beta(\pm1)$ exist and, moreover, as shown in \cite[Lemma 3.4]{Cuccagna2014} we have
    \begin{equation}\label{e |beta(zeta)-beta(+-1)|}
        |\beta(\zeta)-\beta(\mp1)|\leq c\|\nu\|_{H^1(\R)} |\zeta\pm1|^{1/2}.
    \end{equation}
    Now, let $\zeta\in\C\setminus[-1,\infty)$ and choose $-\pi<\arg(-(\zeta+1))<\pi$. An explicit calculation of $\mathfrak{b}^{(-)}$ yields
    \begin{equation*}
        \mathfrak{b}^{(-)}(\zeta)=i\nu_0^-+i\nu_0^-\zeta \log(-\zeta)-i\nu_0^-(\zeta+1)\log(-(\zeta+1))+i\nu_0 \log(-(\zeta+1)),
    \end{equation*}
    and we learn that $|\mathfrak{b}^{(-)}(\zeta)-i\nu_0^--i\nu_0 \log(-(\zeta+1))|\leq c|\nu_0^-|\cdot |\zeta+1|$ if $\zeta$ is closed to $-1$. Making use of (\ref{e mathfrac b}) and (\ref{e |beta(zeta)-beta(+-1)|}), we easily obtain
    \begin{equation*}
        |\gamma(\zeta)-\beta(-1)-\mathfrak{b}^{(+)}(-1)-i\nu_0^--i\nu_0 \log(-(\zeta+1))|  \leq c\|\nu\|_{H^1(\R)}|\zeta+1|^{1/2}.
    \end{equation*}
    This is turn implies that (\ref{e asymptotic delta -1}) holds for
    \begin{equation}\label{e def delta0-}
        \delta_0^-=\exp\{\beta(-1)+\mathfrak{b}^{(+)}(-1) +i\nu_0^-\}.
    \end{equation}
    Using very similar computations estimates around $+1$ can be derived and it follows that
    (\ref{e asymptotic delta +1}) holds for
    \begin{equation}\label{e def delta0+}
        \delta_0^+=\exp\{\beta(1)+\mathfrak{b}^{(-)}(1) -i\nu_0^+\}.
    \end{equation}
    The property $|\delta^{\pm}_0|=1$ is obvious and thus the Proposition is proved.
\end{proof}  
\section{$\overline{\partial}$-extensions of jump factorization}\label{s dbextension}
In this section we will specify the transformation $M^{(1)}\to  M^{(2)}$, given explicitly by (\ref{e def M^2}). As explained in the summary the purpose of this deformation is  to remove the jump along the real axis and introduce jumps on $\Sigma^{(2)}$, where $\Sigma^{(2)}=\Sigma^{(2)}_1\cup...\cup\Sigma^{(2)}_8$ is depicted in Figure \ref{f Sigma2}. Therefore we need to construct the function $\mathcal{W}$ piecewise on $\Omega_1,...,\Omega_{10}$, where we denote by $\Omega_j$ the components of $\C\setminus(\R\cup\Sigma^{(2)})$ as sketched in Figure \ref{f Omega1-8}. 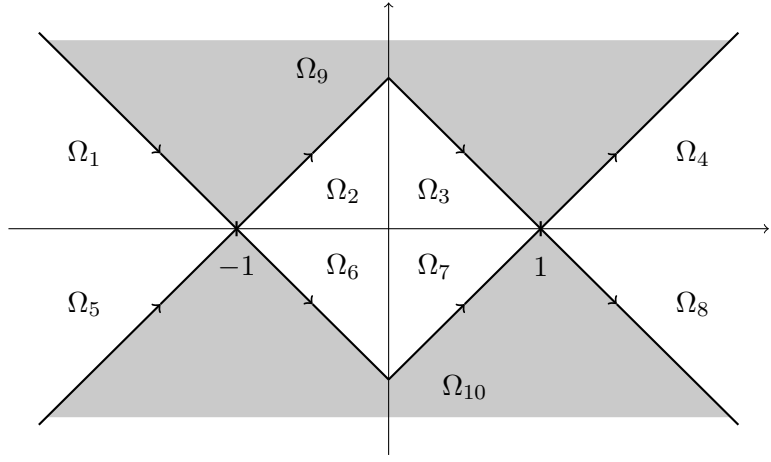
\begin{figure}
\begin{center}
\begin{tikzpicture}
\draw 	[->, thick]  	(-4.6,2.6) -- (-3,1) ;
\draw 	[->, thick]  	(-2,0) -- (-1,1) ;
\draw 	[->, thick]  	(-4.6,-2.6) -- (-3,-1) ;
\draw 	[->, thick]  	(-2,0) -- (-1,-1) ; ;
\draw 	[->, thick]  	(0,2) -- (1,1) ;
\draw 	[->, thick]  	(2,0) -- (3,1) ;
\draw 	[->, thick]  	(0,-2) -- (1,-1) ;
\draw 	[->, thick]  	(2,0) -- (3,-1) ;
\draw 	[->]  	(-5,0) -- (5,0) ;
\draw 	[->]  	(0,-3) -- (0,3) ;
\draw 	[thick]  	(-3,1) -- (-2,0) ;
\draw 	[thick]  	(-1,1) -- (0,2) ;
\draw 	[thick]  	(-3,-1) -- (-2,0) ;
\draw 	[thick]  	(-1,-1) -- (0,-2) ;
\draw 	[thick]  	(1,1) -- (2,0) ;
\draw 	[thick]  	(3,1) -- (4.6,2.6) ;
\draw 	[thick]  	(1,-1) -- (2,0) ;
\draw 	[thick]  	(3,-1) -- (4.6,-2.6) ;
\draw 	[thick]  	(-2,.1) -- (-2,-.1) ;
\draw 	[thick]  	(2,.1) -- (2,-.1) ;

\node at (-2,-.5) {$-1$};
\node at (2,-.5) {$1$};
\node at (-4,1) {$\Omega_1$};
\node at (-.6,.5) {$\Omega_2$};
\node at (.6,.5) {$\Omega_3$};
\node at (4,1) {$\Omega_4$};
\node at (-4,-1) {$\Omega_5$};
\node at (-.6,-.5) {$\Omega_6$};
\node at (.6,-.5) {$\Omega_7$};
\node at (4,-1) {$\Omega_8$};
\node at (-1,2.1) {$\Omega_9$};
\node at (1,-2.1) {$\Omega_{10}$};
\path[fill=black,opacity=0.21] (-4.5,2.5) -- (-2,0) -- (0,2) -- (2,0) -- (4.5,2.5) -- (-4.5,2.5);
\path[fill=black,opacity=0.21] (-4.5,-2.5) -- (-2,0) -- (0,-2) -- (2,0) -- (4.5,-2.5) -- (-4.5,-2.5);
\end{tikzpicture}
\end{center}
  \caption{Decomposition of $\C\setminus(\R\cup\Sigma^{(2)})$ into ten connected components.}\label{f Omega1-8}
\end{figure} 
Assume that functions $R_k:\overline{\Omega}_k\to\C$, $k=1,...,8$ satisfy the following boundary conditions:
\begin{subequations}\label{e def R1...R8}
\begin{align}
   R_1(\zeta)&=
   \left\{
   \begin{aligned}
    &\breve{\rho}(\zeta)\delta^{-2}(\zeta),\\
    &\breve{\rho}(-1)(\delta_0^{-})^{-2} (-(\zeta+1))^{-2i\nu_0^{-}},
   \end{aligned}
   \right.&&
   \begin{aligned}
    &\zeta\in(-\infty,-1) ,\\
    & \zeta\in\Sigma_1^{(2)},
   \end{aligned}
   \label{e def R1}
   \\
   R_2(\zeta)&=
   \left\{
   \begin{aligned}
    &\frac{\rho(\zeta)}{1+\rho(\zeta)\breve{\rho}(\zeta)} \delta_+^{2}(\zeta),\\
    &\frac{\rho(-1)}{1+\rho(-1)\breve{\rho}(-1)} (\delta_0^{-})^{2}(-(\zeta+1))^{2i\nu_0^{-}}(1-\chi(\zeta)),
   \end{aligned}
   \right.&&
   \begin{aligned}
    &\zeta\in(-1,0) ,\\
    & \zeta\in\Sigma_2^{(2)},
   \end{aligned}
   \label{e def R2}
   \\
   R_3(\zeta)&=
   \left\{
   \begin{aligned}
    &\frac{\rho(\zeta)}{1+\rho(\zeta)\breve{\rho}(\zeta)}  \delta_+^{2}(\zeta),\\
    &\frac{\rho(1)}{1+\rho(1)\breve{\rho}(1)} (\delta_0^{+})^{2}(\zeta-1)^{-2i\nu_0^{+}}(1-\chi(\zeta)),
   \end{aligned}
   \right.&&
   \begin{aligned}
    &\zeta\in(0,1) ,\\
    & \zeta\in\Sigma_3^{(2)},
   \end{aligned}
   \label{e def R3}
   \\
   R_4(\zeta)&=
   \left\{
   \begin{aligned}
    &\breve{\rho}(\zeta)\delta^{-2}(\zeta),\\
    & \breve{\rho}(1)(\delta_0^{+})^{-2}(\zeta-1)^{2i\nu_0^{+}},
   \end{aligned}
   \right.&&
   \begin{aligned}
    &\zeta\in(1,\infty) ,\\
    & \zeta\in\Sigma_4^{(2)},
   \end{aligned}
   \label{e def R4}
   \\
   R_5(\zeta)&=
   \left\{
   \begin{aligned}
    &\rho(\zeta)\delta^2(\zeta),\\
    &\rho(-1)(\delta_0^{-})^{2}(-(\zeta+ 1))^{2i\nu_0^{-}},
   \end{aligned}
   \right.&&
   \begin{aligned}
    &\zeta\in(-\infty,-1) ,\\
    &\zeta\in\Sigma_5^{(2)},
   \end{aligned}
   \label{e def R5}
   \\
   R_6(\zeta)&=
   \left\{
   \begin{aligned}
    &\frac{\breve{\rho}(\zeta)} {1+\rho(\zeta)\breve{\rho}(\zeta)}\delta_-^{-2}(\zeta),\\
    &\frac{\breve{\rho}(-1)}{1+\rho(-1)\breve{\rho}(-1)} (\delta_0^{-})^{-2}(-(\zeta+1))^{-2i\nu_0^{-}}(1-\chi(\zeta)),
   \end{aligned}
   \right.&&
   \begin{aligned}
    &\zeta\in(-1,0) ,\\
    & \zeta\in\Sigma_6^{(2)},
   \end{aligned}
   \label{e def R6}
   \\
   R_7(\zeta)&=
   \left\{
   \begin{aligned}
    &\frac{\breve{\rho}(\zeta)} {1+\rho(\zeta)\breve{\rho}(\zeta)}\delta_-^{-2}(\zeta),\\
    & \frac{\breve{\rho}(1)}{1+\rho(1)\breve{\rho}(1)} (\delta_0^{+})^{-2}(\zeta-1)^{2i\nu_0^{+}}(1-\chi(\zeta)),
   \end{aligned}
   \right.&&
   \begin{aligned}
    &\zeta\in(0,1) ,\\
    & \zeta\in\Sigma_7^{(2)},
   \end{aligned}
   \label{e def R7}
   \\
   R_8(\zeta)&=
   \left\{
   \begin{aligned}
    &\rho(\zeta)\delta^{
         2}(\zeta),\\
    &\rho(1)(\delta_0^{+})^{2}(\zeta-1)^{-2i\nu_0^{+}},
   \end{aligned}
   \right.&&
   \begin{aligned}
    &\zeta\in(1,\infty) ,\\
    & \zeta\in\Sigma_8^{(2)}.
   \end{aligned}
   \label{e def R8}
\end{align}
\end{subequations}
The function $\chi:\C\to [0,1]$ is a smooth cutoff function supported on a neighborhood of $\pm i$ such that
\begin{equation}\label{e chi}
    \chi(\zeta)=
    \left\{
      \begin{array}{ll}
        1, & \text{ if }|\zeta+ i|<3/5\text{ or }|\zeta-i|<3/5, \\
        0, & \text{ if }|\zeta+ i|>4/5\text{ and }|\zeta-i|>4/5.
      \end{array}
    \right.
\end{equation}
We set
\begin{equation}\label{e def W}
    \mathcal{W}(\tau;\zeta):=
    \left\{
      \begin{array}{ll}    \vspace{1mm}
        \left[
          \begin{array}{cc}
            1 & -R_k(\zeta)e^{i\tau Z(\zeta)} \\
            0 & 1 \\
          \end{array}
        \right]
        , & \zeta\in\Omega_k,\quad k\in\{1,4\}\\ \vspace{1mm}
        \left[
          \begin{array}{cc}
            1 & 0 \\
            -R_k(\zeta)e^{-i\tau Z(\zeta)} & 1 \\
          \end{array}
        \right]
        , &  \zeta\in\Omega_k,\quad k\in\{2,3\},\\ \vspace{1mm}
        \left[
          \begin{array}{cc}
            1 & 0 \\
            R_k(\zeta)e^{i\tau Z(\zeta)} & 1 \\
          \end{array}
        \right], & \zeta\in\Omega_k,\quad k\in\{5,8\}, \\ \vspace{1mm}
        \left[
          \begin{array}{cc}
            1 & R_k(\zeta)e^{i\tau Z(\zeta)} \\
            0 & 1 \\
          \end{array}
        \right], &\zeta\in\Omega_k,\quad k\in\{6,7\}, \\
        \left[
          \begin{array}{cc}
            1 & 0 \\
            0 & 1 \\
          \end{array}
        \right]
        ,&\zeta\in\Omega_9\cup\Omega_{10},
      \end{array}
    \right.
\end{equation}
and by (\ref{e def R1})--(\ref{e def R8}) (first line in each case) we find that $\mathcal{W}$ satisfies (\ref{e condition1 W}). Moreover it follows from (\ref{e def R1})--(\ref{e def R8}) (second line in each case) that $M^{(2)}$ admits jumps on $\Sigma^{(2)}$ of the form $M^{(2)}_+=M^{(2)}_-(1+R_{\tau}^{(2)})$ with $R_{\tau}^{(2)}$ precisely given by (\ref{e def R^2-}) and (\ref{e def R^2+}).\\
The lack of analyticity is measured by means of the following differential operator:
\begin{equation*}
    \db:=\frac12\left(\frac{\partial}{\partial x}+i\frac{\partial}{\partial y}\right),\quad(\zeta=x+iy).
\end{equation*}
We would have $\db\mathcal{W}=0$ if $\mathcal{W}$ was analytic. The general case is
\begin{equation}\label{e del W}
    \db\mathcal{W}(\tau;\zeta):=
    \left\{
      \begin{array}{ll}    \vspace{1mm}
        \left[
          \begin{array}{cc}
            0 & -\db R_k(\zeta)e^{i\tau Z(\zeta)} \\
            0 & 0 \\
          \end{array}
        \right]
        , & \zeta\in\Omega_k,\quad k\in\{1,4\}\\ \vspace{1mm}
        \left[
          \begin{array}{cc}
            0 & 0 \\
            -\db R_k(\zeta)e^{-i\tau Z(\zeta)} & 0 \\
          \end{array}
        \right]
        , &  \zeta\in\Omega_k,\quad k\in\{2,3\},\\ \vspace{1mm}
        \left[
          \begin{array}{cc}
            0 & 0 \\
            \db R_k(\zeta)e^{i\tau Z(\zeta)} & 0 \\
          \end{array}
        \right], & \zeta\in\Omega_k,\quad k\in\{5,8\}, \\ \vspace{1mm}
        \left[
          \begin{array}{cc}
            0 & \db R_k(\zeta)e^{i\tau Z(\zeta)} \\
            0 & 0 \\
          \end{array}
        \right], &\zeta\in\Omega_k,\quad k\in\{6,7\}, \\
        \left[
          \begin{array}{cc}
            0 & 0 \\
            0 & 0 \\
          \end{array}
        \right]
        ,&\zeta\in\Omega_9\cup\Omega_{10}.
      \end{array}
    \right.
\end{equation}
Similar to \cite{LiuPerrySulem2017,Jenkins2014,Cuccagna2014,Dieng2008} by explicit construction the following Lemma can be obtained.
\begin{lem}\label{l key estimate db R}
    Let $\rho,\breve{\rho}$ satisfy assumptions (\ref{e assumption1})--(\ref{e assumption4}). Then for $j=1,...,8$ there exist functions $R_k:\overline{\Omega}_k\to\C$ satisfying (\ref{e def R1})--(\ref{e def R8}), so that for $k\in\{1,2,5,6\}$ and $\zeta\in\Omega_k$
    \begin{equation}\label{e key estimate db R 1256}
       |\db R_k(\zeta)|\leq c_1(\|\rho\|_{H^1(\R)}+ \|\breve{\rho}\|_{H^1(\R)}) |\zeta+1|^{-1/2}+c_2 (|\rho'(\Real(\zeta))|+ |\breve{\rho}'(\Real(\zeta))|)+ \Gamma_0|\db \chi(\zeta)|,
    \end{equation}
    whereas for $k\in\{3,4,7,8\}$ and  $\zeta\in\Omega_k$,
    \begin{equation}\label{e key estimate db R 3478}
       |\db R_k(\zeta)|\leq c_1(\|\rho\|_{H^1(\R)}+ \|\breve{\rho}\|_{H^1(\R)}) |\zeta-1|^{-1/2}+c_2 (|\rho'(\Real(\zeta))|+ |\breve{\rho}'(\Real(\zeta))|)+\Gamma_0|\db \chi(\zeta)|.
    \end{equation}
\end{lem}
\begin{proof}
    The functions $R_k$ can be defined explicitly. Let us exemplarily consider the case $k=2$. For $\zeta\in\Omega_2$ we use $0<\arg(\zeta+1)<\pi/4$ and set $G(\zeta):= g(\arg(\zeta+1))$, where the function $g:[0,\pi/4]\to[0,1]$ is a smooth function such that $g(\varphi)=1$ for all $\varphi\in [0,\pi/6]$ and $g(\pi/4)=0$, see Figure \ref{f g}. \begin{figure}
\begin{center}
\begin{tikzpicture}
\draw 	[->]  	(-0.2,0) -- (7,0) ;
\draw 	[->]  	(0,-0.2) -- (0,2.5) ;
\node at (4,-0.5) {$\frac{\pi}{6}$};
\node at (6,-0.5) {$\frac{\pi}{4}$};
\node at (-0.4,2) {$1$};
\node at (7.2,-0.2) {$\varphi$};
\node at (5.4,1.8) {$g(\varphi)$};
\draw[thick] (4,2).. controls (5,2) and (5,1.5)..(6,0);
\draw 	[thick]  	(0,2) -- (4,2);
\draw  	(0,2) -- (-0.1,2);
\draw  	(4,0) -- (4,-0.1);
\draw  	(6,0) -- (6,-0.1);
\end{tikzpicture}
\end{center}
  \caption{The function $g$ used in the construction of the functions $R_k$ in the proof of Lemma \ref{l key estimate db R}} \label{f g}
\end{figure}
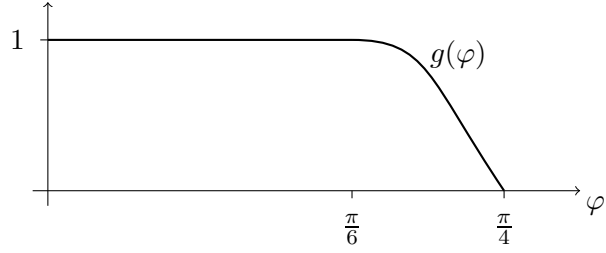
    The function $G$ is continuous in $\overline{\Omega}_2\setminus\{-1\}$ and fulfills
    \begin{equation*}
        G(\zeta)=
        \left\{
          \begin{array}{ll}
            1, & \hbox{for }\zeta\in (-1,0), \\
            0, & \hbox{for }\zeta\in \Sigma_2^{(2)}.
          \end{array}
        \right.
    \end{equation*}
    Moreover, we can find a constant $c$ such that
    \begin{equation}\label{e db G}
        |\db G(\zeta)|\leq c|\zeta+1|^{-1}
    \end{equation}
    for all $\zeta\in \Omega_2$. Now let us use the function just introduced for an interpolation between the first and second line of (\ref{e def R2}):
    \begin{equation}\label{e construction of R2}
        \begin{aligned}
            R_2(\zeta)&:= G(\zeta) \frac{\rho(\Real(\zeta))}{1+\rho(\Real(\zeta)) \breve{\rho}(\Real(\zeta))} \delta^{2}(\zeta) (1-\chi(\zeta))\\
            &\qquad + \left(1-G(\zeta)\right)
            \frac{\rho(-1)}{1+\rho(-1)\breve{\rho}(-1)} (\delta_0^{-})^{2}(-(\zeta+1)) ^{2i\nu_0^{-}}(1-\chi(\zeta))
        \end{aligned}
    \end{equation}
    Immediately from this definition and from the definitions of $G$ and $\chi$, (\ref{e chi}), we can verify (\ref{e def R2}). Thus it remains to prove the desired bound (\ref{e key estimate db R 1256}). A direct computation yields $|\db R_2(\zeta)|\leq W_1+W_2+W_3$ with
    \begin{equation*}
    \begin{aligned}
        W_1&=C_1 |\db G(\zeta)| \left|\frac{\rho(\Real(\zeta))}{1+\rho(\Real(\zeta)) \breve{\rho}(\Real(\zeta))} \delta^{2}(\zeta)- \frac{\rho(-1)}{1+\rho(-1)\breve{\rho}(-1)} (\delta_0^{-})^{2}(-(\zeta+1)) ^{2i\nu_0^{-}} \right|,\\
        W_1&= C_2 \left|\db\left( \frac{\rho(\Real(\zeta))}{1+\rho(\Real(\zeta)) \breve{\rho}(\Real(\zeta))} \right)\right|, \\
        W_3&=C_3 \Gamma_0 |\db \chi(\zeta)|,
    \end{aligned}
    \end{equation*}
    where the constants $C_1,C_2,C_3$ are independent of $\rho$ and $\breve{\rho}$ by assumption. Using $|f(a)- f(b)|\leq c \|f\|_{H^1(\R)}|a-b|^{1/2}$ for all functions $f\in H^{1}(\R)$, Proposition \ref{p delta near +-1} and (\ref{e db G}) we find
    \begin{equation*}
        W_1\leq c_1 (\|\rho\|_{H^1(\R)}+ \|\breve{\rho}\|_{H^1(\R)})|\zeta+1|^{-1/2}.
    \end{equation*}
    Differentiating yields
    \begin{equation*}
        W_2\leq c_2 (|\rho'(\Real(\zeta))|+ |\breve{\rho}'(\Real(\zeta))|),
    \end{equation*}
    where $c_2$ depends on the constant $\Gamma_0$ in the assumption (\ref{e assumption2}). This completes the proof of (\ref{e key estimate db R 1256}) for $k=2$. The other $R_k$'s can be defined similar to (\ref{e construction of R2}). The idea in each case is to use $\arg(\zeta\pm 1)$ and the function $g$ to interpolate between the first and second line of (\ref{e def R1})--(\ref{e def R8}). For technical reasons it is necessary to include the function $\chi(\zeta)$, see Remark \ref{r sense of chi} below.
\end{proof}
\begin{remark}\label{r sense of chi}
    Since $\mathcal{W}$ is defined piecewise on $\Omega_1,...,\Omega_{10}$, one would expect a discontinuity on the line $[-i,i]$. But the reader might check that the presence of the cutoff function $\chi$, introduced in (\ref{e chi}), and the fact that the function $g$ is constant on $[0,\pi/6]$, as depicted in Figure \ref{f g}, avoid that kind of discontinuity for $\mathcal{W}$.
\end{remark}
As explained in the summary, the definition of $R^{(2)}_{\tau}$ in (\ref{e def R^2-}) and (\ref{e def R^2+}) involve that for $\tau\to\infty$, $R^{(2)}_{\tau}(\zeta)\to 0$ point-wise for all $\zeta\in\Sigma^2\setminus\{\pm 1\}$, see also (\ref{e sign of i Z}) and Figure \ref{f signaturetable}. However, we do not have $\lim_{\tau\to\infty} \|R^{(2)}_{\tau}\|_{L^{\infty}(\Sigma^{(2)})}=0$. But we will need the following.
\begin{prop}\label{p L^1norm of R^2}
    Under the assumption (\ref{e assumption1})--(\ref{e assumption4}) there exists a constant $c>0$ such that
    \begin{equation*}
        \|R^{(2)}_{\tau}\|_{L^{1}(\Sigma^{(2)})}\leq  c\,\Gamma_0\,\tau^{-1/2}
    \end{equation*}
    for all $\tau\in\R^+$. Here, the constant $\Gamma_0$ is defined in (\ref{e assumption3}).
\end{prop}
\begin{proof}
    We start with calculating the $L^1$-norm on  $\Sigma_3^{(2)}\subset\Sigma^{(2)}$. For that we use the parametrization $\zeta=a+(1-a)i$ with $0\leq a\leq 1$. A simple computation shows that $\Imag(Z(a+(1-a)i))\leq-a(a-1)^2\leq0$. We immediately see that for
    \begin{equation}\label{e def help fct I}
        I(a):=
        \left\{
          \begin{array}{ll}
            -\frac{1}{4}a^2, & 0\leq a \leq \frac{1}{2},\phantom{\int_{\int_{\int}}} \\
            -\frac{1}{4}(a-1)^2,&\frac{1}{2}\leq a \leq 1,
          \end{array}
        \right.
    \end{equation}
    we have $\Imag(Z(a+(1-a)i))\leq I(a)\leq0$.
    Furthermore, by assumption (\ref{e assumption3}) and, since $(\zeta -1)^{-2i\nu_0^{+}}$ is bounded on $\Sigma_3^{(2)}$, we find a constant $c_1$ such that
    \begin{equation*}
         \|R^{(2)}_{\tau}\|_{L^{1}(\Sigma_3^{(2)})}\leq c_1\,\Gamma_0 \int_{\Sigma_3^{(2)}} \left|e^{-i\tau Z(\zeta)}\right|d\zeta\leq c_1\,\Gamma_0\sqrt{2}\int_0^1 e^{\tau I(a)}da=2c_1\,\Gamma_0\sqrt{2}\int_0^{\frac12} e^{-\tau a^2/4}da\leq \frac{c\,\Gamma_0}{\tau^{1/2}}.
    \end{equation*}
    The same argument can be used to estimate $\|R^{(2)}_{\tau}\|_{L^{2}(\Sigma_j^{(2)})}$ for $j\in\{2,6,7\}$. Now we consider the ray $\Sigma_4^{(2)}$ and use the parametrization $\zeta=1+a(1+i)$. Having a look at the imaginary part of $Z(1+a(1+i))$ we discover that it is possible to identify positive constants $a_0,b_1$ and $b_2$ such that
    \begin{equation}\label{e estimate I check}
        0\leq \check{I}(a)\leq\Imag(Z(1+a(1+i))),
        \qquad\text{for}\qquad
        \check{I}(a):=
        \left\{
          \begin{array}{ll}
            b_1 a^2, & 0\leq a \leq  a_0,\phantom{\int_{\int_{\int}}} \\
            b_2 a,&a_0\leq a.
          \end{array}
        \right.
    \end{equation}
    Analogously to the above estimates we now find
    \begin{equation*}
         \|R^{(2)}_{\tau}\|_{L^{2}(\Sigma_4^{(2)})}\leq c_1\,\Gamma_0\int_0^{\infty} e^{-\tau I(a)}da=c_1\,\Gamma_0\left[\int_0^{a_0} e^{-b_1\tau a^2}da+\int_{a_0}^{\infty} e^{-b_2\tau a}da\right]\leq c\,\Gamma_0(\tau^{-1/2}+\tau^{-1}).
    \end{equation*}
    The same can be done for the remaining rays $\Sigma^{(2)}_1$, $\Sigma^{(2)}_5$ and $\Sigma^{(2)}_8$. Thus, the proof of the Proposition is done.
\end{proof}
The reason for splitting the notation in (\ref{e def R^2-}) and (\ref{e def R^2+}) into two contributions will become clear later in Proposition \ref{p M^4 approx M^5}. Therein we will show that as $\tau\to \infty$ the interaction between the two crosses
\begin{equation}\label{e def Sigma 4-}
    \Sigma^{(4-)}:=\Sigma^{(2)}_1\cup \Sigma^{(2)}_2\cup \Sigma^{(2)}_5\cup \Sigma^{(2)}_6
\end{equation}
and
\begin{equation}\label{e def Sigma 4+}
    \Sigma^{(4+)}:=\Sigma^{(2)}_3\cup \Sigma^{(2)}_4\cup \Sigma^{(2)}_7\cup \Sigma^{(2)}_8
\end{equation}
with regard to the reconstruction formula (\ref{e rec u 2}) is negligible. See Figure \ref{f Sigma4} for an illustration of the contours $\Sigma^{(4-)}$ and $\Sigma^{(4+)}$. \\

\section{Analysis of pure RHP}\label{s pure RHP}
The RHP contribution (\ref{e jump M^2}) is mainly responsible for the long-time asymptotics of $q(\tau)$ stated in Theorem \ref{t main thm general}. We can provide this explicit result since the function $q^{(3)}(\tau)$ associated to RHP($\Sigma^{(3)}$,$R_{\tau}^{(3)}$) converges to the sum $q^{(5-)}(\tau)+q^{(5+)}(\tau)$, where both $q^{(5\pm)}$ are associated to model RHP's which can be solved explicitly.
\subsection{Two model RHPs}\label{s two model RHPs}
The Riemann-Hilbert problems RHP($\Sigma^{(5\pm)}$,$R_{\tau}^{(5\pm)}$) with $\Sigma^{(5\pm)}$ depicted in Figure \ref{f Sigma5} and $R_{\tau}^{(5\pm)}$ given in (\ref{e def R^5-}) and (\ref{e def R^5+}) are explicitly solvable. Usually the solution procedure is presented in the following way. We begin with defining a change of variables
\begin{equation}\label{e scaling - case}
  \eta(\zeta)=-\sqrt{\tau}(\zeta+1),\quad
  \zeta(\eta)=\frac{-1}{\sqrt{\tau}}\eta-1,
\end{equation}
and set
\begin{equation}\label{e M^PC- <-> M^5-}
    M^{(PC-)}(\tau;\eta)=M^{(5-)}(\tau;\zeta(\eta)), \quad \eta\in\C\setminus  \Sigma^{(PC)},
\end{equation}
with the new contour
\begin{figure}
\centering
\vskip 15pt
\begin{tikzpicture}[scale=0.9]
%
%

\draw 	[->, thick]  	(4,4) -- (5,5) ;								 
\draw		[thick] 		(5,5) -- (6,6) ;
\draw		[->, thick] 	(2,6) -- (3,5) ;								 
\draw		[thick]		(3,5) -- (4,4);	
\draw		[->, thick]	(2,2) -- (3,3);								 
\draw		[thick]		(3,3) -- (4,4);
\draw		[->,thick]	(4,4) -- (5,3);								 
\draw		[thick]  		(5,3) -- (6,2);
\draw 	[->]  	(1,4) -- (7,4) ;
\draw 	[->]  	(4,1) -- (4,7) ;
\node at (6.7,6.3)  	{$\Sigma_2^{(PC)}$};
\node at (1.5,6.3) 	{$\Sigma_1^{(PC)}$};
\node at (1.5,2.3)	{$\Sigma_3^{(PC)}$};
\node at (6.7,2.3) 	{$\Sigma_4^{(PC)}$};

\end{tikzpicture}
\caption{The contour $\Sigma^{(PC)}$ of the model Riemann-Hilbert problems.}
\label{f SigmaPC}
\end{figure}
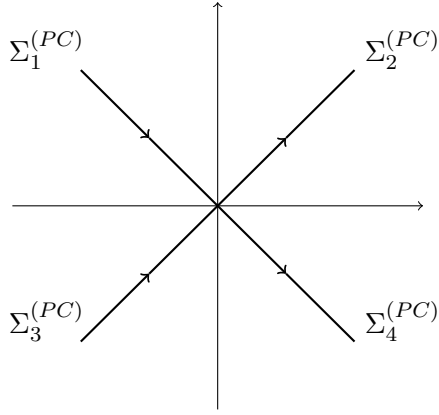
\begin{equation}\label{e def Sigma^PC}
    \Sigma^{(PC)}:=(e^{i\pi/4}\R)\cup( e^{-i\pi/4}\R).
\end{equation}
As depicted in Figure \ref{f SigmaPC} we let $\Sigma^{(PC)}$ inherit the orientation of $\R$. Thus, it is important to notice that $\Sigma^{(PC)}$ is not simply the image of $\Sigma^{(5-)}$ under the transformation $\zeta\to\eta$ defined in (\ref{e scaling - case}). This is because the change of variables in (\ref{e scaling - case}) would also rotate the contour by an angle of $\pi$ and thus reverse the orientation. As an important consequence we have $M^{(PC-)}_{\pm}(\tau;\eta)=M_{\mp}^{(5-)}(\tau;\zeta(\eta))$ for $\eta\in\Sigma^{(PC)}$. It follows that the condition
\begin{equation*}
    M_+^{(5-)}(\tau;\zeta)=M_-^{(5-)}(\tau;\zeta) \left(1+R^{(5-)}_{\tau}(\zeta)\right), \qquad \zeta\in \Sigma^{(5-)},
\end{equation*}
is transformed into the jump condition
\begin{equation*}
    \begin{aligned}
        M^{(PC-)}_+(\tau;\eta)&=M^{(PC-)}_- (\tau;\eta)\left(1+R^{(5-)}_{\tau} (\zeta(\eta))\right)^{-1}\\
        &=M^{(PC-)}_- (\tau;\eta)\left(1-R^{(5-)}_{\tau} (\zeta(\eta))\right), \qquad \eta\in \Sigma^{(PC)}.
    \end{aligned}
\end{equation*}
This gives rise to the definition
\begin{equation}\label{e def R^PC-}
    R_{\tau}^{(PC-)}(\eta):=-R_{\tau}^{(5-)}(\zeta(\eta)) \qquad \eta\in \Sigma^{(PC)}.
\end{equation}

With a view to (\ref{e def R^5-}) it will be useful to have the following identities
\begin{equation*}
    e^{i\tau \left(1+\frac{1}{2}(\zeta+1)^2\right)}= e^{i\eta^2/2}e^{i\tau},\qquad
    (-(\zeta+1))^{2i\nu_0^{-}}=\eta^{2i\nu_0^{-}} e^{-i\nu_0^{-}\ln(\tau)}.
\end{equation*}
Note that we define $-\pi<\arg(\eta)<\pi$ so that the branch cut of $\eta^{2i\nu_0^{-}}$ is given by $\R^-$ whereas $(-(\zeta+ 1))^{2i\nu_0^{-}}$ is cut along $(-1,\infty)$. Define new parameters
\begin{equation}\label{e def r0-}
    \begin{aligned}
        &\rho_0^-:=\rho(-1)(\delta_0^-)^2e^{i\tau }e^{-i\nu_0^{-}\ln(\tau)},\\
        &\breve{\rho}_0^-:=\breve{\rho}(-1)(\delta_0^-)^{-2} e^{-i\tau }e^{i\nu_0^{-}\ln(\tau)},
    \end{aligned}
\end{equation}
such that $\rho_0^-\breve{\rho}_0^-=\rho(-1)\breve{\rho}(-1)$ and $\nu_0^-=\frac{1}{2\pi}\log(1+\rho_0^-\breve{\rho}_0^-)$.
Using the notation just introduced we find for $\eta\in \Sigma^{(PC)}$
\begin{equation}\label{e formula R^PC-}
    R_{\tau}^{(PC-)}(\eta)=
    \left\{
      \begin{array}{ll}    \vspace{1mm}
        \left[
          \begin{array}{cc}
             0& \frac{-\breve{\rho}_0^-} {1+\rho_0^-\breve{\rho}_0^-} \eta^{-2i\nu_0^{-}}e^{-i\eta^2/2}  \\
            0 & 0 \\
          \end{array}
        \right]
        , & \eta\in \Sigma^{(PC)}_1:=e^{-i\pi/4}\R_-\\ \vspace{1mm}
        \left[
          \begin{array}{cc}
            0 & 0 \\
            -\rho_0^-\eta^{2i\nu_0^{-}}e^{i\eta^2/2} & 0 \\
          \end{array}
        \right]
        , &  \eta\in\Sigma^{(PC)}_2:= e^{i\pi/4}\R_+\\ \vspace{1mm}
        \left[
          \begin{array}{cc}
            0 & 0 \\
            \frac{-\rho_0^-}{1+\rho_0^-\breve{\rho}_0^-} \eta^{2i\nu_0^{-}}e^{i\eta^2/2} & 0 \\
          \end{array}
        \right], & \eta\in\Sigma^{(PC)}_3 :=e^{i\pi/4}\R_-,\\ \vspace{1mm}
        \left[
          \begin{array}{cc}
             0& -\breve{\rho}_0^- \eta^{-2i\nu_0^{-}}e^{-i\eta^2/2} \\
            0 & 0 \\
          \end{array}
        \right], &\eta\in\Sigma^{(PC)}_4 :=e^{-i\pi/4}\R_+. \\
      \end{array}
    \right.
\end{equation}
We may write $R^{(PC-)}_{\tau}(\eta)$ in the form
\begin{equation}\label{e R^5- alternative}
    R_{\tau}^{(PC-)}(\eta)=S(\eta)R_0^-[S(\eta)]^{-1},\qquad S(\eta):= \eta^{-i\nu_0^-\sigma_3} e^{-\frac{i}{4}\eta^2\sigma_3},
\end{equation}
where $R_0^-$ is shown in Figure \ref{f R0-}\begin{figure}
\centering
\vskip 15pt
\begin{tikzpicture}[scale=0.9]
%
%
\draw [fill] (4,4) circle [radius=0.075];						 
\node at (4.0,3.65) {$0$};										 
%
%
\draw 	[->, thick]  	(4,4) -- (5,5) ;								 
\draw		[thick] 		(5,5) -- (7,7) ;
\draw		[->, thick] 	(1,7) -- (3,5) ;								 
\draw		[thick]		(3,5) -- (4,4);	
\draw		[->, thick]	(1,1) -- (3,3);								 
\draw		[thick]		(3,3) -- (4,4);
\draw		[->,thick]	(4,4) -- (5,3);								 
\draw		[thick]  		(5,3) -- (7,1);
%
%
\draw [  thick, red, decorate, decoration={snake,amplitude=0.5mm}] (0,4)  -- (4,4);				 
\node at (6.5,4) {$-\pi<\arg(\eta)<\pi$};
%
%
\node at (7,6)  {$\Sigma_2^{(PC)}$};
\node at (3,6) 	{$\Sigma_1^{(PC)}$};
\node at (3,2)	{$\Sigma_3^{(PC)}$};
\node at (7,2) 	{$\Sigma_4^{(PC)}$};
%
%
\node at (8,8) {${
\left[
  \begin{array}{cc}
    0 & 0 \\
            -\rho_0^-&0\\
  \end{array}
\right]
}$};						
\node at (0,8) {${
\left[
  \begin{array}{cc}
    0& \frac{-\breve{\rho}_0^-} {1+\rho_0^-\breve{\rho}_0^-}\\
            0 & 0 \\
  \end{array}
\right]
}$};						
\node at (0,0) {${
\left[
  \begin{array}{cc}
    0 & 0 \\
    \frac{-\rho_0^-}{1+\rho_0^-\breve{\rho}_0^-} & 0 \\
  \end{array}
\right]
}$};						
\node at (8,0) {${
\left[
  \begin{array}{cc}
    0 & -\breve{\rho}_0^-  \\
    0 & 0 \\
  \end{array}
\right]
}$};						
\end{tikzpicture}
\caption{The Jump Matrix $R_0^-$. The snaked line illustrates the branch cut of $\eta^{2i\nu_0^-}$.}
\label{f R0-}
\end{figure}
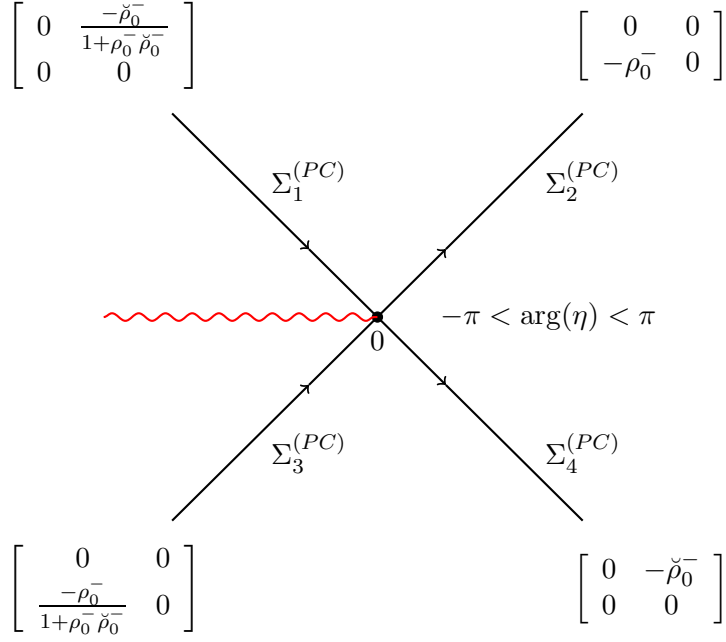. Now let us repeat the above procedure for the positive stationary phase point $+1$. That means to transform RHP($\Sigma^{(5+)}$,$R_{\tau}^{(5+)}$) into a Riemann-Hilbert problem on $\Sigma^{(PC)}$. We start with introducing another change of variables
\begin{equation}\label{e scaling + case}
  \eta(\zeta)=\sqrt{\tau}(\zeta-1),\quad
  \zeta(\eta)=\frac{1}{\sqrt{\tau}}\eta+1,
\end{equation}
and set
\begin{equation}\label{e M^PC+ <-> M^5+}
    M^{(PC+)}(\tau;\eta):=M^{(5+)} (\tau;\zeta(\eta)).
\end{equation}
In contrast to the scaling (\ref{e scaling - case}) used for the model RHP at the negative stationary phase point we do not rotate the contour if we use the new variable $\eta$ defined in (\ref{e scaling + case}). For that reason, $M^{(PC+)}(\tau,\eta)$ defined in (\ref{e M^PC+ <-> M^5+}) is a solution of RHP($\Sigma^{(PC)}$,$R_{\tau}^{(PC+)}$) where we recall the definition of $\Sigma^{(PC)}$, see (\ref{e def Sigma^PC}), and the jump is simply given by $R_{\tau}^{(PC+)}(\eta):=R_{\tau}^{(5+)}(\zeta(\eta))$. The following identities hold:
\begin{equation*}
    e^{i\tau \left(1+\frac{1}{2}(\zeta-1)^2\right)}= e^{i\eta^2/2}e^{i\tau},\qquad
    (\zeta-1)^{2i\nu_0^{+}}=\eta^{2i\nu_0^{+}}e^{-i\nu_0^+\ln(\tau)}.
\end{equation*}
We set
\begin{equation}\label{e def r0+}
    \begin{aligned}
        &\rho_0^+:=\rho(1)(\delta_0^+)^2e^{-i\tau }e^{i\nu_0^{+}\ln(\tau)},\\
        &\breve{\rho}_0^+:=\breve{\rho}(1)(\delta_0^+)^{-2} e^{i\tau }e^{-i\nu_0^{-}\ln(\tau)},
    \end{aligned}
\end{equation}
so that $\rho_0^+\breve{\rho}_0^+=\rho(1)\breve{\rho}(1)$ and $\nu_0^+=\frac{1}{2\pi}\log(1+\rho_0^+\breve{\rho}_0^+)$. We find
\begin{equation}\label{e formula R^PC+}
    R^{(PC+)}_{\tau}(\eta)=
    \left\{
      \begin{array}{ll}    \vspace{1mm}
        \left[
          \begin{array}{cc}
             0&0 \\
            \frac{\rho_0^+}{1+\rho_0^+\breve{\rho}_0^+} \eta^{-2i\nu_0^{+}}e^{-i\eta^2/2} & 0 \\
          \end{array}
        \right]
        , & w\in\Sigma^{(PC)}_1\\ \vspace{1mm}
        \left[
          \begin{array}{cc}
            0 &  \breve{\rho}_0^+\eta^{2i\nu_0^{+}} e^{i\eta^2/2}  \\
            0 & 0 \\
          \end{array}
        \right]
        , &  w\in\Sigma^{(PC)}_2,\\ \vspace{1mm}
        \left[
          \begin{array}{cc}
             0& \frac{\breve{\rho}_0^+} {1+\rho_0^+\breve{\rho}_0^+} \eta^{2i\nu_0^{+}} e^{i\eta^2/2} \\
            0 & 0 \\
          \end{array}
        \right]
        , & w\in\Sigma^{(PC)}_3,\\ \vspace{1mm}
        \left[
          \begin{array}{cc}
            0 & 0 \\
            \rho_0^+ \eta^{-2i\nu_0^{+}} e^{-i\eta^2/2} & 0 \\
          \end{array}
        \right]
        , &w\in\Sigma^{(PC)}_4. \\
      \end{array}
    \right.
\end{equation}
In Figure \ref{f R0+} 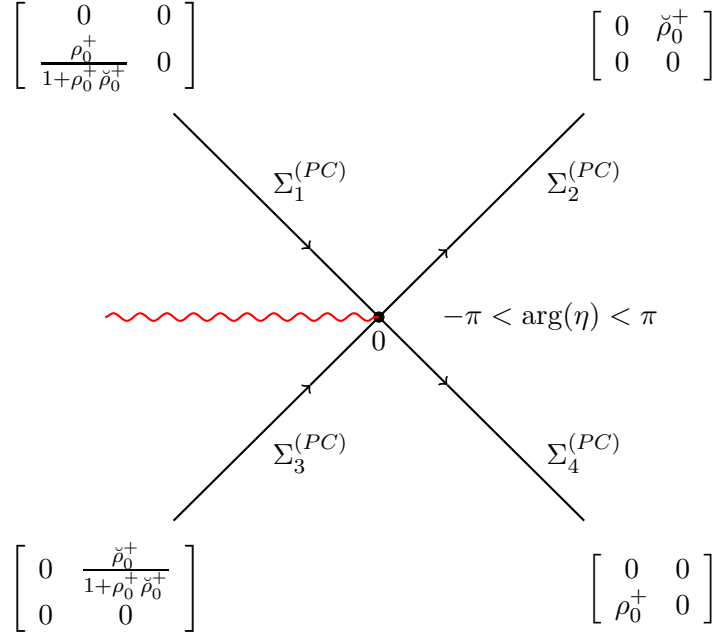
\begin{figure}
\centering
\vskip 15pt
\begin{tikzpicture}[scale=0.9]
%
%
\draw [fill] (4,4) circle [radius=0.075];						 
\node at (4.0,3.65) {$0$};										 
%
%
\draw 	[->, thick]  	(4,4) -- (5,5) ;								 
\draw		[thick] 		(5,5) -- (7,7) ;
\draw		[->, thick] 	(1,7) -- (3,5) ;								 
\draw		[thick]		(3,5) -- (4,4);	
\draw		[->, thick]	(1,1) -- (3,3);								 
\draw		[thick]		(3,3) -- (4,4);
\draw		[->,thick]	(4,4) -- (5,3);								 
\draw		[thick]  		(5,3) -- (7,1);
%
%
\draw [  thick, red, decorate, decoration={snake,amplitude=0.5mm}] (0,4)  -- (4,4);				 
\node at (6.5,4) {$-\pi < \arg(\eta) < \pi$};
%
%
\node at (7,6)  {$\Sigma_2^{(PC)}$};
\node at (3,6) 	{$\Sigma_1^{(PC)}$};
\node at (3,2)	{$\Sigma_3^{(PC)}$};
\node at (7,2) 	{$\Sigma_4^{(PC)}$};
%
%
\node at (8,8) {${
\left[
  \begin{array}{cc}
    0 &  \breve{\rho}_0^+ \\
    0& 0 \\
  \end{array}
\right]
}$};						
\node at (0,8) {${
\left[
  \begin{array}{cc}
    0 & 0  \\
    \frac{\rho_0^+}{1+\rho_0^+\breve{\rho}_0^+} & 0 \\
  \end{array}
\right]
}$};						
\node at (0,0) {${
\left[
  \begin{array}{cc}
    0 & \frac{\breve{\rho}_0^+} {1+\rho_0^+\breve{\rho}_0^+} \\
    0 & 0 \\
  \end{array}
\right]
}$};						
\node at (8,0) {${
\left[
  \begin{array}{cc}
    0 & 0\\
    \rho_0^+ & 0 \\
  \end{array}
\right]
}$};						
\end{tikzpicture}
\caption{The Jump Matrix $R_0^+$. The snaked line illustrates the branch cut of $\eta^{2i\nu_0^+}$.}
\label{f R0+}
\end{figure}  we illustrate the matrix $R_0^+$ defined on $\Sigma^{(PC)}$, where $R_0^+$ has the meaning that
\begin{equation*}
    1+R_{\tau}^{(PC+)}(\eta)=S(\eta) R_0^+ [S(\eta)]^{-1},\quad S(\eta)=\eta^{i\nu_0^+\sigma_3} e^{\frac{i}{4}\eta^2\sigma_3}.
\end{equation*}
The solution of the two Riemann-Hilbert problems RHP($\Sigma^{(PC)}$,$R_{\tau}^{(PC\pm)}$) is not worked out in this paper. However, the reader might wonder why the superscript $(PC)$ is used for the model Riemann-Hilbert problems. Here, $PC$ stands for Parabolic Cylinder functions. This well-known class of functions plays a certain role in the derivation of the following which we take from \cite{LiuPerrySulem2017}.
\begin{prop}\label{p solution PC}
    Given complex non-zero constants $\rho_0^{\pm}$ and $\breve{\rho}_0^{\pm}$ such that $1+\rho_0^{\pm}\breve{\rho}_0^{\pm}\in\R^+$, the Riemann-Hilbert problems
    RHP($\Sigma^{(PC)}$,$R_{\tau}^{(PC\pm)}$) with $R_{\tau}^{(PC\pm)}$ defined in (\ref{e formula R^PC-}) and (\ref{e formula R^PC+}) and with $\Sigma^{(PC)}$ depicted in Figure \ref{f SigmaPC} are solvable. The solutions take the form
    \begin{equation}\label{e expansion M^PC+-}
        M^{(PC\pm)}(\tau;\eta)=1+\frac{1}{\eta}
        \left[
          \begin{array}{cc}
             0 & \mp i\beta^{\pm}_{12} \\
             \pm i\beta^{\pm}_{21} & 0 \\
          \end{array}
        \right]
        + \mathcal{O}\left(\frac{1}{\eta^2}\right),
    \end{equation}
    as $|\eta|\to\infty$, where the constants $\beta^{\pm}_{12}$ and $\beta^{\pm}_{21}$ can be computet from $\rho^{\pm}_0$ and $\breve{\rho}^{\pm}_0$ through
    \begin{equation}\label{e beta12-and beta21}
        \begin{aligned}
            &\beta^-_{12}=\frac{\sqrt{2\pi} e^{\pi\nu_0^-/2} e^{i\pi/4}}{\rho_0^-\Gamma(i\nu_0^-)}
            ,&&\qquad\beta^-_{21}= \frac{\sqrt{2\pi}e^{\pi\nu_0^-/2} e^{-i\pi/4}}{\breve{\rho}_0^- \Gamma(-i\nu_0^-)},\\
            &\beta^+_{12}=\frac{\sqrt{2\pi}e^{\pi\nu_0^+/2} e^{3\pi i/4}}{\rho_0^+\Gamma(-i\nu_0^+)}
            ,&&\qquad\beta^+_{21}= \frac{\sqrt{2\pi}e^{\pi\nu_0^+/2} e^{-3\pi i/4}}{\breve{\rho}_0^+\Gamma(i\nu_0^+)}.
            \phantom{\int^{\int^{\int}\int^{\int^{\int}}}}
        \end{aligned}
    \end{equation}
\end{prop}


Using the relations (\ref{e M^PC- <-> M^5-}) and (\ref{e M^PC+ <-> M^5+}) and the two different scalings (\ref{e scaling - case}) and (\ref{e scaling + case}) we obtain $M^{(5\pm)}(\tau;\zeta)=M^{(PC\pm)} (\tau;\pm\sqrt{\tau}(\zeta\mp 1))$ and then we find that
\begin{equation*}
    \lim_{\zeta\to\infty} \zeta\cdot \left(M^{(5\pm)}(\tau;\zeta)-1\right) =
    \frac{\pm 1}{\sqrt{\tau}}
    \lim_{\eta\to\infty} \eta\cdot \left(M^{(PC\pm)}(\tau;\eta)-1\right)
\end{equation*}
and
\begin{equation*}
    M^{(5\pm)}(\tau;0)=M^{(PC\pm)} (\tau;-\sqrt{\tau}).
\end{equation*}
Thus, if we substitute the definitions of $\rho_0^{\pm}$ and $\breve{\rho}_0^{\pm}$, see (\ref{e def r0-}) and (\ref{e def r0+}), into the formulas for $\beta^+_{12}$ and $\beta^-_{12}$, see (\ref{e beta12-and beta21}), we find the following proposition as a Corollary to Proposition \ref{p solution PC}:
\begin{prop}\label{p M^5}
    Under the assumptions that $|\log(1+\rho(\pm 1)\breve{\rho}(\pm 1)|\leq c$ for some constant $c>0$, and that $\rho(\pm1)\neq0$ and $\breve{\rho}(\pm1)\neq0$, the Riemann-Hilbert problems RHP($\Sigma^{(5\pm)}$,$R^{(5\pm)}$) are uniquely solvable and the functions
    \begin{equation*}
       q^{(5\pm)}(\tau)=\lim_{\zeta\to\infty} \zeta\cdot \left[M^{(5\pm)}(\tau;\zeta)\right]_{12}
    \end{equation*}
    are explicitly given by
    \begin{equation}\label{e q^5- q^5+}
        \begin{aligned}
            q^{(5-)}(\tau)&=\frac{e^{-i\tau }e^{i\nu_0^{-}\ln(\tau)}}{\tau^{1/2}} \frac{\sqrt{2\pi}e^{\pi\nu_0^-/2} e^{-i\pi/4}} {\rho(-1)(\delta_0^-)^2\Gamma(i\nu_0^-)}\\
            q^{(5+)}(\tau)&=\frac{e^{i\tau }e^{-i\nu_0^{+}\ln(\tau)}}{\tau^{1/2}} \frac{\sqrt{2\pi}e^{\pi\nu_0^+/2} e^{i\pi/4}} {\rho(1)(\delta_0^+)^2\Gamma(-i\nu_0^+)}
        \end{aligned}
    \end{equation}
    Moreover, we have
    \begin{equation}\label{e M^5(0)}
        M^{(5\pm)}(\tau;0)=
        \left[
          \begin{array}{cc}
            1 & 0 \\
            0 & 1 \\
          \end{array}
        \right]+\mathcal{O}(\tau^{-1/2}),
    \end{equation}
    as $\tau\to\infty$. The implied constant is independent of $\rho(\pm1)$ and $\breve{\rho}(\pm1)$. Also,
    \begin{equation}\label{e L^infty M^5}
        \|M^{(5\pm)}(\tau;\cdot)\|_{L^{\infty}(\C)}\leq C.
    \end{equation}
\end{prop} 
\subsection{Truncated crosses}\label{s M^4 approx M^5}
This subsection is devoted to the Riemann-Hilbert problems RHP($\Sigma^{(4\pm)}$,$R_{\tau}^{(4\pm)}$), where the contours $\Sigma^{(4+)}$ and $\Sigma^{(4-)}$ are depicted in Figure \ref{f Sigma4}. We recall that $R_{\tau}^{(4\pm)}$ are given by
\begin{equation*}
    R_{\tau}^{(4\pm)}=\left.R_{\tau}^{(2)}\right| _{\Sigma^{(4\pm)}}.
\end{equation*}
See (\ref{e def R^2-}) and (\ref{e def R^2+}) for the definition of $R_{\tau}^{(2)}$. 
Our next basic result is the following proposition.
\begin{prop}\label{p M^4 approx M^5}
    Under the same assumptions as in Proposition \ref{p M^5}, there exist a constant $\tau_0>0$ such that
    the Riemann-Hilbert problems RHP($\Sigma^{(4\pm)}$,$R^{(4\pm)}$) are uniquely solvable for all $\tau\geq\tau_0$. Moreover, there exist positive constants $C_1$ and $C_2$ such that the functions
    \begin{equation*}
       q^{(4\pm)}(\tau)=\lim_{\zeta\to\infty} \zeta\cdot \left[M^{(4\pm)}(\tau;\zeta)\right]_{12}
    \end{equation*}
    satisfy
    \begin{equation}\label{e M^4 approx M^5}
        \left|q^{(4\pm)}(\tau) -q^{(5\pm)}(\tau)\right|\leq C_1\tau^{-1},
    \end{equation}
    for all $\tau\geq\tau_0$, where $q^{(5\pm)}(\tau)$ are given in (\ref{e q^5- q^5+}).  We also have
    \begin{equation}\label{e L^infty M^4}
        \|M^{(4\pm)}(\tau;\cdot)\|_{L^{\infty}(\C)}\leq C_1,
    \end{equation}
    and for a fixed $\epsilon>0$ and any $\zeta_0$ with $\dist(\Sigma^{(4\mp)},\zeta_0)>\epsilon$,
    \begin{equation}\label{e M^4(zeta0)}
        \left|M^{(4\pm)}(\tau;\zeta_0)-
        \left[
          \begin{array}{cc}
            1 & 0 \\
            0 & 1 \\
          \end{array}
        \right]\right|\leq C_2\tau^{-1/2}.
    \end{equation}
    $C_1$ and $\tau_0$ depend on $\Gamma_0$ as in (\ref{e assumption3}) only. $C_2$ may also depend on $\epsilon$.
\end{prop}
\begin{proof}
    We only give a proof for the "$-$'' case of (\ref{e M^4 approx M^5})--(\ref{e M^4(zeta0)}). The idea is to construct the solution of RHP($\Sigma^{(4-)}$,$R_{\tau}^{(4-)}$) from the solution of RHP($\Sigma^{(5-)}$,$R_{\tau}^{(5-)}$) which is provided by Proposition \ref{p M^5}. Therefore we seek for a matrix-valued function $F$ such that
    \begin{equation}\label{e solution formula M^4}
        M^{(4-)}(\tau;\zeta)=F(\tau;\zeta)M^{(5-)}(\tau;\zeta).
    \end{equation}
    A direct computation shows that $F$ needs to be the solution of a normalized Riemann-Hilbert problem RHP($\Sigma^{(5-)}$,$R_{\tau}^{(F)}$), where the jump is given by
    \begin{equation}\label{e formula R^F}
        1+R_{\tau}^{(F)}=M_-^{(5-)}(1+R_{\tau}^{(4-)})
        (1-R_{\tau}^{(5-)})\left[M_-^{(5-)}\right]^{-1},\quad \zeta\in\Sigma^{(5-)}.
    \end{equation}
    Here we set $R^{(4-)}_{\tau}(\zeta)=0$ for $\zeta\in\Sigma^{(5-)}\setminus\Sigma^{(4-)}$. Otherwise (\ref{e formula R^F}) would not make sense because $R_{\tau}^{(4-)}$ is not everywhere defined on $\Sigma^{(5-)}$. Our goal is to apply the small norm RHP theory presented in the Appendix \ref{a RHP theory}, see Theorem \ref{t rhp theory}. This requires bounds for the $L^{\infty}$ and $L^1$ norms of $R^{(F)}_{\tau}$. For this purpose we use the triangularity of $1+R_{\tau}^{(4-)}$ and $1+R_{\tau}^{(5-)}$ and arrange (\ref{e formula R^F}) in the following way.
    \begin{equation*}
        R_{\tau}^{(F)}(\zeta)=M_-^{(5-)}(\tau;\zeta) (R_{\tau}^{(4-)}(\zeta) -R_{\tau}^{(5-)}(\zeta)) \left[M_-^{(5-)}(\tau;\zeta)\right]^{-1}
    \end{equation*}
    We learn that
    for all $w\in\Sigma^{(5-)}$,
    \begin{equation}\label{e estimate R^F}
        |R^{(F)}_{\tau}(w)|\leq c |R_{\tau}^{(4-)}(w)-R_{\tau}^{(5-)}(w)|.
    \end{equation}
    The constant $c$ is determined by $\|M_-^{(5-)}\|_{L^{\infty}(\Sigma^{(5-)})}$ and thus independent of $\tau$, see (\ref{e L^infty M^5}). Using the notation
    \begin{equation*}
        \widetilde{Z}(\zeta) =1+\frac{1}{2}(\zeta+1)^2
    \end{equation*}
    we can find a constant $c$ such that
    \begin{equation*}
        |Z(\zeta)-\widetilde{Z}(\zeta)|\leq c |\zeta+1|^3,\quad\text{for all }|\zeta+1|\leq \frac{1}{2}.
    \end{equation*}
    It follows that for all $\zeta\in B_{1/2}(-1):=\{\zeta\in\C:|\zeta+1|< 1/2\}$ we have
    \begin{equation*}
        \left|e^{\pm i\tau Z(\zeta)} -e^{\pm i\tau \widetilde{Z}(\zeta)}\right| \leq c\tau \left|e^{\pm i\tau\widetilde{Z}(\zeta)} \right||\zeta+1|^3.
    \end{equation*}
    Taking $\zeta\in\Sigma_2^{(5-)}\cap B_{1/2}(-1)$ and parameterizing $\zeta+1=se^{i\pi/4}$ with $0\leq s \leq 1/2$ we find
    \begin{equation*}
        |e^{i\tau Z(\zeta)} -e^{ i\tau \widetilde{Z}(\zeta)}|\leq c\gamma_{\tau}(s),\qquad\text{ with }\gamma_{\tau}(s)=\tau e^{-\tau s^2/2}s^3.
    \end{equation*}
    As it is shown easily, $\gamma_{\tau}$ has the following properties
    \begin{equation*}
        \|\gamma_{\tau}\|_{L^{\infty}(\R_+)}\leq c_1\tau^{-1/2},\qquad
        \|\gamma_{\tau}\|_{L^{1}(0,1/2)}\leq c_2\tau^{-1}.
    \end{equation*}
    From these observations and similar estimates on the other rays $\Sigma_1^{(5-)}$, $\Sigma_3^{(5-)}$ and $\Sigma_4^{(5-)}$ and from (\ref{e assumption3}) we can deduce that
    \begin{equation*}
        \|R_{\tau}^{(4-)}-R_{\tau}^{(5-)} \|_{L^{\infty}(\Sigma^{(5-)}\cap B_{1/2}(-1))}\leq c \Gamma_0 \tau^{-1/2}
    \end{equation*}
    and
    \begin{equation*}
        \|R_{\tau}^{(4-)}-R_{\tau}^{(5-)}\|_{L^{1}(\Sigma^{(5-)}\cap B_{1/2}(-1))}\leq c \Gamma_0\tau^{-1}.
    \end{equation*}
    In order to show that equivalent estimates are also available on the contour $\Sigma^{(5-)}\setminus B_{1/2}(-1)$ we use the estimate $|R_{\tau}^{(4-)}-R_{\tau}^{(5-)}|\leq |R_{\tau}^{(4-)}|+|R_{\tau}^{(5-)}|$. Elementary computations show that the $L^{\infty}$-norm and also the $L^1$-norm of both, $R_{\tau}^{(4-)}$ and $R_{\tau}^{(5-)}$, decay exponentially as $\tau\to\infty$. Thus, combining all estimates, we finally find $\|R_{\tau}^{(F)}\|_{L^{\infty}(\Sigma^{(5-)})}\leq c_1 \tau^{-1/2}$ and $\|R_{\tau}^{(F)}\|_{L^{1}(\Sigma^{(5-)})}\leq c_2 \tau^{-1}$ with constants independent of $\tau$. By (\ref{e estimate R^F}), Theorem \ref{t rhp theory} and (\ref{e solution formula M^4}) these estimates are sufficient for (\ref{e M^4 approx M^5})--(\ref{e L^infty M^4}) to be valid. The Proposition is proven.
\end{proof} 
\subsection{Combining the two crosses}\label{s M^3 approx M^4}
In this subsection we attend to the Riemann-Hilbert part of the mixed $\db$-RHP \ref{dbrhp M^2} Therefore we will construct the solution $M^{(3)}$ of RHP($\Sigma^{(3)}$,$R_{\tau}^{(3)}$) from the two solutions of RHP($\Sigma^{(4\pm)}$,$R_{\tau}^{(4\pm)}$), which were constructed in the proof of Proposition \ref{p M^4 approx M^5} from two model Riemann-Hilbert problems. We recall that  the contour $\Sigma^{(3)}=\Sigma^{(2)}$ is depicted in Figure \ref{f Sigma2} and we also recall that
\begin{equation*}
    R_{\tau}^{(3)}(\zeta)=R_{\tau}^{(2)}(\zeta),\qquad \zeta \in \Sigma^{(3)}.
\end{equation*}
See (\ref{e def R^2-}) and (\ref{e def R^2+}) for the definition of $R_{\tau}^{(2)}$.
We have the following proposition.
\begin{prop}\label{p M^3 approx M^4}
    Under the same assumptions as in Proposition \ref{p M^5}, there exist a constant $\tau_0>0$ such that
    the Riemann-Hilbert problem RHP($\Sigma^{(3)}$,$R^{(3)}$) is uniquely solvable for all $\tau\geq\tau_0$. Moreover, there exists a positive constants $C$ such that the function
    \begin{equation*}
       q^{(3)}(\tau)=\lim_{\zeta\to\infty} \zeta\cdot \left[M^{(3)}(\tau;\zeta)\right]_{12}
    \end{equation*}
    satisfies
    \begin{equation}\label{e M^3 approx M^4}
        \left|q^{(3)}(\tau) -\left(q^{(4-)}(\tau)+ q^{(4+)}(\tau)\right)\right|\leq C\tau^{-3/4},
    \end{equation}
    for all $\tau\geq\tau_0$, where $q^{(4\pm)}(\tau)$ satisfy (\ref{e M^4 approx M^5}).  We also have
    \begin{equation}\label{e L^infty M^3}
        \|M^{(3)}(\tau;\cdot)\|_{L^{\infty}(\C)}\leq C,
    \end{equation}
    and
    \begin{equation}\label{e M^3(zeta0)}
        \left|M^{(3)}(\tau;0)-
        \left[
          \begin{array}{cc}
            1 & 0 \\
            0 & 1 \\
          \end{array}
        \right]\right|\leq C\tau^{-1/2},
    \end{equation}
    Either of the constants, $C$ and $\tau_0$,  does depend on the constant $\Gamma_0$ given in (\ref{e assumption3}) only.
\end{prop}
\begin{proof}
    We consider RHP($\Sigma^{(4-)}$,$R_{\tau}^{(E)}$), where the jump matrix $R_{\tau}^{(E)}$ is given by
    \begin{equation*}
        1+R_{\tau}^{(E)}=M_-^{(4-)}M^{(4+)} (1+R_{\tau}^{(3-)})[M^{(4+)}]^{-1}(1-R_{\tau}^{(3-)}) [M_-^{(4-)}]^{-1},\quad \zeta\in\Sigma^{(4-)}.
    \end{equation*}
    Denoting the solution of RHP($\Sigma^{(4-)}$,$R_{\tau}^{(E)}$) by $E(\tau,\zeta)$ we then have
    \begin{equation*}
        M^{(3)}(\tau;w)=E(\tau;w)M^{(4-)}(\tau;w)M^{(4+)}(\tau;w),
    \end{equation*}
    which is verified by computing explicitly the the jumps on $\Sigma^{(3)}=\Sigma^{(4-)}\cup\Sigma^{(4+)}$. Furthermore it follows that
    \begin{equation*}
        q^{(3)}(\tau) =[E_1(\tau)]_{12}+q^{(4-)}(\tau)+ q^{(4+)}(\tau),
    \end{equation*}
    where
    \begin{equation*}
        E(\tau;\zeta)=1+\frac{E_1(\tau)}{\zeta}+ \mathcal{O}\left(\zeta^{-2}\right),\qquad\text{as }\zeta\to\infty.
    \end{equation*}
    Thus, it suffices to show that RHP($\Sigma^{(4-)}$,$R^{(E)}$) is indeed solvable and we have to prove estimates on $E_1$. Similar to the above proof of Proposition \ref{p M^4 approx M^5} we intend to apply theory for RHPs with jump matrix $R^{(E)}$ near zero, see Appendix \ref{a RHP theory}. It follows from Theorem \ref{t rhp theory} that (\ref{e M^3 approx M^4})--(\ref{e M^3(zeta0)}) are proven if the following estimates can be verified.
    \begin{equation}\label{e L^infty R^E}
        \|R^{(E)}_{\tau}(\cdot)\|_{L^{\infty}(\Sigma^{(4-)})}\leq c_1\tau^{-1/2}
    \end{equation}
    and
    \begin{equation}\label{e L^2 R^E}
        \|R^{(E)}_{\tau}(\cdot)\|_{L^2(\Sigma^{(4-)})}\leq c_2\tau^{-1}.
    \end{equation}
    We will provide the proof of (\ref{e L^infty R^E}) and (\ref{e L^2 R^E}) in the following. Writing
    \begin{eqnarray*}
      R_{\tau}^{(E)}\!\!&=&M_-^{(4-)}\left(M^{(4+)}-1\right) R_{\tau}^{(3-)} [M^{(4+)}]^{-1} \left(1-R_{\tau}^{(3-)}\right) [M_-^{(4-)}]^{-1}\\
       && + M_-^{(4-)}R_{\tau}^{(3-)} \left(1-R_{\tau}^{(3-)}\right) \left([M^{(4+)}]^{-1}-1\right)[M_-^{(4-)}]^{-1},\qquad \zeta\in\Sigma^{(4-)},
    \end{eqnarray*}
    where we also used that $[R_{\tau}^{(3-)}]^2=0$, we learn that
    \begin{equation}\label{e e1}
        \left|R_{\tau}^{(E)}(\zeta)\right|\leq c \left|M^{(4+)}(\zeta)-1\right|\cdot \left|R_{\tau}^{(3-)}(\zeta)\right|.
    \end{equation}
    Here the constant $c$ is determined by $L^{\infty}(\Sigma^{(4-)})$-norm of $M_-^{(4-)}$, $M^{(4+)}$ and $R^{(3-)}$ and thus independent of $\tau$, see (\ref{e L^infty M^4}). For any $\zeta\in\Sigma^{(4-)}$ we have either $\dist(\Sigma^{(4+)},\zeta)>1/3$ or $R_{\tau}^{(3-)}(\zeta)=0$ which follows from the cut-off function occurring in (\ref{e def R^2-}) and (\ref{e def R^2+}). Making use of this observation and considering estimate (\ref{e M^4(zeta0)}) we realize that (\ref{e e1}) implies
    \begin{equation*}
        \left|R_{\tau}^{(E)}(\zeta)\right|\leq c \tau^{-1/2} \left|R_{\tau}^{(3-)}(\zeta)\right|.
    \end{equation*}
    Thus, by (\ref{e assumption3}) we find (\ref{e L^infty R^E}). Inequality (\ref{e L^2 R^E}) follows by Proposition \ref{p L^1norm of R^2}. The proof of the Proposition is now completed.
\end{proof}

\section{Analysis of pure $\db$ problem}\label{s pure db} The remaining part of the proof of Theorem \ref{t main thm general} is the analysis of the $\db$-part of $M^{(2)}$ in (\ref{e def M^2}), see (\ref{e db M^2}). We will show that the contribution of $\db \mathcal{W}$ in (\ref{e db M^2}) goes to $0$ to higher order and the asymtotics of $q^{(2)}$ through equation (\ref{e def q^j}) are determined by the RHP part of $M^{(2)}$, see (\ref{e jump M^2}).
$\db$-problem \ref{db} is solved by finding a solution of $D=1+J(D)$, where
\begin{equation}\label{e def J}
    J(D)(\zeta):=\frac{1}{\pi}\int_{\C} \frac{D(s)\Upsilon(s)}{s-\zeta}dA(s).
\end{equation}
The definition of $\Upsilon$ was given in (\ref{e db D}). For $J$ we can prove the following.
\begin{prop}\label{p operator J}
    For the operator in (\ref{e def J}) we have $J:L^{\infty}(\C)\to L^{\infty}(\C)$ and there exists a $C>0$ such that
    \begin{equation}\label{e J norm}
        \|J\|_{L^{\infty}(\C)\to L^{\infty}(\C)}\leq C\tau^{1/4}
    \end{equation}
    for all $\tau\in\R^+$. $C$ depends on $\rho$ and $\breve{\rho}$ through theirs $H^{1,1}(\R)$-norms and $\Gamma_0$, see (\ref{e assumption3}).
\end{prop}
\begin{proof}
    We decompose the operator into $J=J_1+...+J_8$, where
    \begin{equation*}
        J_k (D)(\zeta):=\frac{1}{\pi}\int_{\Omega_k} \frac{D(s)\Upsilon(s)}{s-\zeta}dA(s),\qquad k\in\{1,...,8\},
    \end{equation*}
    and prove (\ref{e J norm}) first for $J_4$ and afterwards for $J_3$. Other values for $j$ are similar to one of these. By Lemma \ref{l key estimate db R} it is clear that to bound $J_4$ it is enough to bound $I_1$, $I_2$ and $I_3$ defined by
    \begin{equation*}
          I_1:=\int_{\Omega_4} \frac{ \left(|\rho\,'(\Real(s))|+ |\breve{\rho}\,'(\Real(s))|\right)\left|e^{i\tau Z(s)}\right|} {|s-\zeta|} dA(s),\quad
          I_2:=\int_{\Omega_4} \frac{|s-1|^{-1/2}|e^{i\tau Z(s)}|} {|s-\zeta|} dA(s),
    \end{equation*}
    \begin{equation*}      
          I_3:=\Gamma_0\int_{\Omega_4} \frac{\left|\db \chi(s)\right| \left|e^{i\tau Z(s)}\right|} {|s-\zeta|} dA(s).
    \end{equation*}
    We start with presenting two useful estimates. First, by monotonicity and by (\ref{e estimate I check}) we have
    \begin{equation}\label{e e^theta  on Omega4}
        \sup_{x\in[y,\infty)}|e^{i\tau Z(1+x+iy)}| =|e^{i\tau Z(1+y+iy)}|\leq e^{-\tau \check{I}(y)}
    \end{equation}
    for any $y>0$. Secondly, we will need the following observation which can be obtained by elementary computations. For arbitrary $\alpha,\beta\in\R$ and any real $y\neq\beta$ we have
    \begin{equation}\label{e L^2 of |s-w|^-1}
        \left\|\frac{1}{\sqrt{(x-\alpha)^2 +(y-\beta)^2}}\right\|_{L^2_x(\R)}\leq
        c|y-\beta|^{-1/2}.
    \end{equation}
    Let us now begin with estimating $I_1$. Writing $s-1=x+iy$ and $\zeta-1=\alpha+i\beta$ and making use of (\ref{e e^theta  on Omega4}) and (\ref{e L^2 of |s-w|^-1}) we obtain
    \begin{equation}\label{e estimate I1}
        \begin{aligned}
            I_1&=&& \int_0^{\infty}\int_y^{\infty} \frac{(|\rho\,'(x+1)|+|\breve{\rho}\,'(x+1)|)\left|e^{i\tau Z(1+x+iy)}\right|}{\sqrt{(x-\alpha)^2 +(y-\beta)^2}}dxdy\\
            &\leq&&(\|\rho\,'\|_{L^2(\R)} +\|\breve{\rho}\,'\|_{L^2(\R)} ) \int_{0}^{\infty}e^{-\tau \check{I}(y)} \left\|\frac{1}{\sqrt{(x-\alpha)^2 +(y-\beta)^2}}\right\|_{L^2_x((y,\infty))}dy \\
            &\leq&& c(\|\rho\,'\|_{L^2(\R)} +\|\breve{\rho}\,'\|_{L^2(\R)} ) \int_{0}^{\infty}e^{-\tau \check{I}(y)} |y-\beta|^{-1/2}dy.
        \end{aligned}
    \end{equation}
    By monotonicity of $\check{I}$ we find for any $\beta\in\R$
    \begin{equation}\label{e y-beta story}
            \int_{\R}\frac{e^{-\tau \check{I}(|y|)}}{ |y-\beta|^{1/2}} dy\leq\int_{|y|\leq |y-\beta|}\frac{e^{-\tau \check{I}(|y|)}}{|y|^{1/2}} dy
            +\int_{|y|\geq |y-\beta|}\frac{e^{-\tau \check{I}(|y-\beta|)}}{|y-\beta|^{1/2}}dy
            \leq 2\int_{\R}\frac{e^{-\tau \check{I}(|y|)}}{ |y|^{1/2}} dy.
    \end{equation}
    Furthermore, by definition of $\check{I}$ and appropriate substitutions we find
    \begin{equation*}
        \int_{0}^{\infty}\frac{e^{-\tau \check{I}(|y|)}} {\sqrt{y}} dy \leq \frac{c_1}{\tau^{1/4}} \int_{0}^{\infty}\frac{e^{-y^2}}{\sqrt{y}} dy+  \frac{c_2}{\tau^{1/2}} \int_{0}^{\infty}\frac{e^{-y}}{\sqrt{y}} dy\leq c \left(\frac{1}{\tau^{1/4}}+\frac{1}{\tau^{1/2}}\right)
    \end{equation*}
    and thus
    \begin{equation*}
        I_1\leq  \frac{c}{\tau^{1/4}}(\|\rho\,'\|_{L^2(\R)} +\|\breve{\rho}\,'\|_{L^2(\R)} ) .
    \end{equation*}
    Proceeding like in (\ref{e estimate I1}) we now consider
    \begin{equation*}
        I_2\leq \int_{0}^{\infty}e^{-\tau \check{I}(y)} \left\|\frac{1}{(x^2 +y^2)^{1/4}}\right\|_{L^2_x((y,\infty))}\left\|\frac{1}{\sqrt{(x-\alpha)^2 +(y-\beta)^2}}\right\|_{L^{2}_x((y,\infty))}dy.
    \end{equation*}
    A direct computation shows that $\left\|(x^2 +y^2)^{-1/4}\right\|_{L^2_x((y,\infty))}$ does not depend on $y>0$ . Thus, we can copy the arguments from above in order to obtain $I_2\leq c \tau^{-1/4}$. Similarly, $I_3\leq c \Gamma_0 \tau^{-1/4}$ which proves (\ref{e J norm}) for $J_4$. The case $J_3$ is similar. We define
    \begin{equation*}
        K_1:=\int_{\Omega_3} \frac{\left(|\rho\,'(\Real(s))|+ |\breve{\rho}\,'(\Real(s))|\right)|e^{-i\tau Z(s)}|} {|s-\zeta|} dA(s),\quad
        K_2:=\int_{\Omega_3} \frac{|s-1|^{-1/2}|e^{-i\tau Z(s)}|} {|s-\zeta|} dA(s),
    \end{equation*}
    \begin{equation*}
          K_3:=\Gamma_0\int_{\Omega_3} \frac{\left|\db \chi(s)\right| \left|e^{i\tau Z(s)}\right|} {|s-\zeta|} dA(s),
    \end{equation*}
    such that by Lemma \ref{l key estimate db R} $\|J_3\|_{L^{\infty}(\C)\to L^{\infty}(\C)}\leq c(K_1+K_2+K_3)$. Using again obvious properties of the Joukowsky transform $Z$ we can easily see that for all $y\in(0,1)$,
    \begin{equation}\label{e e^theta on Omega3}
        \sup_{x\in(-1,-y]}|e^{-i\tau Z(1+x+iy)}| =|e^{-i\tau Z(1-y+iy)}|\leq e^{\tau I(y)},
    \end{equation}
    where $I(y)$ is the same function as defined in (\ref{e def help fct I}). Using (\ref{e e^theta on Omega3}) and (\ref{e L^2 of |s-w|^-1}) we find
    \begin{equation*}
        \begin{aligned}
            K_1&=&& \int_0^{1}\int_{-1}^{-y} \frac{(|\rho\,'(x+1)|+|\breve{\rho}\,'(x+1)|) \left|e^{i\tau Z(1+x+iy)}\right|}{\sqrt{(x-\alpha)^2 +(y-\beta)^2}}dxdy\\
            &\leq&&(\|\rho\,'\|_{L^2(\R)} +\|\breve{\rho}\,'\|_{L^2(\R)} ) \int_{0}^{1}e^{\tau I(y)} \left\|\frac{1}{\sqrt{(x-\alpha)^2 +(y-\beta)^2}}\right\|_{L^2_x((-1,-y))}dy \\
            &\leq&& c(\|\rho\,'\|_{L^2(\R)} +\|\breve{\rho}\,'\|_{L^2(\R)} ) \int_{0}^{1}e^{\tau I(y)} |y-\beta|^{-1/2}dy\\
            &=&& c(\|\rho\,'\|_{L^2(\R)} +\|\breve{\rho}\,'\|_{L^2(\R)} ) \left\{\int_{0}^{1/2}e^{-\frac{\tau}{4}y^2} |y-\beta|^{-1/2}dy+ \int_{1/2}^{1}e^{-\frac{\tau}{4} (y-1)^2} |y-\beta|^{-1/2}dy\right\}\\
            &=&& c(\|\rho\,'\|_{L^2(\R)} +\|\breve{\rho}\,'\|_{L^2(\R)} ) \left\{\int_{0}^{1/2}e^{-\frac{\tau}{4} y^2} |y-\beta|^{-1/2}dy+ \int_0^{1/2}e^{-\frac{\tau}{4} y^2} |y+1-\beta|^{-1/2}dy\right\}.\\
        \end{aligned}
    \end{equation*}
    Applying the idea of (\ref{e y-beta story}) we immediately end up with $K_1\leq  c (\|\rho\,'\|_{L^2(\R)} +\|\breve{\rho}\,'\|_{L^2(\R)} )\tau^{-1/4}$. For estimating $K_2$ we write the same as we did for $I_2$:
    \begin{equation*}
        K_2\leq \int_{0}^{1}e^{\tau I(y)} \left\|(x^2 +y^2)^{-1/4}\right\|_{L^2_x((-1,-y))} \left\|\frac{1}{\sqrt{(x-\alpha)^2 +(y-\beta)^2}}\right\|_{L^2_x((-1,-y))}dy
    \end{equation*}
    Since $\left\|(x^2 +y^2)^{-1/4}\right\|_{L^2_x((-w_0,-y))}$ can be bounded by a constant that is independent of $y$ we deduce $K_2\leq  c \tau^{-1/4}$. Since $\chi$ is finitely supported we can estimate $K_3$ in the same way as $K_1$ and obtain $K_3\leq  c\Gamma_0\tau^{-1/4}$ Therefore, we have shown that $\|J_3+J_4\|_{L^{\infty}\to L^{\infty}}\leq c \tau^{-1/4}(1+\|\breve{\rho}\,'\|_{L^2(\R)} +\|\rho'\|_{L^2(\R)})$. This suffices to prove the Proposition.
\end{proof}
\begin{lem}\label{l existence and asymptotic of db problem}
    For sufficiently large $\tau>0$, there exists a unique solution $D$ for $\db$-problem \ref{db}. For $|\Imag(\zeta)|\to \infty$ $D$ has the property that
    \begin{equation}\label{e D expansion}
        D(\tau;\zeta)=1+\frac{D_1(\tau)}{\zeta}+\mathcal{O} \left(\zeta^{-2}\right),
    \end{equation}
    where
    \begin{equation}\label{e estimate D1}
        |D_1(\tau)|\leq C\tau^{-3/4}.
    \end{equation}
    Moreover,
    \begin{equation}\label{e estimate D(zeta1)}
        |D(\tau;0)-1|\leq C\tau^{-3/4}.
    \end{equation}
    Here the constant $C$ is independent of $\tau$ and depends on $\rho$ and $\breve{\rho}$ through theirs $H^{1,1}(\R)$-norms and $\Gamma_0$ as in (\ref{e assumption3}).
\end{lem}
\begin{proof}
    We find a solution of $\db$-problem \ref{db} by solving the integral equation $D=1+J(D)$. By the above proposition this equation is uniquely solvable in the space $L^{\infty}(\C)$ for all $\tau\geq\tau_0$ for some $\tau_0>0$. Moreover, we have $\|D\|_{L^{\infty}(\C)}\leq c$ uniformly for all $\tau\geq\tau_0$. Therefore the coefficient $D_1$ in the expansion (\ref{e D expansion}) can be expressed by $D_1=\frac{1}{\pi}\int_{\C}D \Upsilon dA$. It follows that to bound $|D_1|$ it suffices to bound $\int_{\Omega_j}\Upsilon dA$ for $j=1,...,8$. Let us consider the case of $j=4$. Analogously to the proof of Proposition \ref{p operator J} we have to bound the following three integrals:
    \begin{equation*}
          I_1':=\int_{\Omega_4} \left(|\rho\,'(\Real(s))|+ |\breve{\rho}\,'(\Real(s))|\right)\left|e^{i\tau Z(s)}\right|dA(s),\quad
          I_2':=\int_{\Omega_4} |s-1|^{-1/2}|e^{i\tau Z(s)}|dA(s),
    \end{equation*}
    \begin{equation*}
          I_3':=\Gamma_0\int_{\Omega_4} \left|\db \chi(s)\right| \left|e^{i\tau Z(s)}\right| dA(s).
    \end{equation*}
    Writing again $s-1=x+iy$ and differentiating between cases $x\leq 1$ and $x\geq 1$, see Figure \ref{f decomposition Omega4}
    \begin{figure}
\begin{center}
\begin{tikzpicture}
\path[fill=black,opacity=0.21] (1,0) -- (1,3) -- (-2,0) -- (1,0);
\path[fill=black,opacity=0.09] (1,0) -- (1,3) -- (3,5) -- (3,0) -- (1,0);
\path[fill=black,opacity=0.21] (-5,1.5) -- (-3.5,1.5) -- (-2,0) -- (-5,0) -- (-5,1.5);
\path[fill=black,opacity=0.09] (-5,1.5) -- (-3.5,1.5) -- (-5,3) -- (-5,1.5);
\draw 	[->]  	(-5.5,0) -- (3.5,0) ;
\draw 	[->]  	(-5,-.5) -- (-5,5) ;
\draw  	(-2.2,-0.2) -- (3.2,5.2) ;
\draw  	(-1.8,-0.2) -- (-5.,3) ;
\draw 	[thick]  	(-2,.1) -- (-2,-.1) ;
\draw 	[thick]  	(1,.1) -- (1,-.1) ;
\draw 	[dashed]  	(1,0) -- (1,3) ;
\draw 	[dashed]  	(-5,1.5) -- (-3.5,1.5) ;

\node at (-2,-.5) {$1$};
\node at (1,-.5) {$2$};
\node at (-5.5,3) {$i$};
\node at (-5.5,1.5) {$\frac{i}{2}$};
\node at (-.1,.8) {$\Omega_{4,1}$};
\node at (2.2,1.8) {$\Omega_{4,2}$};
\node at (-4,.65) {$\Omega_{3,1}$};
\node at (-4.5,1.9) {$\Omega_{3,2}$};
\end{tikzpicture}
\end{center}
  \caption{Decompositon of $\Omega_3$ and $\Omega_4$ for the proof of Lemma \ref{l existence and asymptotic of db problem}.} \label{f decomposition Omega4}
\end{figure}
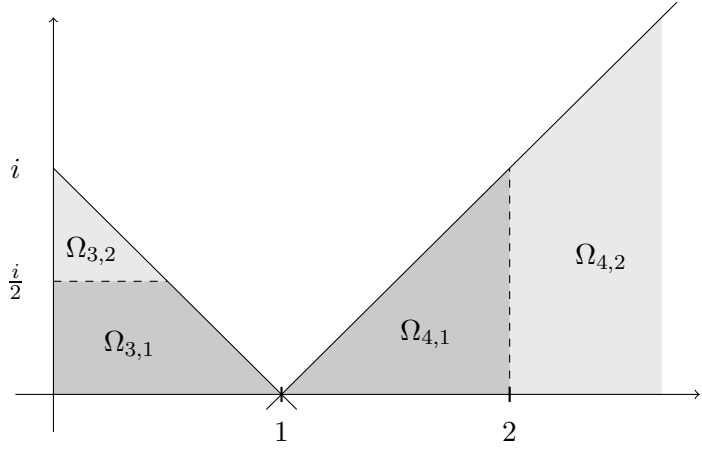 
    , we can obtain from elementary computations that
    \begin{equation}\label{e estimate e^theta Omega41 and Omega 42}
        \left|e^{-i\tau \theta_{w_0}(s)}\right|\leq
        \left\{
          \begin{array}{ll}
            e^{-\tau c_1 xy}, & s\in\Omega_{4,1}, \\
            e^{-\tau c_2 y}, & s\in\Omega_{4,2}.
          \end{array}
        \right.
    \end{equation}
    We can now begin with estimating $I_1'$.
    \begin{eqnarray*}
      I_1' &\leq& \int_{0}^{\infty}\int_y^{\infty}(|\breve{\rho}\,' (x+1)|+|\rho\,' (x+1)|)e^{-\tau c_1 xy}dxdy\\&&\quad + \int_{0}^{\infty}\int_y^{\infty}(|\breve{\rho}\,' (x+1)|+|\rho\,' (x+1)|)e^{-\tau c_2 y}dxdy  \\
       &\leq& (\|\rho\,'\|_{L^{2}(\R)} + \|\breve{\rho}\,'\|_{L^{2}(\R)} ) \int_{0}^{\infty}\|e^{-\tau c_1 xy }\|_{L_x^{2}(y,\infty)}dy+  (\|\rho\,'\|_{L^{1}(\R)} + \|\breve{\rho}\,'\|_{L^{1}(\R)} ) \int_{0}^{\infty}e^{-\tau c_2 y}dy\\
       &\leq& c (\|\rho\|_{H^{1,1}(\R)} +\|\breve{\rho}\|_{H^{1,1}(\R)} ) \left(\frac{1}{\tau^{1/2}} \int_{0}^{\infty}\frac{e^{-\tau c_1 y^2}}{y^{1/2}}dy+\frac{1}{\tau}\int_{0}^{\infty}e^{-c_2 y}dy\right)\\
       &\leq& c (\|\rho\|_{H^{1,1}(\R)} +\|\breve{\rho}\|_{H^{1,1}(\R)} ) \left(\frac{1}{\tau^{3/4}} +\frac{1}{\tau}\right).
    \end{eqnarray*}
    Recalling the bound $\|(x^2 +y^2)^{-1/4}\|_{L^2_x((y,\infty))}\leq c$ from the proof of Proposition \ref{p operator J} we can bound $I_2'$ $I_3'$ in a similar way so that finally
    \begin{equation*}
        \left|\int_{\Omega_4}\Upsilon(s) dA(s)\right|\leq c \tau^{-3/4}.
    \end{equation*}
    Let us now estimate $\int_{\Omega_3}\Upsilon(s) dA(s)$. Writing again $s-1=x+iy$ we have for any $y\in(0,\frac{1}{2})$ and any $x\in(-1,-y)$ that
    \begin{equation*}
        1-|s|^2>-cx,
    \end{equation*}
    where $c$ is a sufficiently small and positive constant not depending on $x$ and $y$ . This observation can be used to find $|e^{i\tau Z(s)}|\leq e^{\tau xy/2}$ for $s\in \Omega_{3,1}$ as depicted in Figure \ref{f decomposition Omega4}. It follows that
    \begin{eqnarray*}
      \int_{\Omega_{3,1}} |{\rho}\,'(\Real(s))||e^{-i\tau Z(s)}| dA(s) &\leq&
      \int_0^{1}\int_{-1}^{-y} |{\rho}\,'(x+1)|\left|e^{\tau  xy/2}\right|dxdy\\
      &\leq&\|{\rho}\,'\|_{L^2(\R)} \int_0^{\infty}\|e^{-\tau  xy/2}\|_{L_x^2(y,\infty)}dy\\
      &\leq& C\tau^{-3/4}\|{\rho}\,'\|_{L^2(\R)}
    \end{eqnarray*}
    and the same for $\breve{\rho}$ instead of $\rho$. By the same argument we also find
    \begin{equation*}
        \int_{\Omega_{3,1}} \left(|s-1|^{-1/2}+|\db\chi(s)|\right)|e^{-i\tau Z(s)}| dA(s) \leq C\tau^{-3/4}.
    \end{equation*}
    To bound the integral over $\Omega_{3,2}$ (see Figure \ref{f decomposition Omega4}) we write $s-i=x+iy$. For $y\in(-1/2,0)$ and $x\in(0,-y)$ (which is equivalent to $s\in\Omega_{3,2}$) we find
    \begin{equation*}
        |s|^2-1=(x^2+y^2)+2y\leq 2y^2+2y\leq y.
    \end{equation*}
    Thus, $|e^{-i\tau Z(s)}|= e^{\tau \frac{\Imag(s)}{2}\frac{|s|^2-1}{|s|^2}}\leq e^{\tau y/4}$ for $s\in\Omega_{3,2}$. It follows that
    \begin{eqnarray*}
      \int_{\Omega_{3,2}} |\rho'(\Real(s))|\left|e^{-i\tau Z(s)}\right| dA(s) &\leq&
      \int_{-1/2}^0\int_{0}^{-y} |\rho'(x)|e^{\tau y/4}dxdy\\
      &\leq&\|\rho'\|_{L^2(\R)}\int_0^{\infty}e^{-\tau y/4}\sqrt{y}dy\\
      &\leq& C\tau^{-3/2}\|\rho'\|_{L^2(\R)}
    \end{eqnarray*}
    and the same for $\breve{\rho}$ instead of $\rho$. By the same argument we also find
    \begin{equation*}
        \int_{\Omega_{3,2}} \left(|s-1|^{-1/2}+|\db\chi(s)|\right)|e^{-i\tau Z(s)}| dA(s) \leq C\tau^{-3/2}.
    \end{equation*}
    Altogether we have proven that
    \begin{equation*}
        \left|\int_{\Omega_3\cup\Omega_4}\Upsilon(\tau;s) dA(s)\right|\leq c \tau^{-3/4}.
    \end{equation*}
    Since other regions $\Omega_j$ can be considered in a similar way we can conclude the proof of the Lemma.
\end{proof} 
\section{Soliton resolution}\label{s solitons}
The preceding sections were dealing with pure radiation solutions of (\ref{e mtm}). That is the case $N=0$ in RHP's \ref{rhp M} and \ref{rhp M hat}. Let us now assume $N\geq1$ and $\mathcal{S}(u_0,v_0)=(p;\{\lambda_j,C_j\}_{j=1}^N)$ as in (\ref{e def Scattering data}). For velocities $-1<v_1\leq v_2<1$ and initial points $-\infty < x_1\leq x_2 < \infty$ we want to compute the asymptotic behavior of solutions $(u,v)$ for (\ref{e mtm}) in the space-time cone
\begin{equation}\label{e space-time cone}
    \mathcal{K}(v_1,v_2,x_1,x_2)=\{(t,x): t>0\text{ and } x=\xi+\eta t\text{ for }\xi\in[x_1,x_2], \eta\in[v_1,v_2]\}.
\end{equation}
For arbitrary $\epsilon>0$ and sufficiently large $t$ the set $\mathcal{K}$ corresponds to those $(t,x)$ for which $z_0=w_0^{-1}\in (L_1-\epsilon,L_2+\epsilon)$ with
\begin{equation}\label{e L-j}
    L_j=\sqrt{\frac{1-v_j}{1+v_j}},\qquad j\in \{1,2\}.
\end{equation}
We set
\begin{equation}\label{e Lambda(S)}
    \Lambda(\mathcal{K})=\{k:L_1\leq |\lambda_k|^2 \leq L_2\}
\end{equation}
and it will turn out that only solitons corresponding to eigenvalues $\lambda_k$ with $k\in\Lambda(\mathcal{K})$ will be visible in $\mathcal{K}$. The remaining solitons corresponding to eigenvalues $\lambda_k$ with $k$ belonging to one of the sets
\begin{equation}\label{e Lambda(S) left right}
    \Lambda^{\leftarrow}(\mathcal{K})=\{k:|\lambda_k|^2< L_1\},\qquad
    \Lambda^{\rightarrow}(\mathcal{K})=\{k:|\lambda_k|^2 > L_2\},
\end{equation}
will eventually leave $\mathcal{S}$ to the left (if $k\in \Lambda^{\leftarrow}(\mathcal{K})$) or to the right (if $k\in \Lambda^{\rightarrow}(\mathcal{K})$).
A first step of proving the soliton resolution conjecture is to prove the following Theorem:
\begin{thm}\label{t solres1}
    Suppose that $u_0,v_0\in H^2(\R)\cap H^{1,1}(\R)$ with scattering data $\mathcal{S}(u_0,v_0)=(p;\{\lambda_k,C_k\}_{k=1}^N)$. For a given space-time cone $\mathcal{K}(v_1,v_2,x_1,x_2)$ of the form (\ref{e space-time cone}), define new scattering data $$(\widetilde{p}; \{\lambda_k,\widetilde{C}_k\} _{k\in\Lambda(\mathcal{K})})$$ by
    \begin{equation}\label{e p tilde Ck tilde}
        \widetilde{p}(\lambda)=p(\lambda)\prod_{j\in \Lambda^{\leftarrow}(\mathcal{K})} \frac{\overline{\lambda}_j^2}{\lambda_j^2} \left(\frac{\lambda^2-\lambda^2_j} {\lambda^2-\overline{\lambda}_j^2}\right)^2,\qquad \widetilde{C}_k=C_k \prod_{j\in \Lambda^{\leftarrow}(\mathcal{K})} \frac{\overline{\lambda}_j^2}{\lambda_j^2} \left(\frac{\lambda_k^2-\lambda^2_j} {\lambda_k^2-\overline{\lambda}_j^2}\right)^2,
    \end{equation}
    and denote by $(u_{\mathcal{K}}(t,x),v_{\mathcal{K}}(t,x))$ the solution of (\ref{e mtm}) with the modified scattering data. Then, as $t\to\infty$ in $\mathcal{K}(v_1,v_2,x_1,x_2)$ we have
    \begin{equation}\label{e u - u S}
        |u(x,t)-u_{\mathcal{K}}(t,x)|+ |v(x,t)-v_{\mathcal{K}}(t,x)| \leq ce^{-\theta t}
    \end{equation}
    with positive constants $C$ and $\theta$ not depending on $t$.
\end{thm}
\begin{remark}
  The formulae in (\ref{e p tilde Ck tilde}) can also be expressed in terms of the transformed sets of scattering data $\mathcal{S}_w(u_0,v_0) =(r, \{w_j,c_j\}_{j=1}^N)$ and $\mathcal{S}_z(u_0,v_0) =(\widehat{r},\{z_j,\widehat{c}_j\}_{j=1}^N)$. Namely, the new transformed data
  \begin{equation*}
    (\widetilde{r}; \{w_k,\widetilde{c}_k\} _{k\in\Lambda(\mathcal{K})}),\qquad
    (\widetilde{\hat{r}}; \{z_k,\widetilde{\hat{c}}_k\} _{k\in\Lambda(\mathcal{K})})
  \end{equation*}
  are given by
  \begin{equation}\label{e r tilde Ck tilde}
        \begin{aligned}
          &\widetilde{r}(w)=r(w)\prod_{j\in \Lambda^{\leftarrow}(\mathcal{K})} \frac{\overline{w}_j}{w_j} \left(\frac{w-w_j} {w-\overline{w}_j}\right)^2,&& \widetilde{c}_k=c_k \prod_{j\in \Lambda^{\leftarrow}(\mathcal{K})} \frac{\overline{w}_j}{w_j} \left(\frac{w_k-w_j} {w_k-\overline{w}_j}\right)^2,\\
          &\widetilde{\hat{r}}(z)=\hat{r}(z)\prod_{j\in \Lambda^{\leftarrow}(\mathcal{K})} \frac{\overline{z}_j}{z_j} \left(\frac{z-z_j} {z-\overline{z}_j}\right)^2,&& \widetilde{\hat{c}}_k=\hat{c}_k \prod_{j\in \Lambda^{\leftarrow}(\mathcal{K})} \frac{\overline{z}_j}{z_j} \left(\frac{z_k-z_j} {z_k-\overline{z}_j}\right)^2.
        \end{aligned}
  \end{equation}
  These expressions can be derived from the relations (\ref{e new reflection coeff}), (\ref{e rel new c_j}) and (\ref{e rel new poles}) and the following equalities
  \begin{equation*}
     \frac{\overline{\lambda}_j^2}{\lambda_j^2} \left(\frac{\lambda^2-\lambda^2_j} {\lambda^2-\overline{\lambda}_j^2}\right)^2=
    \frac{\overline{z}_j}{z_j} \left(\frac{z-z_j} {z-\overline{z}_j}\right)^2=\frac{\overline{w}_j}{w_j} \left(\frac{w-w_j} {w-\overline{w}_j}\right)^2,
  \end{equation*}
  that are true for any complex $\lambda^2=z=w^{-1}$ and $\lambda^2_j=z_j=w^{-1}_j$.
\end{remark}
\begin{center}
\begin{figure}

\begin{tikzpicture}

\draw[->] (-3.2,0) -- (3.2,0) node[below] {$x$};
\draw[->] (0,-0.2) -- (0,5.6) node[left] {$t$};
\draw[thick,domain=1.5:2.3] plot (\x,{6.31*(\x-1.5)}) ;
\draw[thick, domain=-3.15:0.5] plot (\x,{-1.38*(\x-0.5)}) ;
\draw[domain=-3.05:0, dashed] plot (\x,-\x) ;
\draw[domain=0:3.05, dashed] plot (\x,\x) ;
\draw [fill] (.5,0) circle [radius=0.05];
\draw [fill] (1.5,0) circle [radius=0.05];
\node at (.5,-.3) {$x_1$};
\node at (1.5,-.3) {$x_2$};
\node at (1,2) {$\mathcal{K}$};
\node at (2.3,6.31*.8+.3) {$x-v_2t=x_2$};
\node at (-3.15,-3.65*-1.38+.3) {$x-v_1t=x_1$};
\path[fill=black,opacity=0.21] (0.5,0) -- (1.5,0) --(2.3,6.31*.8) -- (-3.15,-3.65*-1.38) --(0.5,0);

\draw[->] (4.5,0) -- (10.7,0) node[below] {$\quad\Real(\lambda)$};
\draw[->] (10,-.1) -- (10,4.6) node[left] {$\Imag(\lambda)$};
\draw[thick] (10,4)  arc (90:180:4);
\draw[thick] (10,2.5)  arc (90:180:2.5);
\draw[thick] (10,4)  -- (10.07,4);
\draw[thick] (10,2.5)  -- (10.07,2.5);
\draw[thick] (7.5,0)  -- (7.5,-.07);
\draw[thick] (6,0)  -- (6,-.07);
\path[fill=black,opacity=0.21] (10,4)  arc (90:180:4) --(7.5,0) arc (180:90:2.5)--  (10,4);
\draw [fill] (7,1.8) circle [radius=0.03];
\node at (7,1.55) {$\lambda_1$};
\draw [fill] (6.3,3) circle [radius=0.03];
\node at (6.3,2.75) {$\lambda_2$};
\draw [fill] (7,4.2) circle [radius=0.03];
\node at (7,3.95) {$\lambda_9$};
\draw [fill] (8.5,1.5) circle [radius=0.03];
\node at (8.5,1.25) {$\lambda_4$};
\draw [fill] (9.5,1.25) circle [radius=0.03];
\node at (9.5,1) {$\lambda_5$};
\draw [fill] (8.35,3) circle [radius=0.03];
\node at (8.35,2.75) {$\lambda_7$};
\draw [fill] (9.5,3.7) circle [radius=0.03];
\node at (9.5,3.45) {$\lambda_6$};
\draw [fill] (5.2,.8) circle [radius=0.03];
\node at (5.2,.55) {$\lambda_3$};
\draw [fill] (5,4) circle [radius=0.03];
\node at (5,3.75) {$\lambda_8$};
\draw [fill] (5,2) circle [radius=0.03];
\node at (5,1.7) {$\lambda_{10}$};
\node at (6,.65-1) {$\sqrt{L_2}$};
\node at (7.5,.65-1) {$\sqrt{L_1}$};
\end{tikzpicture}
  \caption{Inside a cone $\mathcal{K}(v_1,v_2,x_1,x_2)$ the asymptotic behavior of $u(t,x)$ and $v(t,x)$ is preassigned by the reduced spectrum shaded on the right which only contains eigenvalues $\lambda_j$ such that $L_1\leq|\lambda_j|^2\leq L_2$ where $L_1$ and $L_2$ are numbers determined from $v_1$ and $v_2$, (\ref{e L-j}).\label{f solres}}
\end{figure}
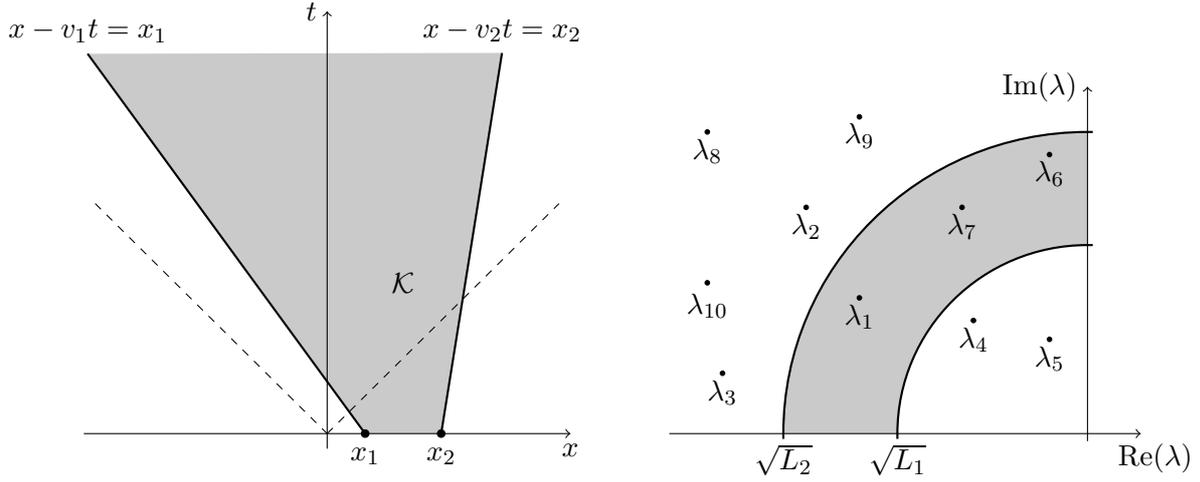
\end{center} 
The proof of Theorem \ref{t solres1} can be obtained from the proofs of Lemma 3.1 and 3.2 in \cite{Saalmann2015} applied to RHP's \ref{rhp M} and \ref{rhp M hat} separately. The technique of steepest descent is not used for this result, because the asymptotic functions $u_{\mathcal{K}}$ and $v_{\mathcal{K}}$ in (\ref{e u - u S}) still have a non-vanishing reflection coefficient $\widetilde{p}$ given by (\ref{e p tilde Ck tilde}) or, equivalently, (\ref{e r tilde Ck tilde}).
\\
Theorem \ref{t solres1} already implies the resolution of $(u,v)$ into the sum of solutions $(u_{\mathcal{K}_1},v_{\mathcal{K}_1}),..., (u_{\mathcal{K}_K},v_{\mathcal{K}_K})$ where the sets $\mathcal{K}_1,...,\mathcal{K}_K$ are supposed to have parameters $v^{(k)}_1,v^{(k)}_2$ sufficiently close together such that $\Lambda(\mathcal{K}_k)$ is either empty or $|\lambda_{j_1}|=|\lambda_{j_2}|$ for all $j_1,j_2\in\Lambda(\mathcal{K}_k)$. Applying Theorem \ref{t main} to those sets $\mathcal{K}_k$ for which $\Lambda(\mathcal{K}_k)=\emptyset$ we find $|u(t,x)|+|v(t,x)|=\mathcal{O}(t^{-1/2})$ in these sectors. Let us now state the complementing result which describes the behavior of $u$ and $v$ in any narrow cone $\mathcal{K}_k$ for which the contributing eigenvalues $\{\lambda_j:j\in\Lambda(\mathcal{S}_k)\}$ form a non-empty set.
\begin{thm}\label{t sol stab}
    Let $L_0\in\R^+$ and suppose that $u_0,v_0\in H^2(\R)\cap H^{1,1}(\R)$ with scattering data $$\mathcal{S}(u_0,v_0)=(p, \{\lambda_j,C_j\}_{j=1}^n)$$ and transformed scattering data
    \begin{equation*}
      \mathcal{S}_w(u_0,v_0) =(r,\{w_k,c_k\}_{k=1}^n),\qquad\mathcal{S}_z(u_0,v_0) =(\widehat{r},\{z_j,\widehat{c}_j\}_{j=1}^n).
    \end{equation*}
    such that
    \begin{equation}\label{e |w_1|=...=|w_n|}
        L_0=|\lambda_1|^2=...=|\lambda_n|^2.
    \end{equation}
    Define $\mathcal{D}=
    \{\lambda_j,\widetilde{C}_j\}_{j=1}^n$ by
    \begin{equation}\label{e def new C tilde j}
      \begin{aligned}
        \widetilde{C}_j &=&& C_j \exp\left\{ \frac{-1}{\pi i}\left(\int_{-\infty}^{-L_0^{-1}} + \int_{L_0^{-1}}^{\infty}\right) \log(1+w|r(w)|^2) \left(\frac{1}{w-w_j}-\frac{1}{2w}\right)dw\right\} \\
       &= &&C_j  \exp\left\{ \frac{1}{\pi i}\int_{-L_0}^{L_0} \log(1+z|\hat{r}(z)|^2) \left(\frac{1}{z-z_j}-\frac{1}{2z}\right)dz\right\}
      \end{aligned}
    \end{equation}
    Then using the notation as in Definition \ref{d solitons} we have for $z_0=w_0^{-1}\in(L_0-\eps,L_0+\eps)$ ($\eps$ sufficiently small) that
    \begin{equation}\label{e u - u^br}
        |u(x,t)-u_{sol}(t,x;\mathcal{D})|+ |v(x,t)-v_{sol}(t,x;\mathcal{D})| \leq ct^{-1/2}
    \end{equation}
    The constant $c$ depends on the initial data through $\mathcal{S}(u_0,v_0)$ and on $\epsilon$.
\end{thm}
\begin{proof}
The proof is obtained by repeating the proof of Theorem \ref{t main thm general}. What makes the difference is the presence of singularities in the RHP's. If RHP \ref{rhp M} admits poles at $w_1,...,w_n,\overline{w}_1,...,\overline{w}_n$ (recall $w_j=\lambda_j^{-2}$), then the function $M^{(0)}(\zeta;\tau)$ defined in terms of $M$ by (\ref{e M <-> M^0}) will admit poles at $\zeta_1,...,\zeta_n,\overline{\zeta}_1 ,...,\overline{\zeta}_n$, where $\zeta_j=w_j/w_0$. The precise residuum relations that we have to add to RHP \ref{rhp M^0} are the following:
\begin{equation}\label{e res M^0}
    \begin{aligned}
        \res_{\zeta=\zeta_j}M^{(0)}(\tau;\zeta) &=\lim_{\zeta\to \zeta_j}M^{(0)}(\tau;\zeta)
        \left[
          \begin{array}{cc}
            0 & 0 \\
            \frac{\zeta_jc_j}{d(w_j)^2}e^{-i\tau Z(\zeta_j)} & 0 \\
          \end{array}
        \right],\\
        \res_{\zeta=\overline{\zeta}_j} M^{(0)}(\tau;\zeta)&=\lim_{\zeta\to \overline{\zeta}_j} M^{(0)}(\tau;\zeta)
        \left[
          \begin{array}{cc}
            0 &
            \frac{-\overline{c}_jd(\overline{w}_j)^2}{w_0} e^{i\tau Z(\overline{\zeta}_j)}\\
            0&0 \\
          \end{array}
        \right].\\
    \end{aligned}
\end{equation}
Assuming
\begin{equation}\label{e w_0 approx L_0}
    |L_0^{-1}-w_0|\leq \tau^{-1/2}
\end{equation}
and using (\ref{e |w_1|=...=|w_n|}) we know that the poles $\zeta_1,...,\zeta_n$ lie in the region $\Omega_9$, see Figure \ref{f Omega1-8}. This latter fact is useful because it guarantees that the modifications $M^{(0)}\mapsto M^{(1)}\mapsto M^{(2)}$ given by the explicit formulas (\ref{e def M^1}) and (\ref{e def M^2}) lead to a matrix valued function $M^{(2)}$ which is still meromorphic around $\zeta_1,...,\zeta_n,\overline{\zeta}_1,...,\overline{\zeta}_n$. This is because $\mathcal{W}=1$ on $\Omega_9\cup\Omega_{10}$ In fact, $M^{(2)}$ is a solution for the mixed $\db$-RHP \ref{dbrhp M^2} amended by
\begin{equation}\label{e res M^2}
    \begin{aligned}
        \res_{\zeta=\zeta_j}M^{(2)}(\tau;\zeta) &=\lim_{\zeta\to \zeta_j}M^{(2)}(\tau;\zeta)
        \left[
          \begin{array}{cc}
            0 & 0 \\
            \frac{\zeta_jc_j\delta(\zeta_j)^2}{d(w_j)^2} e^{-i\tau Z(\zeta_j)} & 0 \\
          \end{array}
        \right],\\
        \res_{\zeta=\overline{\zeta}_j} M^{(2)}(\tau;\zeta)&=\lim_{\zeta\to \overline{\zeta}_j} M^{(2)}(\tau;\zeta)
        \left[
          \begin{array}{cc}
            0 &
            \frac{-\overline{c}_jd(\overline{w}_j)^2}{w_0 \delta(\overline{\zeta}_j)^2}e^{i\tau Z(\overline{\zeta}_j)}\\
            0&0 \\
          \end{array}
        \right].\\
    \end{aligned}
\end{equation}
In contrast to the soliton-free case where we use decomposition (\ref{e equ M^2=DM^3}) to seperate the $\db$-part and the RHP-part of $M^{(2)}$, here we need the following decomposition:
\begin{equation}\label{e equ M^2=D M^mer M^3}
    M^{(2)}(\tau;\zeta)=D(\tau;\zeta) M^{(mer)}(\tau;\zeta)M^{(3)}(\tau;\zeta) .
\end{equation}
The idea is to present the $\db$-component of $M^{(2)}$ by $D$, the singularities by $M^{(mer)}$  and the RHP part by $M^{(3)}(\tau;\zeta)$. To be more precise let us specify:
\begin{itemize}
  \item[(i)] $M^{(3)}$ is supposed to be the solution for RHP($\Sigma^{(3)}$,$R_{\tau}^{(3)}$) with $R_{\tau}^{(3)}$ given in (\ref{e def R^3 Sigma^3}). In particular it is possible to apply Proposition \ref{p M^3 approx M^4}.
  \item[(ii)] $M^{(mer)}(\tau;\zeta)$ is supposed to be a function meromorphic on $\C$ with singularities at $\zeta_1,...,\zeta_n,\overline{\zeta}_1, ...,\overline{\zeta}_n$ satisfying
      \begin{equation}\label{e res M^mer}
        \begin{aligned}
        \res_{\zeta=\zeta_j}M^{(mer)}(\tau;\zeta) &=\lim_{\zeta\to \zeta_j}M^{(mer)}(\tau;\zeta) [M^{(3)}(\tau;\zeta)]^{-1}
        \left[
          \begin{array}{cc}
            0 & 0 \\
            \frac{\zeta_jc_j\delta(\zeta_j)^2}{d(w_j)^2} e^{-i\tau Z(\zeta_j)} & 0 \\
          \end{array}
        \right]M^{(3)}(\tau;\zeta),\\
        \res_{\zeta=\overline{\zeta}_j} M^{(mer)}(\tau;\zeta)&=\lim_{\zeta\to \overline{\zeta}_j} M^{(mer)}(\tau;\zeta) [M^{(3)}(\tau;\zeta)]^{-1}
        \left[
          \begin{array}{cc}
            0 &
            \frac{-\overline{c}_jd(\overline{w}_j)^2}{w_0 \delta(\overline{\zeta}_j)^2}e^{i\tau Z(\overline{\zeta}_j)}\\
            0&0 \\
          \end{array}
        \right]M^{(3)}(\tau;\zeta).\\
    \end{aligned}
      \end{equation}
  \item[(ii)] $D(\tau;\zeta)$ is supposed to be the solution of $\db$-problem \ref{db} with
      \begin{equation}\label{e def Upsilon mer}
          \Upsilon(\tau;\zeta)=M^{(3)}(\tau;\zeta) M^{(mer)}(\tau;\zeta) \db\mathcal{W}(\tau;\zeta) \left[M^{(3)}(\tau;\zeta) M^{(mer)}(\tau;\zeta)\right]^{-1}
      \end{equation}
      replacing the expression given for $\Upsilon$ in (\ref{e db D}). Since $\|M^{(mer)}(\tau;\cdot)\|_{L^{\infty}(\C\setminus (\Omega_9\cup\Omega_{10}))}\sim \mathcal{O}(1)$, all estimates for $\Upsilon$ that are presented in the analysis of (\ref{e def J}) also hold for the case (\ref{e def Upsilon mer}). As a consequence we can use Lemma \ref{l existence and asymptotic of db problem} for our analysis of $M^{(2)}$ through (\ref{e equ M^2=D M^mer M^3}).
\end{itemize}
Let us recall the definition (\ref{e def q^0}) of $q^{(0)}(\tau)$. By the explicit transformation formulae (\ref{e def M^1}) and (\ref{e def M^2}), by Proposition \ref{p M^3 approx M^4} and by Lemma \ref{l existence and asymptotic of db problem} it follows that
\begin{equation}\label{e M^0 approx M^mer}
   \begin{aligned}
      \left|q^{(0)}(\tau)-\lim_{\zeta\to\infty} \zeta\left[M^{(mer)}(\tau;\zeta)\right]_{12}\right|&\leq c \tau^{-1/2}\\
      \left|M^{(0)}(\tau,0)- M^{(mer)}(\tau,0) [\delta(0)]^{-\sigma_3}\right|&\leq c \tau^{-1/2}.
   \end{aligned}
\end{equation}
We learn that the remaining part of the proof is to analyze $M^{(mer)}$ but we also notice that (\ref{e res M^mer}) does not describe the residuum condition of soliton solutions of the MTM system because, for instance, in the first line of (\ref{e res M^mer}) the term
\begin{equation*}
    [M^{(3)}(\tau;\zeta)]^{-1}
        \left[
          \begin{array}{cc}
            0 & 0 \\
            \frac{\zeta_jc_j\delta(\zeta_j)^2}{d(w_j)^2} e^{-i\tau Z(\zeta_j)} & 0 \\
          \end{array}
        \right]M^{(3)}(\tau;\zeta)
\end{equation*}
is not of the same form as the original one in (\ref{e res M^0}). However, we have $M^{(3)}(\tau;\zeta)\to 1$ as $\tau\to\infty$, see (\ref{e M^3(zeta0)}), and also $\delta(\zeta_j)=\Delta_j + \mathcal{O}(\tau^{-1/2})$ for
\begin{equation}\label{e def Delta j}
  \Delta_j:= \exp\left\{ \frac{1}{2\pi i}\int_{-L_0^{-1}}^{L_0^{-1}} \frac{\log(1+w|r(w)|^2) }{w-w_j}dw \right\},
\end{equation}
which is a consequence of our assumtion (\ref{e w_0 approx L_0}). Then, by the arguments of Lemmas 6.1 and 6.3 in \cite{Saalmann2015} we obtain the following result: the meromorphic function $M^{(sol-0)}(\tau;\zeta)$ with singularities at $\zeta_1,...,\zeta_n,\overline{\zeta}_1, ...,\overline{\zeta}_n$ satisfying
\begin{equation}\label{e res M^sol}
        \begin{aligned}
        \res_{\zeta=\zeta_j}M^{(sol-0)}(\tau;\zeta) &=\lim_{\zeta\to \zeta_j}M^{(sol-0)}(\tau;\zeta)
        \left[
          \begin{array}{cc}
            0 & 0 \\
            \frac{\zeta_jc_j\Delta_j^2}{d(w_j)^2} e^{-i\tau Z(\zeta_j)} & 0 \\
          \end{array}
        \right],\\
        \res_{\zeta=\overline{\zeta}_j} M^{(sol-0)}(\tau;\zeta)&=\lim_{\zeta\to \overline{\zeta}_j} M^{(sol-0)}(\tau;\zeta)
        \left[
          \begin{array}{cc}
            0 &
            \frac{-\overline{c}_jd(\overline{w}_j)^2}{w_0 \delta(\overline{w}_j/L_0)^2}e^{i\tau Z(\overline{\zeta}_j)}\\
            0&0 \\
          \end{array}
        \right],\\
    \end{aligned}
\end{equation}
has, under the assumption (\ref{e w_0 approx L_0}), the following properties:
\begin{equation}\label{e M^mer approx M^sol-0}
      \begin{array}{c}
         \displaystyle\left|\lim_{\zeta\to\infty} \zeta\left[M^{(mer)}(\tau;\zeta)\right]_{12} -\lim_{\zeta\to\infty} \zeta\left[M^{(sol-0)}(\tau;\zeta) \right]_{12}\right|\leq c \tau^{-1/2},\\
         \displaystyle \left|M^{(mer)}(\tau,0) -M^{(sol-0)}(\tau,0)\right|\leq c \tau^{-1/2}.\\
      \end{array}
\end{equation}
Now, let us go back to the original coordinates $t$ and $x$. We recall formula (\ref{e M <-> M^0}) and set analogously
\begin{equation*}
    M^{(sol)}(t,x;w):=M^{(sol-0)}(\tau;w/w_0).
\end{equation*}
From (\ref{e res M^sol}) it follows that $M^{(sol)}(t,x;w)$ satisfies exactly the RHP for the multi-soliton $u_{sol}(t,x;\mathcal{D}')$ with data $\mathcal{D}'=\{\lambda_j,C_j'\}_{j=1}^n$ where
\begin{equation*}
    C'_j=C_j\frac{\Delta_j^2}{d(w_j)^2}.
\end{equation*}
Moreover, by combining (\ref{e M^0 approx M^mer}) and (\ref{e M^mer approx M^sol-0}) and using the reconstruction formulaa (\ref{e rec u 2})--(\ref{e M(0)}) we find
\begin{eqnarray*}
  u(t,x) &=& \left[M(t,x;w)\right]_{11} \overline{\lim_{|w|\to\infty}w \cdot\left[M(t,x;w)\right]_{12}} \\
   &=& \frac{d(0)}{\delta(0)} [M^{(sol)}(t,x;0)]_{11} \overline{\lim_{w\to\infty}w\cdot [M^{(sol)}(t,x;0)]_{12}}+ \mathcal{O}(\tau^{-1/2})\\
   &=&\frac{d(0)}{\delta(0)}u_{sol}(t,x;\mathcal{D}')+ \mathcal{O}(t^{-1/2}).
\end{eqnarray*}
Similarly to (\ref{e def Delta j}) we find
$\delta(0)=\Delta_0 + \mathcal{O}(\tau^{-1/2})$ for
\begin{equation}\label{e def Delta 0}
  \Delta_0:= \exp\left\{ \frac{1}{2\pi i}\int_{-L_0^{-1}}^{L_0^{-1}} \frac{\log(1+w|r(w)|^2) }{w}dw \right\}.
\end{equation}
We mention that (\ref{e w_0 approx L_0}) is again a necessary condition. So far we have $u(t,x) =\frac{d(0)}{\Delta_0}u_{sol}(t,x;\mathcal{D}')+ \mathcal{O}(t^{-1/2})$. Making additionally use of Remark \ref{r e^i alpha u}, we end up with $u(t,x) =u_{sol}(t,x;\mathcal{D})+ \mathcal{O}(t^{-1/2})$. Here the scattering data are given by  $\mathcal{D}=\{\lambda_j,\widetilde{C}_j\}_{j=1}^n$ with modified norming constants
\begin{equation*}
  \widetilde{C}_j =C'_j\frac{d(0)}{\Delta_0} =C_j\frac{d(0)}{\Delta_0}\frac{\Delta_j^2}{d(w_j)^2}.
\end{equation*}
It is left to the reader to verify that the latter expression is equivalent to both lines in (\ref{e def new C tilde j}).

Repeating the above line of argument for RHP \ref{rhp M hat} it follows that
\begin{equation*}
    v(t,x)=v_{sol} (t,x;\widetilde{D})+ \mathcal{O}(t^{-1/2})
\end{equation*}
with the same scattering data as for $u$. Thus, the remaining part of the proof is a discussion of the region where (\ref{e w_0 approx L_0}) does not hold. Equivalently, the region, where $|L_0^{-1}-w_0|>\tau^{-1/2}$. It is a fact that in those regions $(u(t,x),v(t,x))$ behaves like a pure radiation solution and thus $|u(t,x)|+|v(t,x)|\leq c \tau^{-1/2}$ by Theorem \ref{t main}. The proof of this fact can be provided with arguments similar to Theorem \ref{t solres1}. In particular, if $|L_0^{-1}-w_0|>\tau^{-1/2}$, then $|u_{sol} (t,x;\widetilde{D})|+|v_{sol} (t,x;\widetilde{D})|\leq c \tau^{-1/2}$ and we can conclude (\ref{e u - u^br}).
\end{proof} 
\appendix
\section{Existence theory for RHPs}\label{a RHP theory}
The following is presented for the convenience of the reader  even though it does not contain something new.
Let $\Sigma\subset\C$ be a finite union of smooth curves equipped with an orientation. We will call such objects oriented contours. To each contour $\Sigma$ we can associate the Cauchy operator
\begin{equation*}
    C_{\Sigma}[f](\zeta):=\frac{1}{2\pi i}\int_{\Sigma}\frac{f(s)}{s-\zeta}ds,\quad \zeta\in\C\setminus\Sigma.
\end{equation*}
For $\zeta\in\Sigma$ it is a fact that $C_{\Sigma}[f](\zeta')$ can approach different values as $\zeta'\to\zeta$, depending on the side of $\Sigma$ on which the limit is taken. If one moves along the contour in the direction of the orientation, it is a convention to say that the $\oplus$-side lies to the left. The $\ominus$-side lies to the right, respectively. See Figure \ref{f contour}
\begin{figure}
\centering
\vskip 15pt
\begin{tikzpicture}
    \draw[->,thick] (1,5) -- (2,5);
    \draw[-,thick] (2,5) -- (3,5);
    \draw[->,thick] (3,5) -- (3.25,6);
    \draw[-,thick] (3.25,6) -- (3.5,7);
    \draw[->,thick] (3,5) -- (3.5,4);
    \draw[-,thick] (3.5,4) -- (4,3);
    \draw[-,thick] (4,3)-- (5,1);
    \draw[->,thick] (7,0).. controls (6,0) and (5.5,0).. (5,1);
    \draw[->,thick] (1,0).. controls (1.1,.2) and (1,1).. (2,2);
    \draw[->,thick] (2,2).. controls (3,3) and (5,2.7).. (7,3.5);
    \draw[-,thick] (7,3.5) -- (8,3.9);
    \node at (1.8,2.25) {$\oplus$};
    \node at (2,5.3) {$\oplus$};
    \node at (2.9,6.2) {$\oplus$};
    \node at (3.8,4.1) {$\oplus$};
    \node at (6,3.5) {$\oplus$};
    \node at (4.7,.7) {$\oplus$};
    \node at (2.25,1.75) {$\ominus$};
    \node at (2,4.7) {$\ominus$};
    \node at (3.6,6) {$\ominus$};
    \node at (3.1,3.8) {$\ominus$};
    \node at (6.2,2.8) {$\ominus$};
    \node at (5.4,1.1) {$\ominus$};
\end{tikzpicture}
\caption{}\label{f contour}
\end{figure} 
for an example. This gives rise to the following definition.
\begin{equation}\label{e def Cauchy proj op}
    C^{+}_{\Sigma}[f](\zeta) :=\lim_{\substack{\zeta'\to\zeta\\ \zeta'\in \oplus\text{-side}\\ \text{ of }\Sigma}}C_{\Sigma}[f](\zeta'),\qquad
    C^{-}_{\Sigma}[f](\zeta) :=\lim_{\substack{\zeta'\to\zeta\\ \zeta'\in \ominus\text{-side}\\ \text{ of }\Sigma}}C_{\Sigma}[f](\zeta'),\qquad \zeta\in \Sigma.
\end{equation}

\begin{prop}\label{p Cauchy operator}
    \begin{itemize}
      \item [(i)] For every $f\in L^p(\Sigma)$, $1\leq p <\infty$, $\zeta\mapsto C_{\Sigma}[f](\zeta)$ is analytic for $\zeta\in\C\setminus \Sigma$ and satisfies
          \begin{equation}\label{e expansion C[f]}
            \lim_{\substack{|\zeta|\to \infty\\
            \zeta\in\C\setminus \Sigma}} \zeta \cdot C_{\Sigma}[f](\zeta)=-\frac{1}{2\pi i}\int_{\Sigma}f(s)ds
          \end{equation}
      \item [(ii)] For $f\in L^p(\Sigma)$, $1\leq p <\infty$, the values $ C^{\pm}_{\Sigma}[f](\zeta)$ exist for almost every $\zeta\in\Sigma$.
      \item [(iii)] If $1< p <\infty$, then there exists a positive constant $C_p$ such that
          \begin{equation}\label{e Cauchy operator cont}
            \|C^{\pm}_{\Sigma}[f]\|_{L^p(\Sigma)}\leq C_p\|f\|_{L^p(\Sigma)}.
          \end{equation}
      \item [(iv)] (Sokhotski-Plemelj theorem) The following relation holds:
          \begin{equation}\label{e Sokhotski-Plemelj theorem}
            C^{\pm}_{\Sigma}[f](\zeta)= \pm \frac{1}{2} f(\zeta)-\frac{i}{2}\mathcal{H}[f](\zeta),\quad \zeta\in\Sigma,
          \end{equation}
          where the Hilbert transform $\mathcal{H}$ is given by
          \begin{equation*}
            \mathcal{H}[f](\zeta):=\frac{1}{\pi} \lim_{\eps\searrow 0} \int_{\Sigma\setminus B_{\eps}(\zeta)}\frac{f(s)}{s-\zeta}ds,\quad \zeta\in\Sigma.
          \end{equation*}
    \end{itemize}
\end{prop}

\begin{thm}\label{t rhp theory}
    For any given contour $\Sigma\subset \C$, there exists a constant $\Lambda_{\Sigma}$ such that for all functions $R:\Sigma\to \C^{2\times 2}$ satisfying $\det(1+R)\equiv 1$, $R\in L^1(\Sigma)\cap L^{\infty}(\Sigma)$ and
    \begin{equation}\label{e R infty norm}
        \|R\|_{L^{\infty}(\Sigma)}\leq \Lambda_{\Sigma},
    \end{equation}
    the correspondent Riemann-Hilbert problem RHP($\Sigma$,$R$) is uniquely solvable. Moreover, there exists another constant $c_{\Sigma}$ such that for the solution $M$ of RHP($\Sigma$,$R$) we have
    \begin{equation}\label{e lim zeta M(zeta)}
        \left|\lim_{|\zeta|\to \infty}\zeta\cdot\left( M(\zeta)-1\right)\right|\leq c_{\Sigma} \|R\|_{L^1(\Sigma)}
    \end{equation}
    and
    \begin{equation}\label{e M(zeta0)}
        |M(\zeta_0)-1| \leq \frac{c_{\Sigma}}{\dist(\Sigma,\zeta_0)} \|R\|_{L^1(\Sigma)},\qquad \zeta_0\in\C\setminus \Sigma,
    \end{equation}

\end{thm}
\begin{proof}
    By (\ref{e Cauchy operator cont}) and (\ref{e R infty norm}),
    \begin{equation*}
        C_{\Sigma,R}[f]:=C_{\Sigma}^-[f\cdot R]
    \end{equation*}
    defines a bounded operator $C_{\Sigma,R}:L^2(\Sigma,\C^{2\times 2})\to L^2(\Sigma,\C^{2\times 2})$. It satisfies $\|C_{\Sigma,R}\|_{L^2\to L^2}\leq c'_{\Sigma}\|R\|_{L^{\infty}(\Sigma)}$ and thus, if $\|R\|_{L^{\infty}(\Sigma)}$ is sufficiently small we know, that $(1-C_{\Sigma,R})^{-1}$ exists as a bounded operator $L^2(\Sigma,\C^{2\times 2})\to L^2(\Sigma,\C^{2\times 2})$ with $\|(1-C_{\Sigma,R})^{-1}\|_{L^2\to L^2}\leq c \|R\|_{L^{\infty}(\Sigma)}$ with some constant $c>0$ not depending on $\|R\|_{L^{\infty}(\Sigma)}$ if (\ref{e R infty norm}) holds for some sufficiently small $\Lambda_{\Sigma}$. By assumption we have $R\in L^2(\Sigma)$ and thus $C_{\Sigma}^-[R]\in L^2(\Sigma)$ with $\|C_{\Sigma}^-[R]\|_{L^2(\Sigma)}\leq c'_{\Sigma}\|R\|_{L^2(\Sigma)}$. Thus, for $\mu:=(1-C_{\Sigma,R})^{-1}C_{\Sigma}^-[R]$ we find
    \begin{equation*}
        \|\mu\|_{L^2(\Sigma)}\leq c \|R\|_{L^{\infty}(\Sigma)}\|R\|_{L^2(\Sigma)}.
    \end{equation*}
    Now we claim that the solution $M$ of Riemann-Hilbert problem RHP($\Sigma$,$R$) is precisely given by
    \begin{equation}\label{e general RHP solution formula}
        M(\zeta)=1+
        \frac{1}{2\pi i} \int_{\Sigma} \frac{(\mu(s)+1)R(s)}{s-\zeta}ds.
    \end{equation}
    Assuming (\ref{e general RHP solution formula}) for a moment, we can prove (\ref{e lim zeta M(zeta)}) as follows:
    \begin{eqnarray*}
      \left|\lim_{|\zeta|\to \infty}\zeta\cdot\left( M(\zeta)-1\right)\right| &=& \frac{1}{2\pi} \left|\int_{\Sigma} (\mu(s)+1)R(s)ds\right| \\
       &\leq & \frac{1}{2\pi}\left(\|\mu\|_{L^2}\|R\|_{L^2} + \|R\|_{L^1} \right) \\
       &\leq& c \,\left(\|R\|_{L^{\infty}}\|R\|^2_{L^2} + \|R\|_{L^1} \right)  \\
       &\leq&  c_{\Sigma}  \|R\|_{L^1}.
    \end{eqnarray*}
    The last inequality follows from standard inclusion $L^2\subset L^1\cap L^{\infty}$ and (\ref{e R infty norm}). In a very similar way we can prove (\ref{e M(zeta0)}).
    In order to understand why (\ref{e general RHP solution formula}) is indeed a solution formula for RHP($\Sigma$,$R$) we first note that $\mu= C^-_{\Sigma}[(\mu+1)R]$ by defintion. Next let us denote the right hand side of (\ref{e general RHP solution formula}) by $\widetilde{M}$ such that $\widetilde{M}=1+C_{\Sigma}[(\mu+1)R]$. We obviously have  $\widetilde{M}_-=1+\mu$. Furthermore, using (\ref{e Sokhotski-Plemelj theorem}) we find
    \begin{eqnarray*}
      \widetilde{M}_+&=&1+C_{\Sigma}^+[(\mu+1)R]\\
       &=& 1+ (\mu+1)R+C_{\Sigma}^-[(\mu+1)R] \\
       &=& 1+ (\mu+1)R+\mu \\
       &=& (\mu+1)(1+R)\\
       &=&\widetilde{M}_-(1+R).\\
    \end{eqnarray*}
    By Proposition \ref{p Cauchy operator} (i) we also have analyticity of $\widetilde{M}$ on $\C\setminus\Sigma$ and $\widetilde{M}(\zeta)=1+\mathcal{O}(\zeta^{-1})$ as $\zeta\to\infty$. Hence, $\widetilde{M}$ is a solution of RHP($\Sigma$,$R$).
\end{proof} 
\section{Proof of (\ref{e |f pm|^2}) and (\ref{e arg(f pm)})}\label{a forumlae} We define
\begin{equation}\label{e kappa kappa hat}
    \widehat{\kappa}(z)=\frac{1}{2\pi}\log(1+z|\widehat{r}(z)|^2),
    \qquad
    \kappa(w)=\frac{1}{2\pi}\log(1+w|r(w)|^2),
\end{equation}
such that $\widehat{\kappa}$ coincides with that one defined in the introduction, (\ref{e kappa z_0}). Substituting (\ref{e def rho from r}) into the definitions (\ref{e notation 1})--(\ref{e notation 3}) we can use (\ref{e def q^as}), (\ref{e M^0(0)}), (\ref{e asymptotics q}) and Lemma \ref{l M and M hat <-> M^0} to find
\begin{eqnarray*}
  u(t,x) &=& w_0 \overline{q^{(as)}(\tau)}\left[M(t,x;0)\right]_{11} d(0) +\mathcal{O}(\tau^{-3/4})\\
   &=& \sqrt{\frac{w_0}{\tau}}\left( e^{i\tau-i\kappa(-w_0) \log(\tau)}b_-(w_0)-e^{i\tau-i\kappa(w_0) \log(\tau)}b_+(w_0)\right)+\mathcal{O}(\tau^{-3/4})
\end{eqnarray*}
with
\begin{equation*}
    \begin{aligned}
        \arg(b_{\pm}(w_0)) =&\mp \frac{\pi}{4}+\arg(r(\pm w_0))-\arg(d_-(\pm w_0)d_+(\pm w_0)) +\arg(\Gamma(\mp i\kappa(\pm w_0))\\
        &\mp 2\int_{0}^{\pm w_0}\frac{\kappa(s)\mp \frac{s}{w_0} \kappa(\pm w_0)}{s\mp w_0}ds \pm 2 \int_0^{\mp w_0}\frac{\kappa(s)}{s\mp w_0} ds
        \mp \kappa(\pm w_0)\\&+\int_{-w_0}^{w_0} \frac{\kappa(s)}{s}ds-\int_{-\infty}^{\infty} \frac{\kappa(s)}{s}ds
    \end{aligned}
\end{equation*}
and
\begin{eqnarray*}
  |b_{\pm}(w_0)|^2 &=& \left|\frac{\sqrt{2\pi} d_-(\pm w_0)d_+(\pm w_0)e^{\pi \kappa(\pm w_0)/2}}{\sqrt{w_0} \widehat{r}(\pm w_0) \Gamma(\mp i\kappa(\pm w_0))}\right|^2 \\
   &=& \frac{ \widehat{\kappa}(\pm w_0)}{w_0 |\widehat{r}(\pm w_0)|^2 } (e^{2\pi \kappa(\pm w_0)}-1) |d_-(\pm w_0)d_+(\pm w_0)|^2\\
   &=&\pm \widehat{\kappa}(\pm w_0) |d_-(\pm w_0)d_+(\pm w_0)|^2
\end{eqnarray*}
where the identity $|\Gamma(\pm i\kappa)|^2=\pi/(\kappa \sinh(\pi\kappa))$ is useful to obtain the second equality.
Analogously, substituting (\ref{e def rho from r hat}) into the definitions (\ref{e notation 1})--(\ref{e notation 3}) we can use (\ref{e def q^as}), (\ref{e M^0(0)}), (\ref{e asymptotics q}) and Lemma \ref{l M and M hat <-> M^0} to find
\begin{eqnarray*}
  v(t,x) &=& z_0 \overline{q^{(as)}(\tau)}\left[\widehat{M}(t,x;0)\right]_{11} +\mathcal{O}(\tau^{-3/4})\\
   &=& \sqrt{\frac{z_0}{\tau}}\left( e^{i\tau-i\widehat{\kappa}(-z_0) \log(\tau)}\widehat{b}_-(z_0)-e^{i\tau-i\widehat{\kappa}(z_0) \log(\tau)}\widehat{b}_+(z_0)\right)+\mathcal{O}(\tau^{-3/4})
\end{eqnarray*}
with
\begin{equation*}
    \begin{aligned}
        \arg\left(\widehat{b}_{\pm}(z_0)\right) =&\mp \frac{\pi}{4}+\arg(\widehat{r}(\pm z_0)) +\arg(\Gamma(\mp i\widehat{\kappa}(\pm z_0))\\
        &\mp 2\int_{0}^{\pm z_0}\frac{\widehat{\kappa}(s)\mp \frac{s}{z_0} \widehat{\kappa}(\pm z_0)}{s\mp z_0}ds \pm 2 \int_0^{\mp z_0}\frac{\widehat{\kappa}(s)}{s\mp z_0} ds
        \mp \widehat{\kappa}(\pm z_0)+\int_{-z_0}^{z_0} \frac{\widehat{\kappa}(s)}{s}ds
    \end{aligned}
\end{equation*}
and
\begin{equation*}
  |\widehat{b}_{\pm}(z_0)|^2 =\pm \widehat{\kappa}(\pm z_0)
\end{equation*}
\begin{lem}
    We have $b_{\pm}(w_0)=\widehat{b}_{\pm}(z_0)$. In particular, since $w_0=z_0^{-1}$ is constant along rays with $x/t=const.$, it is possible to define the function $f_{\pm}$ of Theorem \ref{t main} through
    \begin{equation*}
        f_{\pm}\left(\frac{x}{t}\right)= b_{\pm}(w_0)=\widehat{b}_{\pm}(z_0).
    \end{equation*}
\end{lem}
\begin{proof}
    The first important observation is
    \begin{equation*}
        \kappa(z^{-1})=\widehat{\kappa}(z)
    \end{equation*}
    and follows directly from relation (\ref{e relation r and r hat}). In particular we have
    \begin{equation}\label{e kappa(z_0)}
        \kappa(\pm w_0)=\widehat{\kappa}(\pm z_0).
    \end{equation}
    Next, let us recall that by (\ref{e def d}) we can write
    \begin{equation*}
        d_-(w)d_+(w)= \exp\{C_{\R}^+[\kappa](w)+C_{\R}^-[\kappa](w)\}.
    \end{equation*}
    Thus, using (\ref{e Sokhotski-Plemelj           theorem}) it follows that
    \begin{equation}\label{e |d_- d_+|}
        |d_-(w)d_+(w)|=1, \qquad\text{for all }w\in\R,
    \end{equation}
    and
    \begin{equation}\label{e arg(d-d+)}
        -\arg(d_-(\pm w_0)d_+(\pm w_0))=
        2\lim_{\eps\searrow 0}\left(\int_{-\infty}^{\pm w_0-\eps}+\int_{\pm w_0+\eps}^{\infty}\right)\frac{\kappa(s)}{s\mp w_0}ds.
    \end{equation}
    Equality (\ref{e kappa(z_0)}) together with (\ref{e |d_- d_+|}) yields $|b_{\pm}(w_0)|=|\widehat{b}_{\pm}(z_0)|$. Formula (\ref{e arg(d-d+)}) can be used to calculate the following:
    \begin{multline*}
        -\arg\left(d_-(\pm w_0)d_+(\pm w_0)\right) \mp 2\int_{0}^{\pm w_0}\frac{\kappa(s)\mp \frac{s}{w_0} \kappa(\pm w_0)}{s\mp w_0}ds \pm 2 \int_0^{\mp w_0}\frac{\kappa(s)}{s\mp w_0} ds\\
        \begin{aligned}
            &=2\lim_{\eps\searrow 0}\left[\left(\int_{-\infty}^{\pm w_0-\eps}+\int_{\pm w_0+\eps}^{\infty}\right)\frac{\kappa(s)}{s\mp w_0}ds\mp\int_{\mp w_0}^{\pm w_0\mp\eps}\frac{\kappa(s)}{s\mp w_0}ds+ \int_{0}^{\pm w_0\mp\eps}\frac{ \frac{s}{w_0} \kappa(\pm w_0)}{s\mp w_0}ds\right]\\
            &=\pm2\int_{\mp\infty}^{\mp w_0}\frac{\kappa(s)}{s\mp w_0}ds\pm 2\int_{\pm w_0}^{\pm \infty}\frac{\kappa(s)\mp \frac{w_0}{s} \kappa(\pm w_0)}{s\mp w_0}ds\\
            &\quad\quad +2\lim_{\eps\searrow 0} \left[\int_{\pm w_0\pm\eps}^{\pm \infty}\frac{ \frac{w_0}{s} \kappa(\pm w_0)}{s\mp w_0}ds+ \int_{0}^{\pm w_0\mp\eps}\frac{ \frac{s}{w_0} \kappa(\pm w_0)}{s\mp w_0}ds\right]\\
            &=\pm2\int_{\mp\infty}^{\mp w_0}\frac{\kappa(s)}{s\mp w_0}ds\pm 2\int_{\pm w_0}^{\pm \infty}\frac{\kappa(s)\mp \frac{w_0}{s} \kappa(\pm w_0)}{s\mp w_0}ds+2\kappa(\pm w_0)\\
            &=\mp2\int_{0}^{\mp z_0}\frac{\widehat{\kappa}(s)} {s}-\frac{\widehat{\kappa}(s)} {s\mp z_0}ds\mp 2\int_{\pm z_0}^{0}\frac{\widehat{\kappa}(s)}{s} -\frac{\widehat{\kappa}(s)\mp \frac{s}{z_0} \kappa(\pm z_0)}{s\mp z_0}ds\\
            &=\mp 2\int_{0}^{\pm z_0}\frac{\widehat{\kappa}(s)\mp \frac{s}{z_0} \widehat{\kappa}(\pm z_0)}{s\mp z_0}ds \pm 2 \int_0^{\mp z_0}\frac{\widehat{\kappa}(s)}{s\mp z_0} ds+2\int_{-z_0}^{ z_0}\frac{\widehat{\kappa}(s)} {s}ds.\\
        \end{aligned}
    \end{multline*}
    Using, $\arg(r(\pm w_0))=\arg(\widehat{r}(\pm z_0))$,
    \begin{equation*}
        \int_{-w_0}^{w_0} \frac{\widehat{\kappa}(s)}{s}ds-\int_{-\infty}^{\infty} \frac{\widehat{\kappa}(s)}{s}ds=-\int_{-z_0}^{z_0} \frac{\widehat{\kappa}(s)}{s}ds
    \end{equation*}
    and (\ref{e kappa(z_0)}) again we can finally conclude that $\arg\left(b_{\pm}(w_0)\right)= \arg(\widehat{b}_{\pm}(z_0))$.
\end{proof}

\bibliographystyle{alpha}
\bibliography{lit}
\end{document}